\pdfoutput=1
\documentclass[11pt]{article}

\usepackage{bbm}
\usepackage{amsmath}
\usepackage{amsfonts}
\usepackage{amssymb}
\usepackage[scr=boondoxupr, scrscaled=1.15] {mathalfa}
\numberwithin{equation}{section}
\usepackage{amsthm}
\newtheorem{theorem}{Theorem}
\newtheorem{corollary}[theorem]{Corollary}
\newtheorem{lemma}[theorem]{Lemma}
\newtheorem{remark}[theorem]{Remark}

\newtheorem{assumption}[theorem]{Assumption}

\usepackage{comment}
\usepackage{setspace}
\usepackage[a4paper,left=3.18cm, right=3.18cm, top=2.54cm, bottom=2.54cm]{geometry}
\usepackage{enumitem}
\setlist[enumerate]{leftmargin=.5in}
\setlist[itemize]{leftmargin=.5in}
\usepackage{xcolor}
\definecolor{trueblue}{rgb}{0.0, 0.45, 0.81}
\definecolor{lightgoldenrodyellow}{rgb}{0.98, 0.98, 0.82}
\definecolor{bananayellow}{rgb}{1.0, 0.88, 0.21}

\usepackage{tikz}
\usetikzlibrary{shapes,arrows}
\usetikzlibrary{calc,patterns,angles,quotes}
\usetikzlibrary{matrix}
\usepackage{tkz-euclide}
\usetikzlibrary{tikzmark}

\usetikzlibrary{3d}
\usetikzlibrary{perspective}
\usetikzlibrary{quotes,angles}
\usepackage{tikz-3dplot}
\usetikzlibrary{shapes,arrows}
\usetikzlibrary{positioning}

\usepackage{graphicx}
\graphicspath{{images/}{./}{Authalic_flow_arXiv/}}
\usepackage{multirow}
\usepackage{subcaption}
\usepackage{epstopdf}
\usepackage{overpic}
\newlength{\Oldarrayrulewidth}
\newcommand{\Cline}[2]{%
  \noalign{\global\setlength{\Oldarrayrulewidth}{\arrayrulewidth}}%
  \noalign{\global\setlength{\arrayrulewidth}{#1}}\cline{#2}%
  \noalign{\global\setlength{\arrayrulewidth}{\Oldarrayrulewidth}}}
\usepackage{algorithm}
\usepackage{algpseudocode}
\usepackage{rotating}
\usepackage{booktabs}
\usepackage{tabu}

\usepackage[colorlinks=true, allcolors=trueblue]{hyperref}

\usepackage{titlesec}
\definecolor{MSBlue}{rgb}{.204,.353,.541}
\titleformat*{\section}{\Large\bfseries\sffamily\color{black}}
\titleformat*{\subsection}{\large\bfseries\sffamily\color{black}}
\renewenvironment{abstract}{%
  \quotation
  \textbf{\textsf{\color{black}{\abstractname.}}}
  \sffamily
}{\endquotation}
\usepackage{authblk}
\newcommand*\samethanks[1][\value{footnote}]{\footnotemark[#1]}
\makeatletter
\def\@fnsymbol#1{\ensuremath{\ifcase#1\or \dagger\or \ddagger\or
   \mathsection\or \mathparagraph\or \|\or **\or \dagger\dagger
   \or \ddagger\ddagger \else\@ctrerr\fi}}
\makeatother

\def\0{\mathbf{0}}
\def\1{\mathbf{1}}

\def\f{\mathbf{f}}

\def\g{\mathbf{g}}

\def\n{\mathbf{n}}

\def\p{\mathbf{p}}

\def\u{\mathbf{u}}
\def\v{\mathbf{v}}

\def\x{\mathbf{x}}
\def\y{\mathbf{y}}

\def\rD{\mathrm{D}}
\def\rS{\mathrm{S}}

\def\bS{\mathbb{S}}
\def\R{\mathbb{R}}

\def\A{\mathcal{A}}

\def\E{\mathcal{E}}
\def\sE{\mathcal{E}}

\def\F{\mathcal{F}}

\def\M{\mathcal{M}}
\def\Mh{{\mathcal{M}_h}}
\def\N{\mathcal{N}}
\def\Nh{{\mathcal{N}_h}}

\def\O{\mathrm{O}}

\def\S{\mathcal{S}}
\def\T{\mathcal{T}}
\def\bT{\mathbb{T}}
\def\V{\mathcal{V}}

\def\I{\mathtt{I}}
\def\B{\mathtt{B}}

\def\d{\mathrm{d}}

\newcommand{\argmin}{\operatornamewithlimits{argmin}}

\newcommand{\diam}[1]{\operatorname{diam}(#1)}
\newcommand{\inrad}[1]{\operatorname{inrad}(#1)}

\title{
\LARGE\bfseries\sffamily\color{black}{
A Discrete-to-Continuum Variational Theory for Area-Preserving Surface Parameterization
} }
\author{
\sffamily\color{black}
{Shu-Yung Liu\thanks{\footnotesize\raggedright Department of Mathematics, National Taiwan Normal University, Taipei, Taiwan (\href{mailto:lii227857@gmail.com}{lii227857@gmail.com}, \href{mailto:yue@ntnu.edu.tw}{yue@ntnu.edu.tw}) },\quad Kento Sakai\thanks{\footnotesize\raggedright Graduate School of Mathematical Sciences, University of Tokyo, Tokyo, Japan (\href{mailto:kento@ms.u-tokyo.ac.jp}{kento@ms.u-tokyo.ac.jp})},
\, and \, Mei-Heng Yueh\samethanks[1] }
}

\date{}

\begin{document}

\captionsetup[figure]{labelfont={bf,sf},name={Figure},labelsep=period}
\captionsetup[table]{labelfont={bf,sf},name={Table},labelsep=period}
\maketitle

\begin{abstract}
We establish a discrete-to-continuum variational theory for area-preserving parameterization of surfaces based on the stretch energy. For orientation-preserving diffeomorphisms between compact, connected, oriented Riemannian surfaces of equal total area, we prove that area-preserving maps are precisely the critical points of the stretch energy under boundary-fixing variations. We then establish discrete-to-continuum convergence of the optimal energy values: within a uniformly geometrically controlled admissible class, the discrete infimum converges to the smooth minimum with second-order accuracy, while discrete almost minimizers have first-order decay of their \(L^2\) area distortion. Thus, the discrete problem not only approximates the smooth stretch energy but also recovers area preservation in the refinement limit. We further derive the $L^2$-gradient flow of the stretch energy and its simplicial counterpart, whose projected quasi-implicit iteration is shown to be globally convergent. The resulting method applies to open and closed surfaces of several topological types, and numerical experiments on benchmark surfaces illustrate its improvement over existing methods.



\end{abstract}

\section{Introduction}
\label{sec:1}

A surface parameterization maps a surface in three-dimensional space to a simpler parameter domain. It plays a central role in geometry processing, computer graphics, and medical imaging~\cite{FlHo05, ShPR06}. In computer graphics, parameterizations enable surface registration by comparing or matching surfaces in a common parameter domain~\cite{LuLY14, LaLu14, ChLL15, YuLW17}. In medical imaging, they have been used to construct spherical harmonic representations for shape analysis~\cite{BrGK95}, with applications to brain ventricles~\cite{GeSJ01}, hippocampal analysis in schizophrenia~\cite{StLP04}, and dementia~\cite{GeCC09, ChLS24}.

A fundamental issue in surface parameterization is the distortion introduced by the map. The principal types of distortion concern lengths, angles, and areas. A locally isometric parameterization would preserve both angles and areas. However, such maps generally do not exist between arbitrary surfaces. In fact, a surface admits a local isometry onto a planar domain only if its Gaussian curvature vanishes. Most of the literature has therefore focused on angle-preserving (conformal) parameterization, which preserves local shape. Representative approaches include linear formulations~\cite{HaAT00}, holomorphic $1$-forms~\cite{GuYa02}, discrete Ricci flow~\cite{ZeGu13, ZhGZ14}, quasiconformal methods~\cite{LuLW13, ChLL15, ChLu18}, and conformal-energy minimization~\cite{MuTA08, YuLW17, KuLY21}. For comprehensive accounts of conformal surface parameterization, we refer to~\cite{GuYa08, JiGH18}.

For applications in which quantitative area information is essential, conformality alone is insufficient, since conformal maps may substantially distort area. An area-preserving (authalic) parameterization preserves the area of every subregion of the surface so that no local expansion or compression occurs. Such parameterizations have been applied to the analysis of anatomical structures~\cite{BrGK95, GeSJ01, StLP04, StOX06, GeCC09}. Existing methods for computing them range from smooth geometric flows to discrete optimization.

Among flow-based methods, a common starting point is Moser's argument~\cite{Mose65, Bany1974}, which deforms one area form into another along the flow of a time-dependent vector field on smooth surfaces. Zou et al.~\cite{ZoHG11} and Dominitz and Tannenbaum~\cite{DoTa10} discretized such flows and constructed area-correcting diffeomorphic evolutions. In particular, Zou et al.~\cite{ZoHG11} formulate the evolution through Lie advection of differential forms, with a velocity field obtained from a Poisson problem, and subsequently discretize it on triangular meshes. In contrast, Lee and Kazhdan~\cite{LeKa19} evolve spherical parameterizations along the negative gradient of the logarithmic scale factor. Alternatively, the density-equalizing map from Choi and Rycroft~\cite{ChRy18} evolves an area ratio by diffusion on the parameter domain. This framework has since been extended to multiply connected planar domains~\cite{LyCh24}, the hemisphere~\cite{GiCK21}, hemispheroids~\cite{ChSh25}, the torus~\cite{YaCh26}, and the sphere~\cite{LyLC24}.

However, implementing such evolutions on triangular meshes requires a discrete representation of the evolving area ratio. Since the Jacobian of a piecewise-linear simplicial map is constant on each triangle, the differential operators appearing in the smooth evolution cannot be applied to it directly. Therefore, existing flow-based implementations introduce vertex-based or smoothed representations of the triangle-wise area ratio before computing the velocity field. While such discretizations can be computationally effective, the resulting discrete evolution acts on a smoothed approximation rather than the actual area ratios of the simplicial map.

A different approach avoids differentiating triangle-wise area ratios. Gu et al.~\cite{GuLS16}, building on the theory of convex polyhedra~\cite{Alex05}, developed a variational formulation of discrete optimal transport and the discrete Monge--Amp\`ere equation on convex polyhedral meshes. This formulation has been applied to area-preserving disk parameterization~\cite{ZhSG13}, multiply connected open surfaces~\cite{SuCQ16}, and spherical parameterization using spherical area measures~\cite{NaSZ17} or spherical power diagrams~\cite{CuQW19}. While these approaches enforce area preservation on a dual polyhedral complex, preserving the areas of its cells does not, in general, imply preserving the areas of the original triangular faces.

Another class of methods works directly with the areas of the primal triangles. Several of them use proxy objectives: Desbrun et al.~\cite{DeMA02} introduced a quadratic chi energy, and Yoshizawa et al.~\cite{YoBS04} considered an $L^2$-stretch measure. Others target triangle-area distortion directly. Ke et al.~\cite{KeXY22} minimize squared triangle-area distortion by coordinate descent for disk parameterization. Yueh et al.~\cite{YuLW19} introduced the stretch energy for simplicial maps. Its relation to area preservation and its gradient were subsequently derived in~\cite{Yueh23}. The same framework has been extended to spherical~\cite{YuLL19} and toroidal~\cite{YuLL20} parameterizations. It has also been incorporated into convergent optimization methods, including Riemannian gradient methods~\cite{SuYu24} and preconditioned nonlinear conjugate-gradient methods~\cite{LiYu24, LiYu25, LiYu26}.

Despite these developments, the smooth and discrete formulations of area-preserving parameterization remain largely disconnected. Flow-based methods build on a smooth theory of area-preserving diffeomorphisms, but their discretizations act on reconstructed area fields rather than on the areas of the triangles themselves. Conversely, methods formulated on the triangles control triangle-wise area distortion directly, but lack a smooth counterpart through which their minimizers can be related to area-preserving diffeomorphisms.

In this paper, we bridge this gap within the variational framework of the stretch energy. We formulate the stretch energy for smooth diffeomorphisms as the continuum counterpart of the existing stretch energy for simplicial maps, and show that its critical points under boundary-fixing variations are precisely the area-preserving maps, which are also its global minimizers. We then show that the discrete problem recovers this variational problem in the refinement limit: the minimal values of the discrete energy converge to the smooth minimum, and discrete almost minimizers have vanishing area distortion. Finally, we derive the gradient flow of the smooth energy and its simplicial counterpart, providing globally convergent algorithms. The main results are summarized below.

\subsection{Main results}
Our first result is a variational characterization of area preservation for diffeomorphisms. Let \(f:(\M,g)\to(\N,h)\) be an orientation-preserving diffeomorphism between compact, connected, oriented Riemannian surfaces of equal total area. Writing \(f^*\d A_{h}=J_f\,\d A_{g}\) for the area ratio \(J_f\), we consider the stretch energy of \(f\) as
\[
\sE_\rS(f)=\int_\M J_f^2\,\d A_{g} .
\]
We prove that its critical points under boundary-fixing variations are precisely the area-preserving maps, which are also its global minimizers; see Theorem~\ref{thm:dEs_area_const}. Area preservation is therefore not merely the optimal state of this energy, but its only stationary one.

We next relate this variational problem to its simplicial counterpart, the minimization of the discrete stretch energy \(E_\rS\). Let \(\Mh\) be a triangular mesh approximating \(\M\) with maximal triangle diameter \(h\), and let \(\mathcal{A}_h\) denote the class of orientation-preserving piecewise linear homeomorphisms \(f_h:\Mh\to\R^3\) whose vertices lie on \(\N\) and whose image meshes approximate \(\N\) under uniform geometric control. We prove that
\[
\left|
\inf_{f_h\in\mathcal{A}_h}E_\rS(f_h)
-
\min_{f\in\mathrm{Diff}^+(\M,\N)}\sE_\rS(f)
\right|
\le Ch^2 ,
\]
and every \(h^2\)-almost minimizer \(f_h^*\in\mathcal{A}_h\), i.e., \(E_\rS(f_h^*)\le\inf_{f_h\in\mathcal{A}_h}E_\rS(f_h)+h^2\), satisfies
\[
\big\|J_{f_h^*}-1\big\|_{L^2(\Mh)}\le Ch;
\]
see Theorem~\ref{thm:minimizer_area}. In other words, the discrete optimal energy converges to the smooth minimum with second-order accuracy, while discrete almost minimizers recover area preservation in the refinement limit. These statements concern optimal values and area distortion. They do not assert convergence of the maps themselves.

Finally, we derive the \(L^2\)-gradient flow of the smooth stretch energy and obtain its simplicial counterpart directly from the exact gradient of the discrete energy, rather than from a reconstructed area field. For closed surfaces, the projected quasi-implicit time discretization with a sufficiently small fixed time step yields a globally convergent method:
\[
\lim_{k \to \infty} \Pi_{T\N} \g_M^{(k)} = \0,
\]
where $\g_M^{(k)}$ is the discrete $L^2$-gradient at the $k$th iterate and $\Pi_{T\N}$ is the vertexwise projection onto the tangent planes of the target surface $\N$; see Theorem~\ref{thm:qi_convergence}. For open surfaces, a similar argument applies under a compactness assumption (Corollary~\ref{cor:open_convergence}). Thus, the resulting algorithm faithfully inherits the gradient-flow structure of the smooth formulation.

\subsection{Organization}
Section~\ref{sec:2} develops the smooth theory: the stretch energy for diffeomorphisms between compact, connected, oriented Riemannian $2$-manifolds of equal total area, the characterization of its critical points, and its $L^2$-gradient flow. Section~\ref{sec:3} introduces the discrete stretch energy for triangulated surfaces and derives its gradient. Section~\ref{sec:4} proves consistency, convergence of the discrete infimum, and the area-distortion estimate for almost minimizers. Section~\ref{sec:5} presents the discrete descent algorithms for surfaces of various topologies, together with benchmark experiments and comparisons with state-of-the-art methods. Section~\ref{sec:6} discusses the the scope of these results and their limitations.

\section{Stretch energy and authalic flow}
\label{sec:2}
In this section, we formulate the stretch energy for diffeomorphisms between compact, connected, oriented Riemannian 2-manifolds of equal total area. We prove that its critical points are area-preserving and derive the associated $L^2$-gradient flow.

\subsection{Stretch energy functional}
\label{sec:2.1}
Let $(\M,g)$ and $(\N,h)$ be two compact, connected, oriented Riemannian $2$-manifolds whose area elements are denoted by $\d A_g$ and $\d A_h$, and let $f:\M\to\N$ be an orientation-preserving diffeomorphism (see Figure~\ref{fig:manifold}). Assume that $\M$ and $\N$ have the same total area, namely
\[
|\M| :=  \int_\M 1 \, \d A_g = \int_\N 1 \, \d A_h =: |\N|.
\]
The \emph{area ratio of $f$} is the smooth function $J_f:\M\to(0,\infty)$ defined by
\begin{equation}\label{eq:area_ratio}
f^*\d A_h = J_f\,\d A_g.
\end{equation}
We say that $f$ is \emph{area-preserving} if $J_f= 1$.

\begin{figure}[tbp]
\centering
\begin{tikzpicture}[>=latex,line width=0.6pt,
  img/.style={inner sep=0pt,outer sep=0pt},
  lbl/.style={font=\small}]
\node[img] (M) at (0,0)   {\includegraphics[width=5.4cm]{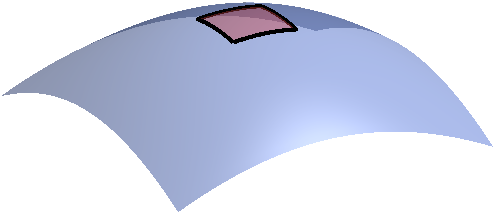}};
\node[img] (N) at (8.6,0) {\includegraphics[width=5.4cm]{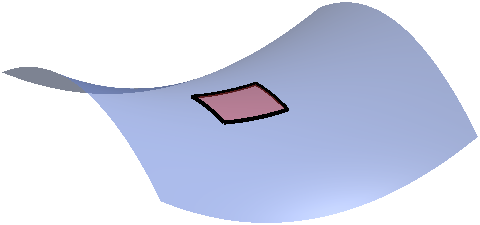}};
\begin{scope}[shift={(M.south west)},x={($(M.south east)-(M.south west)$)},y={($(M.north west)-(M.south west)$)}]
  \node[lbl] at (0.36,0.42) {$(\M,g)$};
  \node[lbl] at (0.50,1.10) {$\d A_g$};
\end{scope}
\begin{scope}[shift={(N.south west)},x={($(N.south east)-(N.south west)$)},y={($(N.north west)-(N.south west)$)}]
  \node[lbl] at (0.83,0.45) {$(\N,h)$};
  \node[lbl] at (0.42,0.86) {$\d A_h$};
\end{scope}
\draw[->] (2.95,0) -- node[above,lbl] {$f$} (5.65,0);
\end{tikzpicture}
\caption{An illustration of a diffeomorphism between Riemannian $2$-manifolds. }
\label{fig:manifold}
\end{figure}

To quantify how far $f$ deviates from area preservation, we consider the variance of $J_f$ with respect to the probability measure $\d P_g = \d A_g / |\M|$:
\[
\mathrm{Var}(J_f) = \int_\M (J_f - \mathbb{E}(J_f))^2 \, \d P_g,
\]
where $\mathbb{E}(\cdot)$ denotes the expected value with respect to $\d P_g$. Under the normalization of equal total area, we have
\[
\mathbb{E}(J_f) = \int_\M J_f \, \d P_g =  \int_\M J_f \, \frac{\d A_g}{|\M|}
= \frac{1}{|\M|} \int_\M f^* \d A_h
= \frac{|\N|}{|\M|} = 1.
\]
Hence, the variance is simplified to
\begin{equation}
\mathrm{Var}(J_f) = \mathbb{E}(J_f^2) - \mathbb{E}(J_f)^2 = \frac{1}{|\M|} \bigg(\int_\M J_f^2 \, \d A_g - |\M| \bigg).
\label{eq:Es_var}
\end{equation}
Motivated by this identity, we define the \emph{stretch energy} as
\begin{equation}
\sE_\rS(f) = \int_\M J_f^2 \, \d A_g \geq |\M|,
\label{eq:Es}
\end{equation}
which directly measures the area distortion of $f$, with equality if and only if $f$ is area-preserving.

\subsection{Variational derivative of stretch energy}
\label{sec:2.2}
Moser~\cite{Mose65} proved that if $\mu_0$ and $\mu_1$ are two volume forms with equal total volume on a compact, connected, oriented manifold without boundary, then there exists an orientation-preserving diffeomorphism $\psi$ such that $\psi^*\mu_1=\mu_0$. Banyaga~\cite{Bany1974} extended this result to compact oriented manifolds with boundary and showed that $\psi$ may be chosen so that its restriction to the boundary is the identity. Bruveris et al.~\cite{Bruv18} later extended the result to manifolds with corners.

Applying this result to $\d A_g$ and $f^*\d A_h$, we obtain a diffeomorphism $\psi:\M\to\M$ such that $\psi^*\d A_g=f^*\d A_h$. Consequently, $f\circ\psi^{-1}$ is an orientation-preserving, area-preserving diffeomorphism from $(\M,g)$ to $(\N,h)$. Thus, within the orientation-preserving diffeomorphism class, $\sE_\rS$ attains its global minimum $|\M|$.

In fact, every critical point of $\sE_\rS(f)$ is area-preserving. Intuitively, a diffeomorphism $f:\M \to \N$ stretches the source surface $\M$ onto the target surface $\N$, and the stretch energy $\sE_\rS$ measures the nonuniformity of this stretching. Thus, its critical points correspond to equilibrium states in which the local area is preserved everywhere.

To prove this claim, we first show that the infinitesimal area change is exact; i.e., under a perturbation of the map, the local area change is determined by the area swept out along the boundary.

\begin{figure}[tbp]
\centering
\begin{tikzpicture}[>=latex,line width=0.6pt,
  img/.style={inner sep=0pt,outer sep=0pt},
  lbl/.style={font=\small}]
\node[img] (M) at (0,0)   {\includegraphics[width=5.4cm]{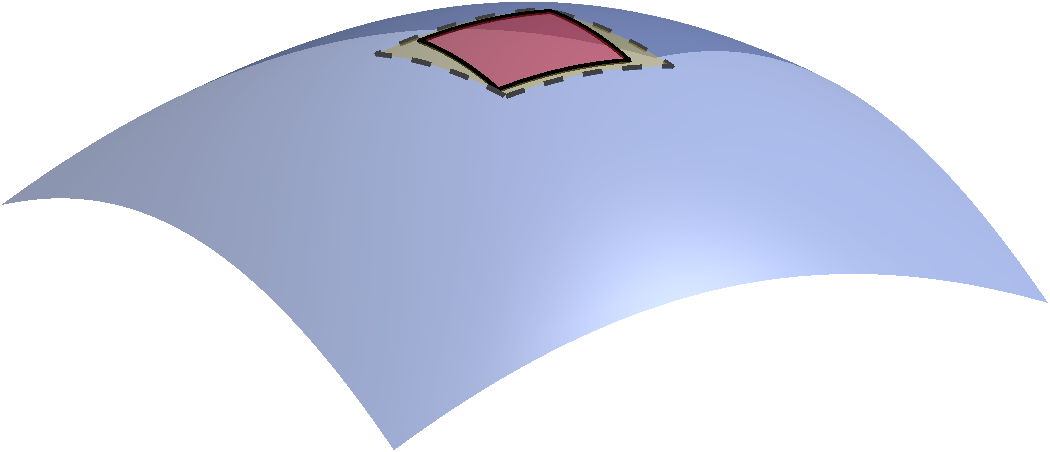}};
\node[img] (N) at (8.6,0) {\includegraphics[width=5.4cm]{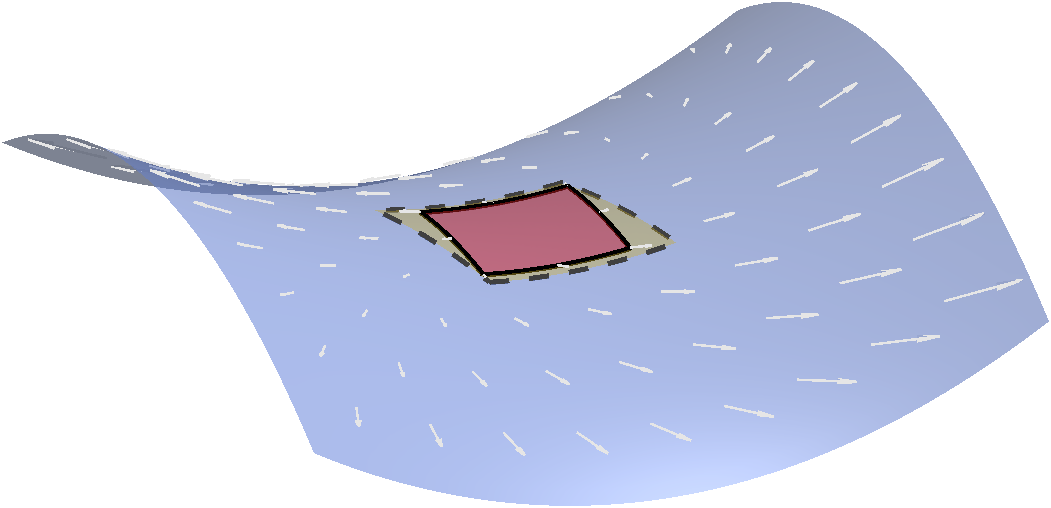}};
\begin{scope}[shift={(M.south west)},x={($(M.south east)-(M.south west)$)},y={($(M.north west)-(M.south west)$)}]
  \node[lbl] at (0.76,0.12) {$\M$};
  \node[lbl] at (0.36,1.08) {$f^*\d A_h$};
  \node[lbl] at (0.68,0.64) {$\delta(f^*\d A_h)$};
  \coordinate (pM) at (0.58,0.99);
\end{scope}
\begin{scope}[shift={(N.south west)},x={($(N.south east)-(N.south west)$)},y={($(N.north west)-(N.south west)$)}]
  \node[lbl] at (0.13,0.16) {$\N$};
  \node[lbl] at (0.29,0.80) {$\d A_h$};
  \node[lbl] at (0.76,0.70) {$\delta(\d A_h)$};
  \node[lbl] at (0.85,0.37) {$V$};
  \coordinate (pN) at (0.49,0.66);
\end{scope}
\draw[->] (2.95,0) -- node[above,lbl] {$f$} (5.65,0);
\draw[->,dashed,gray!45!black] (pN) to[out=125,in=45,looseness=0.55] node[above,lbl,black] {$f^*$} (pM);
\end{tikzpicture}
\caption{The vector field $V$ is defined on $\N$. It induces a variation of $\d A_h$ on $\N$, whose pullback gives the variation $\delta(f^*\d A_h)$ on $\M$.}
\label{fig:patch_change}
\end{figure}

\begin{lemma} \label{lem:delta_area_exact}
Suppose $(\M,g)$ and $(\N,h)$ are oriented Riemannian $2$-manifolds, and $f:\M \to \N$ is a diffeomorphism. Let $V \in \Gamma(T\N)$ be a smooth vector field tangent to
$\partial\N$ when $\partial\N\neq\varnothing$, and let $\phi_t$ be the local flow induced by $V$ such that $\phi_0=\operatorname{id}_{\N}$.
Then the induced first variation of the pullback area form,
$$
\delta(f^*\d A_h) := \left.\frac{\d}{\d t}\right|_{t=0} f^* (\phi_t^*\d A_h),
$$
as illustrated in Figure~\ref{fig:patch_change}, satisfies
\begin{equation} \label{eq:delta_exact}
\delta(f^*\d A_h)=\d\beta, \qquad  \beta:=f^*(\iota_V\d A_h).
\end{equation}
In particular, $\delta(f^*\d A_h)$ is exact.
\end{lemma}

\begin{proof}
First, the variation is written as
\[
\delta( f^* \d A_h)
= f^*\left( \left.\frac{\d}{\d t}\right|_{t=0}\phi_t^* \d A_h \right).
\]
By the definition of the Lie derivative~\cite[Equation 12.8]{Lee13}, we have
\begin{equation}
 \left.\frac{\d}{\d t}\right|_{t=0}\phi_t^* \d A_h
= \mathcal L_V \d A_h.
\label{eq:Lie_1}
\end{equation}
Cartan's formula gives
\begin{equation*}
\mathcal L_V\d A_h = \d(\iota_V\d A_h)+\iota_V\d(\d A_h).
\end{equation*}
Since $\N$ is a surface, every 2-form is closed and $\d(\d A_h)=0$. Hence
\begin{equation}
\mathcal L_V\d A_h = \d(\iota_V\d A_h).
\label{eq:Lie_2}
\end{equation}
Therefore, the commutativity of the pullback and the exterior derivative gives
\[
\delta( f^* \d A_h)
= f^* \d (\iota_V\d A_h)
= \d \bigl(f^*(\iota_V\d A_h)\bigr),
\]
which proves \eqref{eq:delta_exact} with $\beta=f^*(\iota_V\d A_h)$.
\end{proof}

For $W=V\circ f$ with $V$ as in Lemma~\ref{lem:delta_area_exact}, we write $\delta\sE_\rS(f)[W]:=\left.\frac{\d}{\d t}\right|_{t=0}\sE_\rS(\phi_t\circ f)$ for the first variation of $\sE_\rS$ in the direction $W$. With Lemma~\ref{lem:delta_area_exact}, we establish the following result.

\begin{theorem} \label{thm:dEs_area_const}
Let $(\M, g)$ and $(\N, h)$ be compact, connected, oriented Riemannian $2$-manifolds such that $|\M| = |\N|$. Let $f \colon \M \to \N$ be an orientation-preserving diffeomorphism.
Then $J_f\equiv1$ if and only if $\delta\sE_\rS(f)[W]=0$ for every $W=V\circ f$, where $V\in\Gamma(T\N)$ is a smooth vector field satisfying $V|_{\partial\N}=0$.
\end{theorem}

\begin{proof}
By the definition of the stretch energy \eqref{eq:Es}, we have
\begin{equation}
\delta \sE_\rS(f)[W]
= \delta \int_\M J_f^2 \, \d A_g
= 2 \int_\M J_f\, (\delta J_f \, \d A_g).
\label{eq:dEs_1}
\end{equation}
Using \eqref{eq:area_ratio} and Lemma~\ref{lem:delta_area_exact}, we obtain
\begin{equation}
\delta J_f \, \d A_g = \delta (J_f \, \d A_g) = \delta (f^* \d A_h) = \d \beta.
\label{eq:Lie_3}
\end{equation}
Substituting into \eqref{eq:dEs_1} and expanding via the Leibniz rule gives
\[
\delta \sE_\rS(f)[W] = -2 \int_\M \d J_f \wedge \beta + 2\int_\M \d(J_f\beta).
\]
Since $V|_{\partial\N}=0$, we have $\beta|_{\partial\M}=0$. By Stokes' theorem,
\[
\int_\M \d(J_f\beta)=\int_{\partial\M} J_f\beta = 0.
\]
Therefore, we obtain
\begin{equation}
\delta \sE_\rS(f)[W] = -2 \int_\M \d J_f \wedge \beta.
\label{eq:dEs_3}
\end{equation}
If $J_f \equiv 1$, then $\d J_f=0$, and \eqref{eq:dEs_3} immediately gives $\delta \sE_{\rS}(f)[W]=0$ for all $W$.

Conversely, suppose $\delta \sE_{\rS}(f)[W]=0$ for all admissible variations $W$.
Since $\d A_h$ is non-degenerate, the interior product $V \mapsto \iota_V\d A_h$ is an isomorphism. The composite map $W = V \circ f \mapsto \beta = f^*(\iota_V\d A_h)$ is pointwise invertible since the pullback by the diffeomorphism $f$ is also an isomorphism.
Therefore, choosing $V$ to vanish on $\partial\N$ ensures that $\beta$ can be any smooth $1$-form on $\M$ compactly supported in $\operatorname{int}(\M)$.

In particular, let $\star$ denote the Hodge star on $(\M, g)$. For any $\psi \in C_c^{\infty}(\operatorname{int}(\M))$, we choose $W$ so that
\[
\beta = \psi^2 \star \d J_f
\]
(see Figure~\ref{fig:test_variation}). By inserting this into \eqref{eq:dEs_3} and dividing by $-2$, we obtain
\[
0 = \int_\M \psi^2 \, \d J_f \wedge \star \d J_f
= \int_\M \psi^2 \left| \d J_f \right|_g^2 \, \d A_g.
\]
Since the integrand is nonnegative, $\psi\,\d J_f=0$ pointwise. For any $p\in\operatorname{int}(\M)$, choosing $\psi$ with $\psi(p)\neq 0$ forces $\d J_f(p)=0$. Hence, $\d J_f$ vanishes on $\operatorname{int}(\M)$, and by continuity on all of $\M$.

Finally, since $\M$ is connected, $J_f$ is constant, and the assumption $|\M|=|\N|$ gives
\[
J_f\,|\M| = \int_\M J_f \, \d A_g = \int_\M f^*\d A_h = |\N| = |\M|,
\]
which implies $J_f \equiv 1$.
\end{proof}

\begin{figure}[tbp]
\centering
\begin{tikzpicture}[line width=0.5pt,>=latex,
  declare function={
    sg(\x)=0.5*(1+tanh((\x-5.2)/0.9));
    Jf(\x)=0.35+0.9*sg(\x);
    ps(\x)=(abs(\x-5.2)<3)*(1-((\x-5.2)/3)^4)^3;
    dJ(\x)=4*sg(\x)*(1-sg(\x));
  }]
\fill[gray!6] (2.2,-0.1) rectangle (8.2,3.55);
\draw[dotted,gray!55!black] (2.2,-0.1) -- (2.2,3.55);
\draw[dotted,gray!55!black] (8.2,-0.1) -- (8.2,3.55);
\draw[->,line width=0.6pt] (0,2.1) -- (10.4,2.1) node[anchor=west,font=\small] {$\M$};
\draw[line width=1.1pt,domain=0:10,samples=200] plot (\x,{2.1+Jf(\x)});
\node[font=\small,anchor=south] at (9.3,3.37) {$J_f$};
\draw[->,line width=0.6pt] (0,0) -- (10.4,0) node[anchor=west,font=\small] {$\M$};
\fill[gray!45,domain=2.2:8.2,samples=200] plot (\x,{1.1*ps(\x)*dJ(\x)^2}) -- (8.2,0) -- (2.2,0) -- cycle;
\draw[line width=0.7pt,domain=2.2:8.2,samples=200] plot (\x,{1.1*ps(\x)*dJ(\x)^2});
\draw[dashed,line width=0.8pt,gray!40!black,domain=2.2:8.2,samples=200] plot (\x,{1.1*ps(\x)});
\draw[gray!55!black,line width=0.3pt] (7.3,0.52) -- (8.55,1.0);
\node[font=\footnotesize,anchor=west,gray!40!black] at (8.55,1.02) {$\psi^2$};
\draw[gray!55!black,line width=0.3pt] (5.45,0.4) -- (8.55,0.4);
\node[font=\footnotesize,anchor=west] at (8.55,0.4) {$\psi^2|\d J_f|_g^2$};
\node[font=\footnotesize,anchor=north,gray!40!black] at (5.2,-0.08) {$\operatorname{supp}\psi$};
\end{tikzpicture}
\caption{
For $\beta=\psi^2\star\d J_f$, one has $\delta\sE_\rS(f)[W]=-2\int_\M\psi^2|\d J_f|_g^2\,\d A_g$, which is nonpositive and vanishes only if $\d J_f=0$ on the support of $\psi$.}
\label{fig:test_variation}
\end{figure}

The following local-coordinate computation provides a concrete picture of the exactness of area variation.

\begin{remark}
Let $y = (y^1, y^2)$ be oriented local coordinates on an open set $U \subset \N$, and write $H(y) = (h_{ij}(y))$ for the coordinate matrix of $h$.
The area form on $\N$ takes the form
\[
\d A_h = \sigma(y)\,\d y^1\wedge\d y^2,
\quad \text{with} \quad
\sigma(y)=\sqrt{\det H(y)}.
\]
Write $V$ and a test vector field $X$ locally as
$V = V^i\partial_{y^i}$ and $X = X^i\partial_{y^i}$.
A direct computation gives
\begin{align*}
(\iota_V\d A_h)(X) &= \d A_h(V, X)
= \sigma(y) \, (\d y^1\wedge\d y^2)(V, X) \\
&= \sigma(y) \det
\begin{pmatrix}
    V^1  &  X^1\\
    V^2  &  X^2
\end{pmatrix}
= \sigma(y) \left( V^1 X^2 - V^2 X^1 \right),
\end{align*}
which is the oriented area of the parallelogram spanned by $V$ and $X$, scaled by $\sigma(y)$ to convert this to the Riemannian area on $\N$.
We can rewrite the right-hand side as the action of a $1$-form on $X$:
$$
(\iota_V\d A_h)(X) = \sigma(y) \left( V^1 \d y^2(X) - V^2 \d y^1(X) \right).
$$
Since this identity holds for every vector field $X$, we obtain the explicit $1$-form
$$
\iota_V\d A_h = \sigma(y) \left( V^1 \d y^2 - V^2 \d y^1 \right).
$$

This can be used to verify the pointwise invertibility of $V\mapsto\beta$ directly. Fix local coordinates $x = (x^1, x^2)$ on $\widetilde{U} \subset \M$ with $f(\widetilde{U})\subset U$.
The $1$-form $\beta$ is written as
\[
\beta = f^*(\iota_V\d A_h) = \sigma(f) \left( W^1\,\d f^2 - W^2\,\d f^1 \right),
\]
where $W^i = V^i \circ f$. Under local coordinates, we write
\[
\d f^i = f^i_1 \, \d x^1 + f^i_2 \,\d x^2, \qquad \beta = \beta_1 \, \d x^1 + \beta_2 \, \d x^2.
\]
Therefore, we have
\[
\begin{bmatrix}
    \beta_1\\
    \beta_2
\end{bmatrix}
= \sigma(f)
\begin{bmatrix}
    f_1^2 & -f_1^1\\
    f_2^2 & -f_2^1
\end{bmatrix}
\begin{bmatrix}
    W^1\\
    W^2
\end{bmatrix}.
\]
The coefficient matrix has determinant
\[
\sigma(f)^2\,(f_1^1f_2^2-f_2^1f_1^2)=\det H(f)\,\det(\d f).
\]
Since $H(f)$ is positive definite and $f$ is a diffeomorphism, we have $\det H(f)>0$ and $\det(\d f)\neq 0$. Hence, the coefficient matrix is nonsingular, and the mapping $W\mapsto\beta$ is pointwise invertible.
\end{remark}

Theorem~\ref{thm:dEs_area_const} characterizes area-preserving diffeomorphisms precisely as the critical points of $\sE_\rS$. This motivates seeking area preservation by driving $f$ toward a critical point along the negative $L^2$-gradient of $\sE_\rS$, which we term the \emph{authalic flow}.

\subsection{Authalic flow to area preservation}
\label{sec:2.3}

Let $V\in\Gamma(T\N)$ be a smooth vector field satisfying $V|_{\partial\N}=0$ and let $\phi_t$ be the local flow induced by $V$.
Recall that by \eqref{eq:Lie_1} and \eqref{eq:Lie_3}, we have
\[
\delta J_f \d A_g = \delta (f^* \d A_h) = f^* (\mathcal{L}_V \d A_h).
\]
By~\eqref{eq:Lie_2}, the Lie derivative can be written in terms of the divergence~\cite[Chapter 16, page 423]{Lee13}:
\[
\mathcal{L}_V \d A_h = \d(\iota_V \d A_h) = (\operatorname{div}_h V)\,\d A_h.
\]
Hence, we obtain
\begin{align*}
\delta J_f \d A_g  = f^*(\operatorname{div}_h V \, \d A_h)
&= f^*(\operatorname{div}_h V ) ~f^*( \d A_h).
\end{align*}
Substituting into~\eqref{eq:dEs_1}, we have
\begin{align*}
\delta \sE_\rS(f)[W] &= 2 \int_\M J_f \, (\delta J_f \, \d A_g)
= 2 \int_\M J_f \, f^*(\operatorname{div}_h V ) \,f^*( \d A_h).
\end{align*}

To work intrinsically on $\N$, we define the pushforward area ratio
\[
\rho_f := J_f \circ f^{-1}.
\]
Changing variables from $x\in\M$ to $y = f(x)\in\N$ yields
\begin{align*}
\delta \sE_\rS(f)[W]
= 2 \int_\M J_f \, f^*(\operatorname{div}_h V ) \,f^*( \d A_h)
= 2\int_\N \rho_f\,\operatorname{div}_h V\,\d A_h.
\end{align*}
Applying the product rule for divergence~\cite[Exercise 16.12]{Lee13} gives
\[
\operatorname{div}_h(\rho_f V) = \langle\nabla_h\rho_f,\,V\rangle_h + \rho_f\,\operatorname{div}_h V,
\]
and integrating over $\N$ yields
\[
\delta \sE_\rS(f)[W]
= 2 \int_\N \operatorname{div}_h (\rho_f V) \, \d A_h
- 2 \int_\N \langle \nabla_h \rho_f, \, V \rangle_h \, \d A_h.
\]
Since $V|_{\partial\N} = 0$, the divergence theorem~\cite[Theorem 16.32]{Lee13} implies
\[
\int_\N \operatorname{div}_h(\rho_f V)\,\d A_h
= \int_{\partial\N} \rho_f\langle V,\,n\rangle_h\,\d s_h = 0,
\]
where $n$ is the outward unit normal along $\partial\N$. Therefore, we obtain
\begin{equation}  \label{eq:grad_var_E}
\delta \sE_\rS(f)[W] =  -2 \int_\N \langle \nabla_h \rho_f, \, V \rangle_h \, \d A_h.
\end{equation}

The $L^2$-gradient of $\sE_\rS$ is defined through the $L^2$-inner product by
\[
\delta \sE_\rS(f)[W] = \langle \nabla_{L^2} \sE_\rS(f), \, W \rangle_{L^2(\M)}
= \int_\M \langle \nabla_{L^2} \sE_\rS(f), \, W \rangle_{h}\, \d A_g.
\]
By changing variables to $\N$ and comparing with~\eqref{eq:grad_var_E}, we have
\[
-2 \int_\N \langle \nabla_h \rho_f, \, V \rangle_h \, \d A_h = \int_\N \langle \nabla_{L^2} \sE_\rS(f) \circ f^{-1}, \, V \rangle_{h}\, \frac{\d A_h}{\rho_f}.
\]
This gives the pushforward $L^2$-gradient:
\[
\nabla_{L^2} \sE_\rS(f)\circ f^{-1}
= -2\,\rho_f \,\nabla_h \rho_f = -\nabla_h(\rho_f^2),
\]
which is exactly the negative gradient of $\rho_f^2$.
Hence, by pulling back to $\M$, we have
\[
 \nabla_{L^2} \sE_\rS(f) = -2\,J_f (\nabla_h \rho_f) \circ f, \quad \rho_f = J_f \circ f^{-1}.
\]

Motivated by the negative $L^2$-gradient direction, we formally define the \emph{authalic flow} by
\begin{equation} \label{eq:L2_grad_flow}
\begin{cases}
\partial_t f_t
=2\,J_{f_t}(\nabla_h\rho_t)\circ f_t
&\text{in }\operatorname{int}(\M),\\
f_t=f_0
& \text{on $\partial\M$},
\end{cases}
\end{equation}
where $\rho_t := J_{f_t} \circ f_t^{-1}$. Conditional on the existence of a smooth, orientation-preserving diffeomorphic solution on an interval $[0,T)$,\footnote{If $\partial\M\neq\varnothing$, smoothness up to the boundary implicitly requires $\nabla_h\rho_t=0$ on $\partial\N$.} the stretch energy decreases monotonically:
\begin{align*}
\frac{\d}{\d t}\sE_\rS(f_t)
&= - \int_\M \| \nabla_{L^2}  \sE_\rS(f_t)\|_h^2 \,\d A_g
= - \int_\N \| \nabla_{L^2}  \sE_\rS(f_t) \circ f_t^{-1} \|_h^2 \, \frac{\d A_h}{\rho_t} \\
&= -4 \int_\N \rho_t \| \nabla_h \rho_t \|^2_h  \, \d A_h \leq 0,
\end{align*}
and the equality holds if and only if $\nabla_h\rho_t = 0$, i.e., $\rho_t$ is constant, which agrees with Theorem~\ref{thm:dEs_area_const} under the normalization $|\M| = |\N|$.

When $\N$ is taken to be a canonical target surface, the flow equation~\eqref{eq:L2_grad_flow} provides a descent method toward an area-preserving parameterization.
Discretizing the flow in time then gives a natural iterative method for computing such a parameterization.
However, in practice, surfaces in computer graphics are commonly represented as triangular meshes. The diffeomorphisms are therefore approximated by piecewise affine maps, whose Jacobian is piecewise-constant, and the actual flow cannot be directly implemented in the discrete setting.

Existing flow-based methods thus use a smoothed representation of the area ratio, which may deviate from the true area ratio. To preserve the variational structure from smooth surfaces to triangular meshes, we introduce a discrete analog of the stretch energy as follows.

\section{Simplicial formulation of discrete stretch energy}
\label{sec:3}
In this section, we derive the discrete stretch energy for simplicial maps from the smooth functional \eqref{eq:Es}. We begin by establishing notation for simplicial surfaces and then derive the discrete energy and its gradient.

\subsection{Simplicial surfaces and simplicial mappings}
\label{sec:3.1}
A smooth surface $\M \subset \R^3$ can be approximated by a \emph{simplicial (triangular) surface} $\Mh$, characterized by a set of vertices $\V$, oriented triangular faces $\F$, and edges $\E$:
\begin{align*}
\V &= \left\{ v_\ell = (v_\ell^1, v_\ell^2, v_\ell^3) \in\R^3 \right\}_{\ell=1}^n, \\
\F &= \left\{ [v_i,v_j,v_k] \mid [v_i,v_j,v_k]:
\text{oriented face of }\Mh \right\}, \\
\E &= \left\{ [v_i,v_j] \subset\R^3 \mid [v_i,v_j,v_k]\in\F  \right\}.
\end{align*}
The source-mesh size is defined as the maximal diameter of the triangles
\[
h := \max_{\tau\in\F} \diam{\tau}.
\]

A \emph{simplicial map} $f$ sends the source mesh $\Mh$ to the image mesh $\Nh:=f(\Mh) \subset \R^3$, where the restriction $f|_\tau$ to any triangle $\tau$ is an affine map. In particular, $f(\tau)$ remains a triangle whenever $f$ is non-degenerate on every face. The image-mesh size is defined by
\[
k_h := \max_{\tau\in\F} \diam{f(\tau)}.
\]

A simplicial map is completely determined by its values at the vertices:
\[
f_i := f(v_i) = (f_i^1, f_i^2, f_i^3)^\top, \qquad v_i \in \V.
\]
More precisely, the restriction of $f$ to $\tau = [v_i, v_j, v_k]$ admits the barycentric representation:
\[
f|_{\tau}(p)
= \lambda_i^\tau (p)\,f_i + \lambda_j^\tau (p)\,f_j + \lambda_k^\tau (p)\,f_k,
\]
where the barycentric coordinates are given by area ratios:
\begin{equation}
\lambda_i^\tau (p)=\frac{|[p,v_j,v_k]|}{|\tau|},\qquad
\lambda_j^\tau (p)=\frac{|[v_i,p,v_k]|}{|\tau|},\qquad
\lambda_k^\tau (p)=\frac{|[v_i,v_j,p]|}{|\tau|},
\label{eq:bary_coord}
\end{equation}
with $|\tau|$ denoting the area of $\tau$. Equivalently, $f$ can be expressed globally using the piecewise linear basis functions $\{\varphi_i\}_{i=1}^n$ on $\Mh$:
\begin{equation}
f(p) = \sum_{i=1}^n f_i \varphi_i(p),
\quad \text{where} \quad
\varphi_i(p) =
\begin{cases}
\lambda_i^\tau(p), & p \in \tau, \\
0, & \text{otherwise}.
\end{cases}
\label{eq:basis_fun}
\end{equation}
Accordingly, the simplicial map can be represented by the vertex coordinate vectors
\begin{equation}
\mathbf{f}^1 = \begin{bmatrix} f_1^1 \\ \vdots \\ f_n^1 \end{bmatrix}, \quad
\mathbf{f}^2 = \begin{bmatrix} f_1^2 \\ \vdots \\ f_n^2 \end{bmatrix}, \quad
\mathbf{f}^3 = \begin{bmatrix} f_1^3 \\ \vdots \\ f_n^3 \end{bmatrix}, \quad
\f = \begin{bmatrix}
    \f^1 & \f^2 & \f^3
\end{bmatrix}.
\label{eq:f_matrix}
\end{equation}

\subsection{Discrete stretch energy for simplicial maps}
\label{sec:3.2}
Let $\Mh \subset \R^3$ be a triangular surface and $f:\Mh \to \Nh$ be an orientation-preserving simplicial map.
The total areas of the source and image meshes are denoted by
\begin{equation*}
|\Mh| = \sum_{\tau \in \F} |\tau|, \qquad |\Nh|= \sum_{\tau \in \F} |f(\tau)|.
\end{equation*}
Since $f$ is affine on each triangle $\tau \in \F$, the area ratio $J_{f|_\tau}$ defined in \eqref{eq:area_ratio} is constant on each face and satisfies
\begin{equation} \label{eq:f_tau}
|f(\tau)| = \int_\tau f^* \d A = \int_\tau J_{f|_\tau} \d A =  J_{f|_\tau} \, |\tau|.
\end{equation}
In this discrete setting, we say $f$ is \emph{area-preserving} if $J_{f|_\tau} = 1$ for all $\tau \in \F$.

The stretch energy extends naturally to simplicial maps by integrating over $\Mh$:
\begin{equation}
E_\rS(f)
:= \int_\Mh J_f^2 \, \d A
= \sum_{\tau \in \F}\int_\tau  J_{f|_\tau}^2\, \d A
= \sum_{\tau \in \F} \int_\tau \frac{|f(\tau)|^2}{|\tau|^2} \, \d A
= \sum_{\tau \in \F} \frac{|f(\tau)|^2}{|\tau|},
\label{eq:Es_Geo}
\end{equation}
which we call the \textit{discrete stretch energy}, in line with \cite[Lemma 3.1]{Yueh23}.
As in the smooth case \eqref{eq:Es_var}, the discrete stretch energy can also be interpreted in terms of the variance of the area ratios.

\begin{corollary}  \label{cor:Es_var}
Let $f$ be an orientation-preserving simplicial map with $|\Mh| = |\Nh|$. Then the discrete stretch energy \eqref{eq:Es_Geo} satisfies
\begin{equation} \label{eq:Es_var_discrete}
\mathrm{Var} \bigg( \frac{|f(\tau)|}{|\tau|} \bigg) = \frac{1}{|\Mh|} \Big( E_\rS(f) - |\Mh| \Big),
\end{equation}
where the variance is taken with respect to the discrete probability measure that assigns mass $|\tau|/|\Mh|$ to each triangle $\tau$.
In particular, $E_\rS(f) \geq |\Mh|$, with equality if and only if $f$ is area-preserving.
\end{corollary}

\begin{proof}
Since $|\Mh| = |\Nh|$, the expected value with respect to the measure $|\tau|/|\Mh|$ is
$$
\mathbb{E}(J_{f|_\tau})
= \sum_{\tau \in \F} \frac{|\tau|}{|\Mh|}\frac{|f(\tau)|}{|\tau|}
= \frac{|\Nh|}{|\Mh|}
= 1.
$$
Hence, the variance satisfies
\begin{align*}
\mathrm{Var} ( J_{f|_\tau})
&= \mathbb{E}( J_{f|_\tau}^2) - \mathbb{E}( J_{f|_\tau})^2
= \mathbb{E}( J_{f|_\tau}^2) - 1 \\
&= \sum_{\tau \in \F} \frac{|\tau|}{|\Mh|} \frac{|f(\tau)|^2}{|\tau|^2} - 1
= \frac{1}{|\Mh|} \Big( E_\rS(f) - |\Mh| \Big).
\end{align*}
Since the variance is nonnegative, $E_\rS(f) \geq |\Mh|$ and equality holds if and only if $J_{f|_\tau}=1$ for every triangle $\tau$.
\end{proof}

\subsection{Gradients of the discrete stretch energy}
\label{sec:3.3}
From the work of Pinkall and Polthier \cite{PiPo93}, the Dirichlet energy for a simplicial map can be discretized using the cotangent formula:
\begin{equation}
E_\rD(f)
= \frac{1}{2} \int_\Mh \| \nabla f\|^2 \,\d A
= \frac{1}{2} \sum_{s = 1}^3 {\f^s}^{\top} L_\rD \, \f^s,
\label{eq:Ed_quad}
\end{equation}
where $L_\rD$ is the cotangent-weighted Laplacian matrix on $\Mh$, defined by
\begin{equation} \label{eq:Ld}
[L_\rD]_{i, j} =
\begin{cases}
    -\frac{1}{2} \sum_{ \tau\supset[v_i,v_j]}\cot\theta_{ij}^{\tau} &\text{if }i\neq j\text{ and }[v_i,v_j]\in\E,\\
    -\sum_{\ell \neq i} [L_\rD]_{i, \ell}     &\text{if}~j=i,\\
    0       &\text{otherwise},
\end{cases}
\end{equation}
and $\theta_{ij}^{\tau}$ denotes the angle of $\tau$ opposite the
edge $[v_i,v_j]$.

This discretization directly gives the gradient with respect to $\f$, and the associated harmonic map, which is the critical point of $E_\rD(f)$, can be obtained by solving a sparse linear system.
Similarly, the stretch energy $E_{\rS}(f)$ can be expressed using modified cotangent weights, as shown in \cite[Lemma~3.1, Theorem~3.5]{Yueh23}. We give a simpler and more transparent proof below. We first establish the following lemma.

\begin{lemma} \label{lma:ImgArea}
Let $f:\Mh \to \Nh$ be an orientation-preserving simplicial map, and let $L_\rD(f)$ denote the cotangent-weighted Laplacian on the image mesh $\Nh$, defined by \eqref{eq:Ld}. Then,
\begin{equation}
|\Nh| = \frac{1}{2} \sum_{s=1}^3 {\f^s}^{\top} L_\rD(f) \, \f^s,
\label{eq:ImgArea}
\end{equation}
and
\begin{equation} \label{eq:ImgArea_grad}
\nabla_{\f^s} |\Nh| = L_\rD(f) \, \f^s, \qquad s = 1,2,3.
\end{equation}
\end{lemma}

\begin{proof}
The image area can be written as
\begin{align*}
|\Nh| = \frac{1}{2}\int_{\Nh} \| I_2\|^2 \, \d A,
\end{align*}
which is precisely the Dirichlet energy of the identity map on $\Nh$. Therefore, by \eqref{eq:Ed_quad}, we have
\[
|\Nh| = E_\rD(\mathrm{id}_{\Nh}) = \frac{1}{2} \sum_{s=1}^3 {\f^s}^{\top} L_\rD(f) \, \f^s,
\]
which proves \eqref{eq:ImgArea}.

Next, we rewrite the gradient as
\[
\begin{bmatrix}
    \nabla_{\f^1} |\Nh| & \nabla_{\f^2} |\Nh|   & \nabla_{\f^3} |\Nh|
\end{bmatrix} =
\begin{bmatrix}
    \nabla_{f_1} |\Nh|^\top \\
    \vdots \\
    \nabla_{f_n} |\Nh|^\top
\end{bmatrix} \in \mathbb{R}^{n \times 3},
\]
and consider the gradient with respect to $f_i$. The orientation preservation of $f$ gives $|\Nh| = \sum |f(\tau)|$ and
\[
\nabla_{f_i} |\Nh| = \sum_{\tau \ni v_i} \,\nabla_{f_i} |f(\tau)|.
\]
Thus, to prove \eqref{eq:ImgArea_grad}, it suffices to show that
\begin{align*}
\nabla_{f_i} |f(\tau)|
&= [L_\rD(f|_\tau)]_{i,i} \,f_i + [L_\rD(f|_\tau)]_{i,j} \,f_j + [L_\rD(f|_\tau)]_{i,k} \,f_k\\
&= \frac{1}{2} \cot \theta_{i,j}^k (f) (f_i - f_j) + \frac{1}{2} \cot \theta_{k,i}^j(f) (f_i - f_k),
\end{align*}
where $\theta_{i,j}^k(f)$ denotes the angle at vertex $f_k$ opposite the edge $[f_i,f_j]$.

We follow the geometric argument in \cite{Cran19}; see Figure~\ref{fig:cot_gradient}. We let $f_p$ be the foot of the perpendicular from $f_i$ to the line containing the edge $[f_j, f_k]$.
Since $|f(\tau)|=\tfrac12|[f_j,f_k]|\,\|f_i-f_p\|$, the gradient with respect to $f_i$ is parallel to $f_i-f_p$.
Moreover, since the area depends affinely on $f_i$ along this direction, we have
\[
\nabla_{f_i} |f(\tau)|
= \frac{|[f_i, f_j, f_k]|}{ |[f_p, f_i]|} \, \frac{f_i - f_p}{|[f_p, f_i]|}
= \frac{1}{2} \frac{|[f_j, f_k]|}{|[f_p, f_i]|} \, (f_i - f_p).
\]
Let $s(a,b)$ denote the signed distance along the oriented line through $[f_j,f_k]$, defined by
\[
s(a, b) = (b-a) \cdot \frac{f_k - f_j}{|[f_j, f_k]|}.
\]
Since $s(f_j, f_p) + s(f_p, f_k) = |[f_j, f_k]|$, the point $f_p$ can be represented as
\begin{align*}
f_p &= f_j + s(f_j, f_p) \, \frac{f_k - f_j}{|[f_j, f_k]|}
= \frac{s(f_p, f_k)}{|[f_j, f_k]|} f_j + \frac{s(f_j, f_p)}{|[f_j, f_k]|} f_k,
\end{align*}
and
\[
f_i - f_p = \frac{s(f_p, f_k)}{|[f_j, f_k]|} (f_i -  f_j) + \frac{s(f_j, f_p)}{|[f_j, f_k]|} (f_i - f_k).
\]
Hence, the gradient can be rewritten as
\begin{align*}
\nabla_{f_i} |f(\tau)| &= \frac{1}{2} \frac{|[f_j, f_k]|}{|[f_p, f_i]|} \, (f_i - f_p)
= \frac{s(f_p, f_k)}{2\,|[f_p, f_i]|} (f_i -  f_j) + \frac{s(f_j, f_p)}{2\,|[f_p, f_i]|} (f_i - f_k)\\
&= \frac{1}{2} \cot \theta_{i,j}^k(f) \, (f_i - f_j)
+  \frac{1}{2}  \cot \theta_{k,i}^j(f) \, (f_i - f_k),
\end{align*}
which proves the claim.
\end{proof}

\begin{figure}[tbp]
\centering
\begin{tikzpicture}[line width=0.5pt,>=latex,scale=1.05]
\coordinate (fj) at (0,0);
\coordinate (fk) at (5.2,0);
\coordinate (fi) at (1.7,2.5);
\coordinate (fp) at (1.7,0);
\fill[gray!10] (fj) -- (fk) -- (fi) -- cycle;
\draw[line width=0.7pt] (fj) -- (fk) -- (fi) -- cycle;
\draw[dashed,gray!55!black] (fi) -- (fp);
\draw[gray!55!black] (1.7,0.18) -- (1.88,0.18) -- (1.88,0);
\pic[draw,angle radius=0.62cm,"$\theta^j_{k,i}$",angle eccentricity=1.75,font=\footnotesize] {angle = fk--fj--fi};
\pic[draw,angle radius=0.72cm,"$\theta^k_{i,j}$",angle eccentricity=1.7,font=\footnotesize] {angle = fi--fk--fj};
\draw[->,line width=1.1pt] (fi) -- (1.7,3.75);
\node[anchor=west,font=\footnotesize] at (1.78,3.55) {$\nabla_{f_i}|f(\tau)|$, \ length $\tfrac12|[f_j,f_k]|$};
\draw[<->,dotted,line width=0.7pt,gray!40!black] (0.75,2.5) -- (2.65,2.5);
\node[anchor=east,font=\footnotesize,gray!40!black] at (0.7,2.5) {no area change};
\fill (fj) circle (1.3pt) node[anchor=north east,font=\small] {$f_j$};
\fill (fk) circle (1.3pt) node[anchor=north west,font=\small] {$f_k$};
\fill (fi) circle (1.3pt) node[anchor=south east,font=\small] {$f_i$};
\fill (fp) circle (1.3pt);
\node[anchor=north,font=\small] at (1.7,-0.02) {$f_p$};
\draw[<->,gray!40!black] (0,-0.55) -- (1.7,-0.55) node[midway,below,font=\footnotesize] {$s(f_j,f_p)$};
\draw[<->,gray!40!black] (1.7,-0.55) -- (5.2,-0.55) node[midway,below,font=\footnotesize] {$s(f_p,f_k)$};
\node[anchor=west,font=\footnotesize,gray!40!black] at (1.75,1.1) {$|[f_p,f_i]|$};
\end{tikzpicture}
\caption{The gradient points along $f_i-f_p$, and the signed lengths $s(f_j,f_p)$ and $s(f_p,f_k)$ yield the cotangent weights.}
\label{fig:cot_gradient}
\end{figure}

Lemma~\ref{lma:ImgArea} gives the following gradient formula for the discrete stretch energy.

\begin{theorem} \label{thm:Es_quad}
Let $f:\Mh \to \R^3$ be an orientation-preserving simplicial map. Then, the discrete stretch energy \eqref{eq:Es_Geo} admits the representation
\begin{equation}
E_\rS(f) = \frac{1}{2} \sum_{s=1}^3 {\f^s}^{\top} L_\rS(f) \,\f^s,
\label{eq:Es_quad}
\end{equation}
where $L_\rS(f)$ is the stretch Laplacian, assembled facewise as
\begin{equation}
L_\rS(f) = \sum_{\tau \in \F} L_\rS(f|_\tau),
\qquad
L_\rS(f|_\tau) := \frac{|f(\tau)|}{|\tau|} L_\rD(f|_\tau).
\label{eq:Ls}
\end{equation}
Moreover, the gradient of $E_\rS(f)$ with respect to $\f^s$ is
\begin{equation}  \label{eq:Es_grad}
\nabla_{\f^s} E_\rS(f) = 2 \, L_\rS(f) \, \f^s, \quad s = 1,2,3.
\end{equation}
\end{theorem}

\begin{proof}
For each face $\tau \in \F$, Lemma~\ref{lma:ImgArea} applied to $f|_\tau$ gives
\[
|f(\tau)| = \frac{1}{2} \sum_{s=1}^3 {\f^s}^{\top} L_\rD(f|_\tau) \, \f^s.
\]
Therefore, we have
\begin{align*}
E_\rS(f)
&= \sum_{\tau \in \F} \frac{|f(\tau)|^2}{|\tau|}
= \frac{1}{2} \sum_{\tau \in \F} \frac{|f(\tau)|}{|\tau|}  \sum_{s=1}^3 {\f^s}^{\top} L_\rD(f|_\tau) \f^s \\
&= \frac{1}{2} \sum_{s=1}^3 {\f^s}^{\top} L_\rS(f) \, \f^s,
\end{align*}
which proves \eqref{eq:Es_quad}.

For the gradient, applying \eqref{eq:ImgArea_grad} to $f|_\tau$ gives
\[
\nabla_{\f^s} |f(\tau)| = L_\rD(f|_\tau) \, \f^s, \qquad s = 1,2,3.
\]
Hence, by the chain rule,
\begin{align*}
\nabla_{\f^s} E_\rS(f)
&= \sum_{\tau \in \F} \,\nabla_{\f^s} \left(\frac{|f(\tau)|^2}{|\tau|}\right)
= 2\sum_{\tau \in \F} \frac{|f(\tau)|}{|\tau|} \, \nabla_{\f^s} |f(\tau)|\\
&= 2\sum_{\tau \in \F} \frac{|f(\tau)|}{|\tau|} \, L_\rD(f|_\tau) \, \f^s
= 2\,L_\rS(f)\,\f^s, \nonumber
\end{align*}
which proves \eqref{eq:Es_grad}.
\end{proof}

The above formulation of $E_\rS$ plays an important role in computation, which allows us to discretize the authalic flow \eqref{eq:L2_grad_flow} into a globally convergent algorithm for closed surfaces and, under a compactness assumption, for open surfaces, as described in Section~\ref{sec:5}.

While $E_\rS$ naturally corresponds to $\E_\rS$ as in \eqref{eq:Es_Geo}, it is still not clear whether $E_\rS$ approaches $\E_\rS$ under mesh refinement. More importantly, since $\E_\rS$ is used to characterize area preservation through its variational principle, it is crucial to verify that $E_\rS$ has a consistent variational structure. We address this in Section~\ref{sec:4}.

\section{Consistency and discrete minimization estimates}
\label{sec:4}

We now justify the discretization of the stretch energy in Section~\ref{sec:3} under mesh refinement.
Throughout this section, $\M,\N\subset\R^3$ carry the metrics induced by the Euclidean metric, and their simplicial approximations carry the corresponding orientations.

We first establish the consistency of the discrete stretch energy for piecewise linear interpolants of smooth diffeomorphisms. Then, we use these approximation estimates to compare the minimal discrete and continuous energy levels, showing that the discrete infimum converges to the continuous minimum at a rate $\O(h^2)$. As a consequence, discrete maps whose energies are within $\O(h^2)$ of the discrete infimum have $\O(h)$ area distortion in $L^2$.

\subsection{Assumptions of triangulation}
\label{sec:4.1}
By the Tubular Neighborhood Theorem~\cite[Theorem 6.24]{Lee13}, a tubular neighborhood of $\mathcal{M}$ is well-defined, and we assume $\Mh$ lies within it when $\partial\M=\varnothing$. We define the closest-point projection $\pi_\M: \mathcal{M}_h \to \mathcal{M}$ by
\[
\pi_\M(x) := \argmin_{p\in \M} \| p - x\|^2.
\]
Because minimizing the distance forces the vector $x - \pi_\M(x)$ to be orthogonal to the tangent space of $\mathcal{M}$ at $\pi_\M(x)$, we have the representation
\begin{equation} \label{eq:patch_projection}
    x = \pi_\M(x) + s(x) \,n \big(\pi_\M(x) \big),
\end{equation}
where $s(x)$ is the signed distance, and $n$ is the unit normal.
When $\partial\M\neq\varnothing$, we instead use the closest-point projection of a smooth extension $\M^e\supset\M$.

Following the framework of~\cite[Section 4.1]{DzEl13}, we impose the following assumption to ensure that $\mathcal{M}_h$ is a shape-regular approximation of $\mathcal{M}$.

\begin{assumption} \label{ass:triangulation}
Let $\{\Mh\}_{h>0}$ be a family of simplicial approximations of a compact, connected, oriented smooth surface $\M \subset \R^3$ (possibly with boundary), with all vertices on $\M$. Assume that:
\begin{itemize}
    \item \textbf{Closest-point projection:}
    The closest-point projection of a smooth extension $\M^e\supset\M$ is well-defined and bijective from $\Mh$ onto the lift $\M_h^\ell:=\pi_{\M^e}(\Mh)$. Note that $\M_h^\ell = \M$ if $\partial \M = \varnothing$.

    \item \textbf{Boundary extensions:}
    If $\partial \M \neq \varnothing$, boundary vertices lie on $\partial\M$, i.e., $\V\cap\partial\Mh\subset\partial\M$, and there exists a constant $C>0$ such that the symmetric difference
    \[
    \M_h^\ell\mathbin{\triangle}\M
    \subset
    \bigl\{x\in\M^e \, \mid \,
    \mathrm{dist}_{\M^e}(x, \partial \M)\le Ch^2\bigr\}.
    \]
    Moreover, the closest-point projection $\pi_{\partial\M}:\partial\M_h^\ell\to\partial\M$ is bijective.

    \item \textbf{Shape-regularity:}
    The family $\{\Mh\}_{h>0}$ is uniformly shape-regular; i.e., there exists a constant $C>0$, independent of $h$, such that
    \begin{equation}
        \max_{\tau \in \F} \frac{\diam{\tau}}{\inrad{\tau}} \leq C,
        \label{eq:shape_regular}
    \end{equation}
    where $\diam{\tau}$ and $\inrad{\tau}$ denote the diameter and in-ball radius of $\tau \in \F$.
\end{itemize}
\end{assumption}

\subsection{Consistency of the stretch energy}
\label{sec:4.2}

Under Assumption~\ref{ass:triangulation}, we use the following geometric estimate.

\begin{lemma} \label{lem:area_change}
Suppose $\M$ and $\Mh$ satisfy Assumption~\ref{ass:triangulation}. The projection operator $\pi_{\M^e}$ satisfies
\begin{equation}\label{eq:proj_area_estimate}
\big\| 1 - J_{\pi_{\M^e}} \big\|_{L^\infty(\Mh)} \leq  C h^2
\end{equation}
for some constant $C>0$. Moreover,
\begin{equation}\label{eq:lifted_domain_estimate}
|\M_h^\ell\mathbin{\triangle}\M|\leq Ch^2,
\qquad
\big||\Mh|-|\M|\big|\leq Ch^2.
\end{equation}
\end{lemma}

\begin{proof}
We first prove the estimate~\eqref{eq:proj_area_estimate}.
When $\partial\M=\varnothing$, the estimate is precisely the estimate in~\cite[Lemma 4.1]{DzEl13}. When $\partial\M\neq\varnothing$, the same local argument applies to the smooth extension $\M^e$.

For the left inequality of \eqref{eq:lifted_domain_estimate}, if $\partial\M=\varnothing$, the inequality holds due to $\M_h^\ell=\M$.
If $\partial\M\neq\varnothing$, since the area of a strip of width $h^2$ around the compact smooth curve $\partial\M$ is $\O(h^2)$, by the boundary-extension condition in Assumption~\ref{ass:triangulation},
\[
|\M_h^\ell\mathbin{\triangle}\M|
\leq\left|
\bigl\{x\in\M^e \, \mid \, \mathrm{dist}_{\M^e}(x, \partial \M)\leq Ch^2\bigr\} \right|
\leq Ch^2.
\]

For the right inequality of \eqref{eq:lifted_domain_estimate}, since $\M$ is compact, $|\Mh|$ is bounded, and the change-of-variables formula and~\eqref{eq:proj_area_estimate} give
\[
\big||\Mh|-|\M_h^\ell|\big|
=\left|\int_\Mh (1-J_{\pi_{\M^e}})\,\d A\right|
\leq
\|1-J_{\pi_{\M^e}}\|_{L^\infty(\Mh)}\,|\Mh|
\leq Ch^2.
\]
Therefore, by the fact that
\[
\bigl||\M_h^\ell|-|\M|\bigr|
\leq |\M_h^\ell\mathbin{\triangle}\M|
\leq Ch^2,
\]
we obtain
\[
\bigl||\Mh|-|\M|\bigr|
\leq
\bigl||\Mh|-|\M_h^\ell|\bigr|
+
\bigl||\M_h^\ell|-|\M|\bigr|
\leq Ch^2.
\]
\end{proof}

\begin{figure}[tbp]
\centering
\begin{tikzpicture}[>=latex,line width=0.6pt,
  img/.style={inner sep=0pt,outer sep=0pt},
  lbl/.style={font=\small}]
\node[img] (M)  at (0,3.45)   {\includegraphics[width=5.0cm]{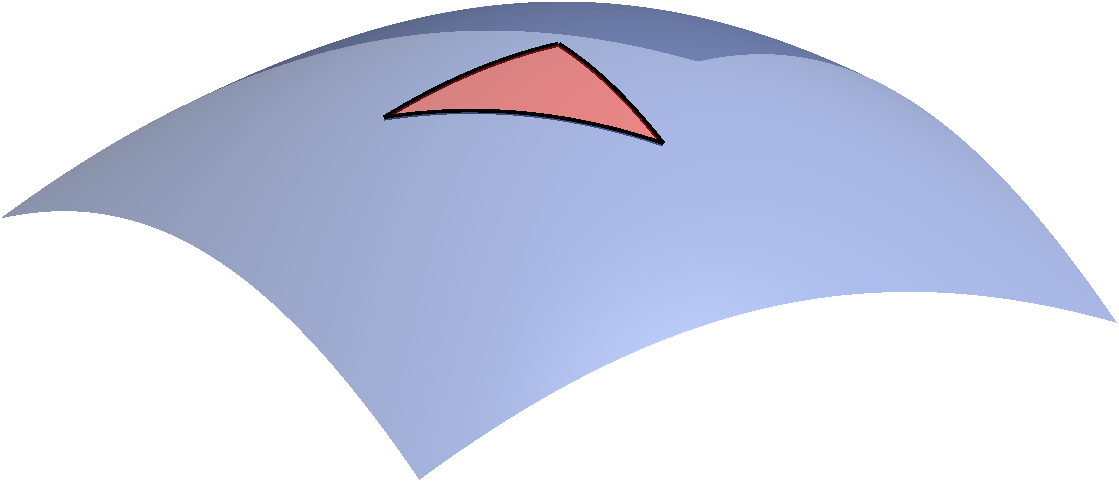}};
\node[img] (N)  at (8.6,3.45) {\includegraphics[width=5.0cm]{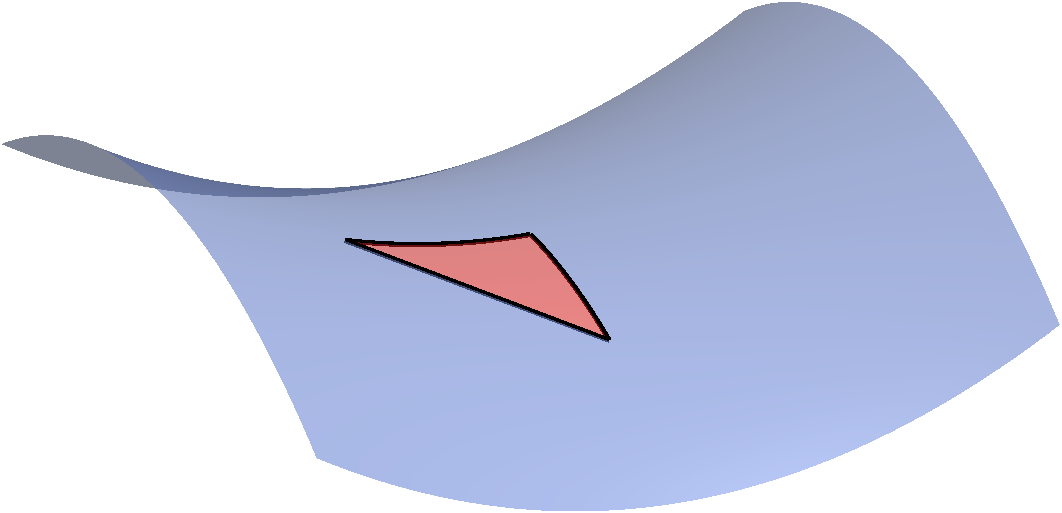}};
\node[img] (Mh) at (0,0)      {\includegraphics[width=5.0cm]{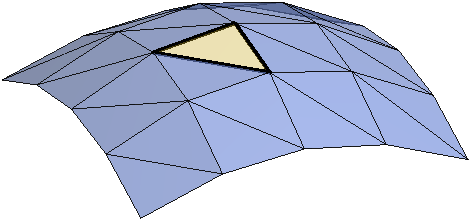}};
\node[img] (Nh) at (8.6,0)    {\includegraphics[width=5.0cm]{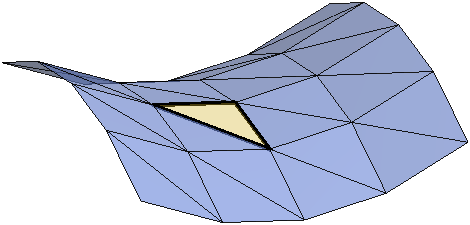}};
\foreach \n/\sx/\sy/\tx/\ty/\S/\T in {
   M/0.86/0.10/0.48/1.13/{\M}/{\pi_\M(\tau)},
   N/0.93/0.08/0.40/0.86/{\N}/{\widehat f(\tau)},
   Mh/0.86/0.07/0.43/1.12/{\Mh}/{\tau},
   Nh/0.93/0.05/0.34/0.80/{\Nh}/{f(\tau)}}{
  \begin{scope}[shift={(\n.south west)},x={($(\n.south east)-(\n.south west)$)},y={($(\n.north west)-(\n.south west)$)}]
    \node[lbl] at (\sx,\sy) {$\S$};
    \node[lbl,font=\footnotesize] at (\tx,\ty) {$\T$};
  \end{scope}}
\draw[->] (2.75,3.45) -- node[above,lbl] {$\widetilde f$} (5.85,3.45);
\draw[->] (2.75,0.05) -- node[below,lbl] {$f$} (5.85,0.05);
\draw[->] (-1.2,1.05) -- node[left,lbl] {$\pi_{\M^e}$} (-1.2,2.35);
\draw[->] (2.3,1.05) -- node[sloped,above,lbl] {$\widehat f=\widetilde f\circ\pi_{\M^e}$} (6.3,2.55);
\end{tikzpicture}
\caption{The diffeomorphism $\widetilde f$, the closest-point projection $\pi_{\M^e}$, the comparison map $\widehat f=\widetilde f\circ\pi_{\M^e}$, and the piecewise linear interpolant $f$, which agrees with $\widehat f$ at the vertices.}
\label{fig:simplicial_approx}
\end{figure}

We then consider the diffeomorphism $\widetilde{f}: \M \to \N$ and its interpolant $f$, as illustrated in Figure~\ref{fig:simplicial_approx}. If $\partial\M\neq\varnothing$, we consider a diffeomorphic extension of $\widetilde f$ between neighborhoods of $\M$ and $\N$ in compatible smooth extensions $\M^e$ and $\N^e$, and we use the same notation for this extension.

To compare the area change directly, we define the smooth comparison map
\begin{equation}\label{eq:composite_map}
\widehat f := \widetilde f \circ \pi_{\M^e} : \Mh \to \N^e,
\end{equation}
and establish the following lemma.

\begin{lemma}  \label{lem:inter_area_change}
Let $f: \Mh \to \Nh$ be the piecewise linear interpolant of the diffeomorphism $\widetilde{f}: \M \to \N$, and suppose $\M$ and $\Mh$ satisfy Assumption~\ref{ass:triangulation}. Then,
\begin{equation}\label{eq:inter_area_change}
\left\|f - \widehat f  \right\|_{L^\infty(\Mh)}
\leq c h^2, \quad
\Big\|\mathrm{D} f-\mathrm{D} \widehat f \Big\|_{L^\infty(\Mh)} \leq Ch, \quad
\left\|J_{f}-J_{\widehat f}  \right\|_{L^\infty(\Mh)}
\leq C' h,
\end{equation}
for some constants $c,C,C'>0$ independent of $h$.
\end{lemma}

\begin{proof}
Let $\tau = [v_i, v_j, v_k] \subset \Mh$ be a fixed triangle.
Choose orthonormal coordinates on the affine plane containing $\tau$ and identify $\tau$ with a planar triangle in $\R^2$.
Since $f$ is the piecewise linear interpolant of $\widetilde f$, it agrees with $\widehat f$ at the vertices of $\tau$:
\[
f(v_\ell) = \widetilde f(v_\ell) = \widehat f(v_\ell),
\qquad \ell \in \{i,j,k\}.
\]
For $x\in \tau$, the barycentric coordinate formula \eqref{eq:bary_coord} gives
\begin{equation} \label{eq:inter_val}
f(x)=\sum_{\ell \in \{i,j,k\}} \widehat f(v_\ell)
\,\lambda^\tau_\ell(x) \in \R^{3}.
\end{equation}
We apply the Taylor expansion of $\widehat f$ about $x$:
\begin{equation} \label{eq:Taylor}
\widehat f(v_\ell)=\widehat f(x)
+ \mathrm{D}\widehat f(x)(v_\ell-x)+ R_\ell(x),
\end{equation}
where the remainder $R_\ell(x) \in \R^3$ satisfies
\begin{equation}
\|R_\ell(x)\| \leq c|v_\ell-x|^2 \leq c \diam{\tau}^2,
\label{eq:inter_remain}
\end{equation}
with a constant $c>0$. Thus, substituting \eqref{eq:Taylor} into \eqref{eq:inter_val}, we obtain
\begin{align*}
f(x)
&= \sum_{\ell \in \{i,j,k\}} \Big(\widehat f(x) + \mathrm{D} \widehat f(x)(v_\ell - x) + R_\ell(x)\Big)  \lambda^\tau_\ell(x) \\
&= \widehat f(x) \sum_{\ell \in \{i,j,k\}} \lambda^\tau_\ell(x) + \mathrm{D} \widehat f(x) \sum_{\ell \in \{i,j,k\}} (v_\ell - x)  \lambda^\tau_\ell(x) + \sum_{\ell \in \{i,j,k\}} R_\ell(x) \lambda^\tau_\ell(x).
\end{align*}
Since the barycentric coordinates satisfy the identities
\begin{equation}
\sum_{\ell \in \{i,j,k\}} \lambda^\tau_\ell(x) = 1, \qquad
\sum_{\ell \in \{i,j,k\}} v_\ell  \, \lambda^\tau_\ell(x) = x,
\label{eq:bary_identity}
\end{equation}
we have
\[
f(x) - \widehat f(x) = \sum_{\ell \in \{i,j,k\}} R_\ell(x) \lambda^\tau_\ell(x).
\]
Since $0\leq \lambda^\tau_\ell(x)\leq 1$ on $\tau$, by~\eqref{eq:inter_remain},
\[
\|R_\ell(x) \, \lambda^\tau_\ell(x)\|  \leq \|R_\ell(x)\| \, |\lambda^\tau_\ell(x)| \leq c \diam{\tau}^2,
\]
and consequently,
\begin{align}\label{eq:interpolant_error}
\Big\| f -  \widehat f \Big\|_{L^\infty(\Mh)}
\leq  \sup_{\tau \in \F}\sup_{x \in \tau}\sum_{\ell \in \{i,j,k\}} \|R_\ell(x) \, \lambda^\tau_\ell(x)\|
\leq c h^2,
\end{align}
as illustrated in Figure~\ref{fig:interp}.

\begin{figure}[tbp]
\centering
\begin{tikzpicture}[line width=0.5pt,>=latex]
\coordinate (vi) at (0,0);
\coordinate (vj) at (3.0,0.25);
\coordinate (vk) at (1.1,2.1);
\coordinate (x)  at (1.45,0.85);
\fill[gray!10] (vi) -- (vj) -- (vk) -- cycle;
\draw[line width=0.7pt] (vi) -- (vj) -- (vk) -- cycle;
\foreach \v in {vi,vj,vk}{\draw[dashed,gray!55!black] (x) -- (\v);}
\fill (x) circle (1.3pt) node[anchor=south west,font=\small,inner sep=1.5pt] {$x$};
\fill (vi) circle (1.3pt) node[anchor=north east,font=\small] {$v_i$};
\fill (vj) circle (1.3pt) node[anchor=north west,font=\small] {$v_j$};
\fill (vk) circle (1.3pt) node[anchor=south,font=\small] {$v_k$};
\node[font=\small] at (2.55,1.55) {$\tau\subset\Mh$};
\node[font=\footnotesize,gray!40!black,align=center] at (1.5,-0.75) {$|v_\ell-x|\le\operatorname{diam}(\tau)\le h$};
\draw[->,line width=0.6pt] (3.7,1.0) -- (5.3,1.0);
\node[font=\small,above] at (4.5,1.0) {$f,\ \widehat f$};
\draw[line width=1.3pt,gray!55] plot[domain=6.2:10.6,samples=80] (\x,{0.3+1.3*(1-((\x-8.4)/2.2)^2)});
\draw[line width=0.9pt] (6.2,0.3) -- (10.6,0.3);
\fill (6.2,0.3) circle (1.5pt);
\fill (10.6,0.3) circle (1.5pt);
\node[anchor=south west,font=\small,gray!45!black] at (10.2,0.72) {$\widehat f(\tau)\subset\N$};
\node[anchor=north,font=\small] at (10.0,0.22) {$f(\tau)$};
\draw[<->,line width=0.5pt] (8.4,0.34) -- (8.4,1.56);
\node[anchor=south,font=\footnotesize] at (8.4,1.62) {$|f-\widehat f|\le ch^2$};
\draw[line width=0.6pt] (6.25,0.4397) -- (7.55,1.557);
\draw[gray!40!black] (6.8,0.9124) -- (7.6,0.9124);
\draw[gray!40!black] (7.4,0.9124) arc[start angle=0,end angle=40.68,radius=0.6];
\draw[gray!55!black,line width=0.3pt] (7.1,0.98) -- (6.95,0.16);
\node[anchor=north,font=\footnotesize] at (6.95,0.18) {$|\mathrm{D}f-\mathrm{D}\widehat f|\le Ch$};
\end{tikzpicture}
\caption{
The maps $f$ and $\widehat f$ agree at the vertices of $\tau$; their values differ by $\O(h^2)$ and their derivatives by $\O(h)$.}
\label{fig:interp}
\end{figure}

The gradient can be estimated with the same approach. Differentiating the barycentric representation \eqref{eq:inter_val} gives
\begin{equation} \label{eq:inter_grad}
\mathrm{D} f(x)=\sum_{\ell \in \{i,j,k\}} \widehat f(v_\ell)
\,\nabla \lambda^\tau_\ell(x)^\top \in \R^{3 \times 2}.
\end{equation}
Substituting the Taylor expansion \eqref{eq:Taylor} into \eqref{eq:inter_grad} gives
\begin{align*}
&\mathrm{D} f(x)
= \sum_{\ell \in \{i,j,k\}} \Big(\widehat f(x) + \mathrm{D} \widehat f(x)(v_\ell - x) + R_\ell(x)\Big) \nabla \lambda^\tau_\ell(x)^\top \\
&= \widehat f(x) \sum_{\ell \in \{i,j,k\}} \nabla \lambda^\tau_\ell(x)^\top + \mathrm{D} \widehat f(x) \sum_{\ell \in \{i,j,k\}} (v_\ell - x) \nabla \lambda^\tau_\ell(x)^\top  + \sum_{\ell \in \{i,j,k\}} R_\ell(x)\nabla \lambda^\tau_\ell(x)^\top.
\end{align*}
Differentiating the barycentric identities~\eqref{eq:bary_identity} gives
\begin{equation*}
\sum_{\ell \in \{i,j,k\}} \nabla \lambda^\tau_\ell(x)^\top = 0, \qquad
\sum_{\ell \in \{i,j,k\}} v_\ell  \, \nabla \lambda^\tau_\ell(x)^\top = I_2,
\end{equation*}
and thus
\begin{equation*}
\mathrm{D} f(x) - \mathrm{D}\widehat f(x) = \sum_{\ell \in \{i,j,k\}} R_\ell(x) \, \nabla\lambda^\tau_\ell(x)^\top.
\end{equation*}
Shape regularity ensures that $\|\nabla \lambda^\tau_\ell\| \leq c \diam{\tau}^{-1}$ for some constant $c > 0$, so by~\eqref{eq:inter_remain},
\[
\|R_\ell(x) \, \nabla\lambda^\tau_\ell(x)^\top\|  \leq \|R_\ell(x)\| \, \|\nabla\lambda^\tau_\ell(x)\| \leq C \diam{\tau},
\]
for some constant $C>0$. Consequently,
\begin{align}\label{eq:interpolant_gradient_error}
\Big\|\mathrm{D} f-\mathrm{D} \widehat f \Big\|_{L^\infty(\Mh)}
\leq \sup_{\tau \in \F} \sup_{x \in \tau}\sum_{\ell \in \{i,j,k\}} \|R_\ell(x) \, \nabla\lambda^\tau_\ell(x)^\top\|
\leq  C h.
\end{align}

It remains to convert this gradient estimate into a Jacobian estimate.
We write
\[
\mathrm{D} f = \Big[ \partial_{x_1} f ~~ \partial_{x_2} f  \Big], \qquad
\mathrm{D} \widehat f = \Big[ \partial_{x_1} \widehat f ~~\partial_{x_2} \widehat f \Big].
\]
For $x \in \tau$, the Jacobian satisfies
\begin{align*}
\Big| J_f(x) - J_{\widehat f}(x) \Big|
&\leq \Big| \|\partial_{x_1} f \times \partial_{x_2} f\| - \|\partial_{x_1} \widehat f \times \partial_{x_2} \widehat f \| \Big| \\
&\leq \Big\| (\partial_{x_1} f -  \partial_{x_1} \widehat f) \times \partial_{x_2} f \Big\| + \Big\| \partial_{x_1} \widehat f \times (\partial_{x_2} f - \partial_{x_2} \widehat f  )\Big\|\\
&\leq \Big\| \partial_{x_1} f -  \partial_{x_1} \widehat f \Big\| \, \Big\| \partial_{x_2} f \Big\| + \Big\| \partial_{x_1} \widehat f \Big\| \, \Big\| \partial_{x_2} f - \partial_{x_2} \widehat f  \Big\|.
\end{align*}
Since $\widetilde f$ and $\pi_{\M^e}$ are smooth on the compact sets, $\mathrm D\widehat f$ is uniformly bounded. Estimate~\eqref{eq:interpolant_gradient_error} then implies that $\mathrm D f$ is also uniformly bounded for sufficiently small $h$. Hence,
\[
\Big| J_f(x) - J_{\widehat f}(x) \Big|  \leq  C \, \Big\|\mathrm{D} f-\mathrm{D} \widehat f \Big\|_{L^\infty(\tau)}  \leq \widetilde C \diam{\tau},
\]
for some constant $\widetilde C>0$.
Since this holds for every triangle $\tau$, the proof is complete.
\end{proof}

With the Jacobian estimate, we now prove the consistency of the stretch energy.

\begin{theorem} \label{thm:consistency}
Let $f:\Mh \to \Nh$ be the piecewise linear interpolant of a diffeomorphism $\widetilde f: \M \to \N$, and $\M$ and $\Mh$ satisfy Assumption~\ref{ass:triangulation}. Then, for sufficiently small $h$,
\begin{equation}\label{eq:consistency}
\Big| E_\rS(f) - \sE_\rS(\widetilde f) \Big| \leq C  h,
\end{equation}
for some constant $C>0$.
\end{theorem}

\begin{proof}
If $\partial\M\neq\varnothing$, we use the fixed smooth extension of $\widetilde f$ introduced above.
By the chain rule and~\eqref{eq:proj_area_estimate},
\begin{align*}
\Big\| J_{\widehat f} - J_{\widetilde f} \circ \pi_{\M^e} \Big\|_{L^\infty(\Mh)}
 &= \Big\|(J_{\widetilde f} \circ \pi_{\M^e}) \,J_{\pi_{\M^e}} - J_{\widetilde f} \circ \pi_{\M^e} \Big\|_{L^\infty(\Mh)}\\
&\leq \Big\|J_{\widetilde f} \circ \pi_{\M^e} \Big\|_{L^\infty(\Mh)} \, \Big\|\,J_{\pi_{\M^e}} - 1 \Big\|_{L^\infty(\Mh)}
 \leq c h^2,
\end{align*}
for some constant $c>0$.
Combining this with the Jacobian estimate~\eqref{eq:inter_area_change} via the
triangle inequality, we obtain
\begin{align*}
\Big\| J_{f} - J_{\widetilde f} \circ \pi_{\M^e} \Big\|_{L^\infty(\Mh)}
&\leq \Big\| J_{f} - J_{\widehat f} \Big\|_{L^\infty(\Mh)}
+ \Big\|J_{\widehat f} - J_{\widetilde f} \circ \pi_{\M^e} \Big\|_{L^\infty(\Mh)} \\
&\leq C h + c h^2 \leq  C' h,
\end{align*}
for some constants $C, C'>0$. Hence, using $a^2-b^2=(a-b)(a+b)$, we get
\begin{equation} \label{eq:squared_jacobian_error}
\Big\|J_{f}^2 - (J_{\widetilde f} \circ \pi_{\M^e})^2 \Big\|_{L^\infty(\Mh)} \leq  C \Big\|J_{f} - J_{\widetilde f} \circ \pi_{\M^e} \Big\|_{L^\infty(\Mh)} \leq  ch
\end{equation}
for some constants $C,c>0$.

Combining these estimates with the symmetric-difference estimate \eqref{eq:lifted_domain_estimate} gives
\begin{align*}
&\bigg|\sE_\rS(\widetilde f) - \int_\Mh J_{\widetilde f}(\pi_{\M^e}(x))^2 \,\d A\bigg| \\
&\leq  \bigg|\int_\M J_{\widetilde f}^2\,\d A_g -\int_{\M_h^\ell}J_{\widetilde f}^2\,\d A_g  \bigg| +
  \bigg|\int_{\M_h^\ell}J_{\widetilde f}^2\,\d A_g  - \int_\Mh J_{\widetilde f}(\pi_{\M^e}(x))^2 \,\d A\bigg| \\
&\leq \bigg|\int_\M J_{\widetilde f}^2\,\d A_g
-\int_{\M_h^\ell}J_{\widetilde f}^2\,\d A_g\bigg|
+\bigg|\int_\Mh J_{\widetilde f}(\pi_{\M^e}(x))^2
\bigl(J_{\pi_{\M^e}}(x)-1\bigr)\,\d A\bigg|\\
&\leq \left\|J_{\widetilde f}^2 \right\|_{L^\infty}
\,|\M_h^\ell\mathbin{\triangle}\M|
+\left\|J_{\widetilde f}^2 \right\|_{L^\infty}
\,\|J_{\pi_{\M^e}}-1\|_{L^\infty(\Mh)}|\Mh|\\
&\leq Ch^2
\end{align*}
for some constant $C>0$, where~\eqref{eq:proj_area_estimate} and~\eqref{eq:lifted_domain_estimate} were used. Note that if $\partial \M = \varnothing$, $\M_h^\ell=\M$ and the first term vanishes.
Consequently, using~\eqref{eq:squared_jacobian_error}, we obtain
\begin{align*}
\big| E_\rS(f) - \sE_\rS(\widetilde f) \big|
&\leq \bigg| E_\rS(f) -  \int_\Mh J_{\widetilde f}(\pi_{\M^e}(x))^2  \d A \bigg| + \bigg|  \int_\Mh J_{\widetilde f}(\pi_{\M^e}(x))^2  \d A - \sE_\rS(\widetilde f) \bigg|   \\
& \leq \int_\Mh \Big| J_{f}^2 - J_{\widetilde f}(\pi_{\M^e}(x))^2 \Big|\, \d A +  C h^2 \\
& \leq c h\int_\Mh 1\, \d A +  C h^2 \leq C'h
\end{align*}
for some constant $C' > 0$.
\end{proof}

\subsection{Convergence of the discrete infimum and area distortion}
\label{sec:4.3}
The previous subsection provides the pointwise convergence between $E_\rS$ and $\E_\rS$ for a fixed diffeomorphism $\widetilde f$ with its piecewise linear (PL) interpolant under mesh refinement. However, in practical computation, the map attaining $\inf E_\rS$ for each mesh size $h$ need not be the interpolant of the same diffeomorphism. In other words, the variational principle of $\E_\rS$ need not pass to $E_\rS$.

Fortunately, in our setting, we have shown that the global minimizers of $\sE_\rS$ are precisely the area-preserving diffeomorphisms. Hence, its piecewise linear interpolant serves as a discrete competitor, which is used for optimal value convergence.

To establish such convergence, we require the geometric control of the image mesh under a sufficiently small mesh size. Note that we denote $h$ and $k_h$ as the mesh sizes of the source and image meshes. We first show that the image mesh is shape-regular for sufficiently small mesh size.

\begin{lemma} \label{lem:image_shape}
Let $f_h$ be the PL interpolant of an orientation-preserving diffeomorphism $\widetilde f:\M \to \N$ with $\M$ and $\Mh$ satisfying Assumption~\ref{ass:triangulation}. Then, for all sufficiently small $h$, the image mesh $f(\Mh) := \Nh$ satisfies the shape-regular condition in \eqref{eq:shape_regular}, and $ch\leq k_h \leq Ch$, where $c,C>0$ are independent of $h$.
\end{lemma}

\begin{proof}
Since $\widetilde f$ in \eqref{eq:composite_map} is a diffeomorphism on the compact sets, its differential and inverse differential are uniformly bounded. Moreover, since $\pi_{\M^e}$ is smooth and by~\eqref{eq:proj_area_estimate}, the two singular values of $\mathrm D\pi_{\M^e}|_\tau$ are bounded above and away from zero. Hence, there exist constants $m,M>0$ such that
\[
m |\xi| \leq |\mathrm{D} \widehat{f}(x) \xi| \leq M |\xi|,
\]
for every $x \in \tau \in \F$ and every tangent vector $\xi$ to the affine plane of $\tau$.
By the estimate of differential in~\eqref{eq:inter_area_change}, for all sufficiently small $h$,
\begin{equation}  \label{eq:admissible_1}
c |\xi| \leq |\mathrm{D} f_h(x) \xi| \leq C |\xi|,
\end{equation}
with $c,C>0$ independent of $h$ and $\tau$, and $f_h|_\tau$ is non-degenerate and orientation-preserving. Since $f_h|_\tau$ is affine, for all $x,y\in\tau$, \eqref{eq:admissible_1} gives
\begin{equation*}
c |x-y|\leq |f_h(x)-f_h(y)|\leq C|x-y|.
\end{equation*}
Consequently,
\[
c\diam{\tau}\leq \diam{f_h(\tau)}\leq C\diam{\tau}
\]
and taking the maximum over $\tau \in \F$  gives
\[
ch\leq k_h\leq Ch.
\]
In addition, since $f_h|_\tau$ is affine, let $r =\inrad{\tau}$ and let $B_{x_0}(r)\subset\tau$ be the inscribed disk of $\tau$. Then \eqref{eq:admissible_1} implies
\[
f_h(B_{x_0}(r))
= f_h(x_0) + \mathrm{D} f_h(B_{0}(r))
\supset f_h(x_0) + cB_{0}(r)
= B_{f_h(x_0)}(c\,r).
\]
Hence, we have $\inrad{f_h(\tau)}\geq c\inrad{\tau}$ and
\[
\frac{\diam{f_h(\tau)}}{\inrad{f_h(\tau)}}
\leq \frac{C}{c}\frac{\diam{\tau}}{\inrad{\tau}} \leq C'
\]
for some constant $C' > 0$, i.e., the image meshes are shape-regular.
\end{proof}

Next, we provide the estimate for the boundary of the image surface.

\begin{lemma}  \label{lem:boundary_estimate}
Let $f_h:\Mh \to \Nh$ be the PL interpolant of an orientation-preserving diffeomorphism $\widetilde f:\M \to \N$ with $\M$ and $\Mh$ satisfying Assumption~\ref{ass:triangulation}. We denote $\widehat\N_h:=\widehat f(\Mh)$ and $\partial \widehat \N_h:= \widehat f(\partial \Mh)$, where $\widehat f$ is given by \eqref{eq:composite_map}. For all sufficiently small $h$,
\begin{itemize}
\item $\partial \widehat \N_h$ differs from $\partial \N$ by $\O(h^2)$ in position and $\O(h)$ in tangent direction.
\item $\widehat \N_h$ is uniformly quasiconvex, i.e., there exist constants $r,L>0$ independent of $h$, such that any $z_1,z_2\in\widehat\N_h$ with $|z_1-z_2|<r$ can be joined by a curve $\gamma\subset\widehat\N_h$ satisfying
\begin{equation} \label{eq:quasiconvex}
\operatorname{length}(\gamma) \leq L |z_1-z_2|.
\end{equation}
\end{itemize}
\end{lemma}

\begin{figure}[t]
\centering
\begin{tikzpicture}[line width=0.5pt, scale=1.25]

\fill[gray!12]
  (-4.2,0.10) -- (-3.4,-0.06) -- (-2.6,0.08) -- (-1.8,-0.05) -- (-1.0,0.07)
  -- (-0.5,-0.30) -- (0,-0.62) -- (0.5,-0.30) -- (1.0,0.07) -- (1.8,-0.05)
  -- (2.6,0.08) -- (3.4,-0.10) -- (4.2,-0.30)
  -- (4.2,-1.95) -- (-4.2,-1.95) -- cycle;
\node[font=\small] at (-3.05,-1.30) {$\widehat\N_h$};

\draw[line width=0.6pt] (-4.2,0) -- (4.25,0);
\node[anchor=west,font=\small] at (4.32,0) {$\partial\N$};

\draw[line width=1.3pt]
  (-4.2,0.10) -- (-3.4,-0.06) -- (-2.6,0.08) -- (-1.8,-0.05) -- (-1.0,0.07)
  -- (-0.5,-0.30) -- (0,-0.62) -- (0.5,-0.30) -- (1.0,0.07) -- (1.8,-0.05)
  -- (2.6,0.08) -- (3.4,-0.10) -- (4.2,-0.30);
\node[anchor=west,font=\small] at (4.32,-0.32) {$\partial\widehat\N_h$};

\coordinate (z1) at (-0.85,-0.22);
\coordinate (z2) at ( 0.85,-0.22);

\draw[dashed,line width=0.5pt,gray!55!black] (z1) -- (z2);
\draw[line width=0.6pt,gray!45!black] (-0.10,-0.32) -- (0.10,-0.12);
\draw[line width=0.6pt,gray!45!black] (-0.10,-0.12) -- (0.10,-0.32);

\draw[line width=1.1pt] (z1) -- (-0.85,-1.22) -- (0.85,-1.22) -- (z2);
\node[anchor=north,font=\small] at (0,-1.25) {$\gamma$};

\fill (z1) circle (1.6pt);
\fill (z2) circle (1.6pt);
\node[anchor=east,font=\small] at (-0.95,-0.20) {$z_1$};
\node[anchor=west,font=\small] at ( 0.95,-0.20) {$z_2$};

\end{tikzpicture}
\caption{
The straight segment between two nearby interior points $z_1,z_2$ may leave the domain, but the path $\gamma$ has length at most a fixed multiple of $|z_1-z_2|$.}
\label{fig:quasiconvex}
\end{figure}

\begin{proof}
Choose a boundary edge $[v_i, v_j] \subset \partial \Mh$ and its corresponding boundary arc on $\partial \M$. We use the arc-length parameterization of that arc,
\[
\gamma:[0, \ell] \to \partial \M, \quad \|\gamma'(u)\| = 1, \quad \gamma(0) = v_i, \quad \gamma(\ell) = v_j,
\]
where $\ell \leq C h$ for some $C>0$. The boundary edge is the chord given by
\begin{equation} \label{eq:boundary_estimate_1}
c(u) = \gamma(0) + \frac{u}{\ell} (\gamma(\ell) - \gamma(0)), \qquad 0 \leq u \leq \ell.
\end{equation}

Taylor's expansion around $u = 0$ gives
\begin{equation} \label{eq:boundary_estimate_2}
\gamma(u) - \gamma(0) = u \,\gamma'(0) + R(u), \qquad \|R(u)\| \leq C u^2.
\end{equation}
Substituting \eqref{eq:boundary_estimate_2} with $u = \ell$ into the chord \eqref{eq:boundary_estimate_1}, we get
\begin{align*}
c(u) &= \gamma(0) + \frac{u}{\ell} (\gamma(\ell) - \gamma(0))\\
&= \gamma(0) + \frac{u}{\ell} \Big( \gamma(0)+ \ell \,\gamma'(0) + R(\ell) - \gamma(0) \Big)\\
&= \gamma(0) + u \, \gamma'(0) + \frac{u}{\ell} R(\ell).
\end{align*}
Therefore, we have
\begin{align*}
\|c(u) - \gamma(u)\| &= \Big\| \frac{u}{\ell} R(\ell) - R(u) \Big\|
  \leq  \Big\| \frac{u}{\ell} R(\ell) \Big\| + \Big\| R(u) \Big\|\\
  &\leq Cu \ell + C u^2 \leq C \ell^2 \leq Ch^2.
\end{align*}

In addition, taking the derivative of the chord \eqref{eq:boundary_estimate_1} and substituting \eqref{eq:boundary_estimate_2} with $u = \ell$ into it gives
\begin{equation*}
c'(u) = \frac{\gamma(\ell) - \gamma(0)}{\ell} = \gamma'(0) + \frac{R(\ell)}{\ell}.
\end{equation*}
Since $\|\gamma'(u) - \gamma'(0)\| \leq C \ell$, we have
\[
\|c'(u) - \gamma'(u)\| \leq \|c'(u) - \gamma'(0)\| +\|\gamma'(0) - \gamma'(u)\| \leq C\ell \leq Ch.
\]
Applying the same arguments to all boundary edges, we we find that $\partial \Mh$ and $\partial \M$ differ by $\O(h^2)$ in position and $\O(h)$ in tangent direction. Finally, since both $\pi_{\M^e}$ and $\widetilde f$ are smooth on a compact set, they and their differentials are Lipschitz continuous, and $\partial \widehat \N_h = \widetilde f\circ \pi_{\M^e} (\partial\Mh)$ and $\partial\N = \widetilde f(\partial\M)$ have the same order of error.

For quasi-convexity, if $z_1, z_2 \in \widehat \Nh$ away from the boundary, this follows directly since $\widehat \N_h = \widetilde f(\M_h^\ell)$ is smooth and compact. If $z_1, z_2 \in \widehat \Nh$ near the boundary, since $\partial \widehat \Nh$ differs from $\partial \N$ by positional error of $\O(h^2)$ and a tangent-direction error of $\O(h)$, both curves are close without reversing direction. As illustrated in Figure~\ref{fig:quasiconvex}, two sufficiently close points on the interior side can be joined by moving inward, across, and back, with total length bounded by a fixed multiple of their distance, independent of $h$; that is, $\widehat \N_h$ is quasiconvex.
\end{proof}

To compare $\widehat \N_h$ and $\N_h^\ell:=\pi_{\N^e}(\Nh)$ directly, we define
\begin{equation}  \label{eq:qh}
q_h:=\pi_{\N^e}\circ f_h\circ\widehat f^{-1}
:\widehat\N_h\to\N^e,
\end{equation}
which is well defined as the image $f_h(\Mh)$ lies in a fixed tubular neighborhood of $\N^e$ for sufficiently small $h$, as illustrated in Figure~\ref{fig:qh_diagram}. We provide the following lemma.

\begin{figure}[t]
\centering
\begin{tikzpicture}[>=latex, line width=0.5pt]

\fill[gray!8,rounded corners=3pt] (-1.35,0.45) rectangle (6.15,1.45);
\fill[gray!8,rounded corners=3pt] (-1.35,-1.45) rectangle (6.15,-0.45);
\node[anchor=west,font=\footnotesize\itshape,gray!55!black] at (6.35,0.95)
      {piecewise linear};
\node[anchor=west,font=\footnotesize\itshape,gray!55!black] at (6.35,-0.95)
      {smooth, in $\N^e$};

\node (Mh)  at (0,0.95)   {$\Mh$};
\node (Nh)  at (4.8,0.95) {$\Nh$};
\node (hNh) at (0,-0.95)  {$\widehat\N_h$};
\node (Nl)  at (4.8,-0.95){$\N_h^\ell$};

\draw[->] (Mh) -- node[above,font=\small] {$f_h$} (Nh);
\draw[->] (Mh) -- node[left,font=\small] {$\widehat f$} (hNh);
\draw[->] (Nh) -- node[right,font=\small] {$\pi_{\N^e}$} (Nl);
\draw[->] (hNh) -- node[below,font=\small] {$q_h$} (Nl);

\node[font=\large,gray!60!black] at (2.4,0) {$\circlearrowright$};
\end{tikzpicture}

\medskip
{\small
$\widehat f=\widetilde f\circ\pi_{\M^e}$,\qquad
$q_h=\pi_{\N^e}\circ f_h\circ\widehat f^{\,-1}$.
}
\caption{The map $q_h$ sends the image of $\widehat f$ to the image of $\pi_{\N^e}\circ f_h$.
}
\label{fig:qh_diagram}
\end{figure}

\begin{lemma} \label{lem:injective_qh_fh}
Let $f_h$ be the PL interpolant of an orientation-preserving diffeomorphism $\widetilde f:\M \to \N$ with $\M$ and $\Mh$ satisfying Assumption~\ref{ass:triangulation}. For all sufficiently small $h$,
\begin{itemize}
\item $q_h$ in \eqref{eq:qh} is injective and satisfies
\begin{equation} \label{eq:injective_qh_fh_1}
\|q_h-\operatorname{id}\|_{L^\infty(\widehat\N_h)}
\leq Ch^2, \qquad
\|\mathrm Dq_h-I\|_{L^\infty(\widehat\N_h)}
\leq Ch.
\end{equation}
\item $f_h$ is an orientation-preserving piecewise linear (PL) homeomorphism.
\end{itemize}
\end{lemma}

\begin{proof}
For $z\in\widehat\N_h \subset \N^e$, set $x=\widehat f^{-1}(z)$. Since
\[
\pi_{\N^e}(z) =
\pi_{\N^e}(\widehat f(x))=\widehat f(x)=z,
\]
the smoothness of $\pi_{\N^e}$ and~\eqref{eq:interpolant_error} give
\begin{align*}
|q_h(z)-z|
&=\big|\pi_{\N^e}(f_h(x))-\pi_{\N^e}(\widehat f(x))\big|\\
&\leq C|f_h(x)-\widehat f(x)|
\leq Ch^2,
\end{align*}
which shows the left estimate in \eqref{eq:injective_qh_fh_1}.
For the derivative estimate, we choose local coordinates on $\N^e$ so that both $\mathrm{D} q_h$ and $I$ are $2$-by-$2$ matrices. The chain rule gives
\[
\mathrm Dq_h(z)
=\mathrm D\pi_{\N^e}(f_h(x))\,\mathrm Df_h(x)\,
\mathrm D\widehat f^{-1}(z), \qquad
I
=\mathrm D\pi_{\N^e}(\widehat f(x))\,\mathrm D\widehat f(x)\,
\mathrm D\widehat f^{-1}(z).
\]
Consequently,
\begin{align*}
\mathrm Dq_h(z)-I
&=\Big(\mathrm D\pi_{\N^e}(f_h(x))
-\mathrm D\pi_{\N^e}(\widehat f(x))\Big)
\mathrm Df_h(x)\,\mathrm D\widehat f^{-1}(z)\\
&\quad+\mathrm D\pi_{\N^e}(\widehat f(x))
\Big(\mathrm Df_h(x)-\mathrm D\widehat f(x)\Big)
\mathrm D\widehat f^{-1}(z).
\end{align*}
The smoothness of $\pi_{\N^e}$, the uniform boundedness of $\mathrm Df_h$ and $\mathrm D\widehat f^{-1}$, and estimates~\eqref{eq:interpolant_error} and \eqref{eq:interpolant_gradient_error} imply
\begin{align*}
|\mathrm Dq_h(z)-I| & \leq
\big\|  \mathrm D^2\pi_{\N^e} \big\| \, \big\| f_h - \widehat f \big\|_{L^\infty(\Mh)} \big\|  \mathrm D f_h\big\| \big\|  \mathrm{D} \widehat{f}^{-1}\big\| \\
& \quad+ \big\|  \mathrm D\pi_{\N^e} \big\| \,\big\| \mathrm{D} f_h - \mathrm{D}\widehat f \big\|_{L^\infty(\Mh)} \big\|  \mathrm{D} \widehat{f}^{-1}\big\|\\
&\leq ch^2 + Ch \leq C'h,
\end{align*}
which proves the claim in \eqref{eq:injective_qh_fh_1}.

Suppose that $q_h(z_1)=q_h(z_2)$. By the left estimate in \eqref{eq:injective_qh_fh_1},
\[
|z_1-z_2|  \leq |z_1-q_h(z_1)|  + |q_h(z_2)-z_2|  \leq 2Ch^2.
\]
Hence $|z_1-z_2|<r$ for sufficiently small $h$. Let $\gamma:[0,1]\to\widehat\N_h$ be a curve from $z_1$ to $z_2$ satisfying \eqref{eq:quasiconvex}, and choose a chart containing it. The fundamental theorem of calculus gives
\begin{align*}
q_h(z_2)-q_h(z_1) &=  \int_0^1 \mathrm{D} q_h (\gamma(t)) \gamma'(t) \, \d t
=  \int_0^1 \gamma'(t) \, \d t  + \int_0^1 ( \mathrm{D} q_h - I) \gamma'(t) \, \d t.
\end{align*}
The reverse triangle inequality gives
\begin{equation} \label{eq:injective_qh_fh_2}
|q_h(z_2)-q_h(z_1) | \geq |z_2 - z_1| - \bigg| \int_0^1 ( \mathrm{D} q_h - I) \gamma'(t) \, \d t \bigg|.
\end{equation}
The quasiconvexity \eqref{eq:quasiconvex} and the right estimate in \eqref{eq:injective_qh_fh_1} imply
\begin{align} \label{eq:injective_qh_fh_3}
\bigg| \int_0^1 ( \mathrm{D} q_h - I) \gamma'(t) \, \d t \bigg|
&\leq \int_0^1 \| \mathrm{D} q_h - I\| |\gamma'(t) | \d t
 \leq C h  \operatorname{length}(\gamma)
\leq C L h |z_2 - z_1|.
\end{align}
Combining both \eqref{eq:injective_qh_fh_2} and \eqref{eq:injective_qh_fh_3}, we obtain
\[
|q_h(z_2)-q_h(z_1)| \geq (1 - Ch L) |z_2- z_1|.
\]
Since $q_h(z_1)=q_h(z_2)$, the left-hand side vanishes and $z_1=z_2$ for all sufficiently small $h$, i.e., $q_h$ is globally injective.
We now transfer the injectivity of $q_h$ to that of $f_h$. If $f_h(x_1)=f_h(x_2)$, then
\[
q_h\bigl(\widehat f(x_1)\bigr)
=\pi_{\N^e}\bigl(f_h(x_1)\bigr)
=\pi_{\N^e}\bigl(f_h(x_2)\bigr)
=q_h\bigl(\widehat f(x_2)\bigr).
\]
Since $q_h$ and $\widehat f$ are injective, $x_1=x_2$ and $f_h$ is globally injective. In addition, since $\Mh$ is compact, $f_h$ is an orientation-preserving homeomorphism onto $\Nh=f_h(\Mh)$.
\end{proof}

With these estimates, we can ensure that the image is geometrically controlled as follows.

\begin{lemma} \label{lem:interpolant_admissible}
Let $\widetilde f:\M \to \N$ be an orientation-preserving diffeomorphism, and let $f_h:\Mh \to \Nh$ be its PL interpolant with $\M$ and $\Mh$ satisfying Assumption~\ref{ass:triangulation}. Then, for all sufficiently small $h$,
the image mesh $\Nh:= f_h(\Mh)$ satisfies Assumption~\ref{ass:triangulation} with respect to $\N$ and $ch\leq k_h \leq Ch$, where $c,C>0$ are independent of $h$.
\end{lemma}

\begin{proof}
The shape-regular condition follows immediately from Lemma~\ref{lem:image_shape}. It remains to show that $\pi_{\N^e}$ satisfies the closest-point-projection and boundary-extension conditions.

Suppose $y_1,y_2\in\Nh$ satisfy $\pi_{\N^e}(y_1)=\pi_{\N^e}(y_2)$. Write $y_m = f_h(x_m)$ with $x_m \in \Mh$. Since $q_h$ and $\widehat f$ are injective and
\[
q_h\bigl(\widehat f(x_1)\bigr)
=\pi_{\N^e}\big( f_h(x_1) \big)
=\pi_{\N^e}\big( f_h(x_2) \big)
=q_h\bigl(\widehat f(x_2)\bigr),
\]
we have $x_1=x_2$, and thus $y_1=y_2$. Hence,
\[
\pi_{\N^e}|_{\Nh}:\Nh\to\N_h^\ell
\]
is bijective.

For the boundary-extension condition, for each $y_h \in \partial\widehat\N_h$ there is $y\in\partial\N$ such that, by Lemma~\ref{lem:boundary_estimate},
\[
\operatorname{dist}_{\mathcal N^e} (y_h,\partial\N ) \leq
C| y_h - y| \leq
Ch^2.
\]
Moreover, the estimate of $q_h$ in Lemma~\ref{lem:injective_qh_fh} ensures $\partial\N_h^\ell=q_h(\partial\widehat\N_h)$ differs from $\partial\widehat\N_h$ by $\O(h^2)$ in position and $\O(h)$ in tangent direction. Together with the estimate in Lemma~\ref{lem:boundary_estimate}, $\partial\widehat\N_h$, $\partial\N_h^\ell$, and $\partial\N$ sit close to each other without reversing direction. Hence,
\[
\pi_{\partial\N}:\partial\N_h^\ell\to\partial\N
\]
is bijective, as illustrated in Figure~\ref{fig:boundary_projection}.

Finally, we claim that
\[
\N_h^\ell\mathbin{\triangle}\N
\subset
\{x\in\N^e\,|\,\operatorname{dist}_{\N^e}(x, \partial \N)\le Ck_h^2\}.
\]
By Lemma~\ref{lem:boundary_estimate}, for sufficiently small $h$, we may view $\partial \widehat \N_h$ as a graph over $\partial \N$ with an $\O(h^2)$ deviation, so that
\[
\widehat{\mathcal N}_h\triangle\mathcal N
\subset
\{x\in\N^e\,|\,\operatorname{dist}_{\N^e}(x, \partial \N)\le Ch^2\}.
\]
Then, since $\|q_h-\operatorname{id}\|_{L^\infty}\leq Ch^2$ by Lemma~\ref{lem:injective_qh_fh}, $q_h(\widehat \Nh) = \N_h^\ell$ moves $\widehat \Nh$ to $\N_h^\ell$ within an $\O(h^2)$ strip. Consequently,
\[
\N_h^\ell\mathbin{\triangle}\N
\subset
\{x\in\N^e\,|\,\operatorname{dist}_{\N^e}(x, \partial \N)\le Ch^2\}
\]
and, since $h\leq k_h/c$ by Lemma~\ref{lem:image_shape}, the bound $Ch^2
$ can be written as $Ck_h^2$, which completes the proof.
\end{proof}

\begin{figure}[t]
\centering
\begin{tikzpicture}[line width=0.5pt, scale=1.3]

\foreach \x in {-2.4,-0.8,1.0,2.6}{
  \draw[dashed,line width=0.45pt,gray!55!black] (\x,0.12) -- (\x,-1.26);
}

\draw[line width=0.6pt] (-4.6,0) -- (4.5,0);
\node[anchor=west,font=\small] at (4.6,0) {$\partial\N$};

\draw[line width=1.2pt]
 (-4.4,-0.46) -- (-3.6,-0.55) -- (-2.8,-0.44) -- (-2.0,-0.57) -- (-1.2,-0.45)
 -- (-0.4,-0.55) -- (0.4,-0.43) -- (1.2,-0.56) -- (2.0,-0.45) -- (2.8,-0.54)
 -- (3.6,-0.46) -- (4.4,-0.52);
\node[anchor=west,font=\small] at (4.6,-0.52) {$\partial\widehat\N_h$};

\draw[line width=1.2pt]
 (-4.4,-1.02) -- (-3.6,-1.12) -- (-2.8,-0.99) -- (-2.0,-1.13) -- (-1.2,-1.00)
 -- (-0.4,-1.10) -- (0.4,-0.98) -- (1.2,-1.12) -- (2.0,-1.01) -- (2.8,-1.09)
 -- (3.6,-1.00) -- (4.4,-1.06);
\node[anchor=west,font=\small] at (4.6,-1.06) {$\partial\N_h^\ell$};

\draw[->,>=latex,line width=0.7pt] (-3.6,-0.55) to[bend left=18] (-3.40,-1.0875);
\fill (-3.6,-0.55) circle (1.5pt);
\fill (-3.40,-1.0875) circle (1.5pt);
\node[anchor=west,font=\small] at (-3.26,-0.84) {$q_h$};

\draw[->,>=latex,line width=0.8pt] (3.3,-1.034) -- (3.3,-0.04);
\fill (3.3,-1.034) circle (1.5pt);
\fill (3.3,0) circle (1.5pt);
\node[anchor=west,font=\small] at (3.42,-0.74) {$\pi_{\partial\N}$};

\draw[<->,>=latex,line width=0.4pt] (-4.88,-0.02) -- (-4.88,-0.48);
\draw[<->,>=latex,line width=0.4pt] (-4.88,-0.56) -- (-4.88,-1.04);
\node[anchor=east,font=\footnotesize] at (-4.95,-0.25) {$\O(h^2)$};
\node[anchor=east,font=\footnotesize] at (-4.95,-0.80) {$\O(h^2)$};

\end{tikzpicture}
\caption{The curve $\partial\widehat\N_h$ is a perturbation of $\partial\N$ of size $\O(h^2)$ in position and $\O(h)$ in direction, and $q_h$ adds a perturbation of the same two orders.}
\label{fig:boundary_projection}
\end{figure}

Therefore, the image mesh $\Nh$ is geometrically controlled and satisfies Assumption~\ref{ass:triangulation} for sufficiently small $h$. We can apply the estimates in Section~\ref{sec:4.2} not only to $\Mh$ but also to $\Nh$. This allows us to analyze the interpolant of an area-preserving diffeomorphism as follows.

\begin{lemma} \label{lem:area_preserving_est}
Let $f_h:\Mh \to \Nh$ be the PL interpolant of an orientation-preserving, area-preserving diffeomorphism $\widetilde f:\M\to\N$ with $\M$ and $\Mh$ satisfying Assumption~\ref{ass:triangulation}. Then, for all sufficiently small $h$,
\begin{equation}  \label{eq:area_preserving_est}
\Big| E_\rS(f_h) - |\Mh|  \Big| \leq Ch^2,
\end{equation}
for some constant $C > 0$.
\end{lemma}

\begin{proof}
We bound the deviation of $E_\rS$ from the minimum in Corollary~\ref{cor:Es_var} by
\begin{align}
\Big| E_\rS(f_h) - |\Mh| \Big|
&\leq \Big| E_\rS(f_h) - 2 |\Nh| + |\Mh| \Big| + 2 \, \Big| |\Nh| - |\Mh| \Big| \nonumber\\
&\leq \sum_{\tau \in \F} \bigg|\frac{|f_h(\tau)|^2}{|\tau|} - 2 |f_h(\tau)| + |\tau| \bigg| + 2\, \Big| |\Nh| - |\Mh| \Big| \nonumber\\
&= \sum_{\tau \in \F} |\tau|\Big| J_{f_h|_\tau} - 1 \Big|^2  +  2\, \Big| |\Nh| - |\Mh| \Big|. \label{eq:energy_gap_decomposition}
\end{align}
Fix $x \in \tau \in \F$. By the triangle inequality,
\begin{equation} \label{eq:area_preserving_est_1}
\bigl| J_{f_h|_\tau}-1 \bigr|
\leq \bigl|J_{f_h|_\tau}-J_{\widehat f|_\tau}(x)\bigr|
+ \bigl|J_{\widehat f|_\tau}(x)-1\bigr|.
\end{equation}

Recall that $\widehat f = \widetilde f \circ \pi_{\M^e}$, where $\widetilde f$ is area-preserving. If $\partial \M = \varnothing$, $\M_h^\ell=\pi_{\M^e}(\Mh) = \M$ and area preservation gives $J_{\widetilde f}(\pi_{\M^e}(x))=1$. If $\partial \M \neq \varnothing$, $\widetilde f$ may not necessarily be area-preserving on the region $\M_h^\ell\setminus\M$, i.e., $J_{\widetilde f}(\pi_{\M^e}(x))\neq 1$ for $\pi_{\M^e}(x)\in\M_h^\ell\setminus\M$.

Still, the boundary extension in Assumption~\ref{ass:triangulation} provides an $\O(h^2)$ bound. We choose $y\in\partial\M$ to be the closest boundary point to $\pi_{\M^e}(x) \in\M_h^\ell\setminus\M$. Since $\widetilde f$ is area-preserving on $\M$, $J_{\widetilde f}(y)=1$ and the mean value theorem gives
\[
\big|J_{\widetilde f}(\pi_{\M^e}(x))-1\big|
= \big|J_{\widetilde f}(\pi_{\M^e}(x))- J_{\widetilde f}(y)\big|
\leq C \big|\pi_{\M^e}(x)-y\big|
\leq Ch^2.
\]

Substituting this into the right term of \eqref{eq:area_preserving_est_1} and using the chain rule and Lemma~\ref{lem:area_change} gives
\begin{align*}
\Big| J_{\widehat f|_\tau}(x) -1\Big|
&= \Big| J_{\widetilde f}\bigl(\pi_{\M^e}|_\tau(x)\bigr) J_{\pi_{\M^e}|_\tau}(x) - 1 \Big| \\
& \leq J_{\pi_{\M^e}|_\tau}(x)\, \Big|J_{\widetilde f}\bigl(\pi_{\M^e}|_\tau(x)\bigr)-1\Big|  +\Big|J_{\pi_{\M^e}|_\tau}(x)-1\Big|
 \leq Ch^2.
\end{align*}
Therefore, combining this with Lemma~\ref{lem:inter_area_change}, the right-hand side of \eqref{eq:area_preserving_est_1} is bounded by
\[
\bigl| J_{f_h|_\tau}-1 \bigr| \leq \bigl| J_{f_h|_\tau}- J_{\widehat f|_\tau}(x) \bigr| +
\bigl| J_{\widehat f|_\tau}(x)-1 \bigr| \leq ch,
\]
for some constant $c >0$.

Then, we estimate the area term in \eqref{eq:energy_gap_decomposition}.
For sufficiently small $h$, $\Nh$ satisfies Assumption~\ref{ass:triangulation} by Lemma~\ref{lem:interpolant_admissible} and $k_h\leq Ch$. Applying estimate \eqref{eq:lifted_domain_estimate} for both $\Mh$ and $\Nh$ gives
\[
\Big| |\M| - |\Mh| \Big| \leq ch^2
\quad\text{and}\quad
\Big||\N| - |\Nh| \Big| \leq Ck_h^2 \leq \widetilde C h^2.
\]
Since $\widetilde f$ is area-preserving, $|\M| = |\N|$ and thus
\begin{equation} \label{eq:total_area_disc}
\Big| |\Nh| - |\Mh| \Big| \leq \Big| |\Nh| - |\N| \Big| + \Big| |\M| - |\Mh| \Big| \leq Ch^2.
\end{equation}

Consequently, \eqref{eq:energy_gap_decomposition} is bounded as follows:
\begin{align*}
\Big| E_\rS(f_h) - |\Mh| \Big|
&\leq \sum_{\tau \in \F} |\tau| \Big|J_{f_h|_\tau}-1 \Big|^2
+ 2 \, \Big| |\Nh| - |\Mh| \Big|\\
&\leq \sum_{\tau \in \F} |\tau|\,(ch)^2 + Ch^2
\leq  C' h^2,
\end{align*}
for some constant $C' > 0$.
\end{proof}

This estimate for the interpolant of an area-preserving diffeomorphism provides an essential tool to establish the following optimal value convergence for piecewise linear (PL) homeomorphisms.

\begin{theorem} \label{thm:minimizer_area}
Let $\M,\N\subset\R^3$ be compact, connected, oriented smooth surfaces that are orientation-preservingly diffeomorphic and satisfy
$|\M|=|\N|$. Let $\{\Mh\}_{h>0}$ satisfy Assumption~\ref{ass:triangulation}, and define
\[
\mathcal A_h =
\left\{ f:\Mh\to\Nh \;\middle|\;
\begin{array}{c}
    f \text{ is an orientation-preserving PL homeomorphism}\\
    \Nh \text{ satisfies Assumption~\ref{ass:triangulation}},\quad f(\V)\subset\N,\quad k_h \leq C_{\A} h
\end{array}
\right\},
\]
where the constants in Assumption~\ref{ass:triangulation} and $C_{\mathcal A}>0$ are large enough and independent of $h$. Then, for each sufficiently small $h$,
\begin{itemize}
\item The discrete infimum converges to the continuous minimum at rate $\O(h^2)$, i.e.,
\begin{equation}
\left| \inf_{f_h \in \mathcal{A}_h} E_\rS(f_h) - \min_{f \in \mathrm{Diff}^+(\M,\N)} \sE_\rS(f)\right| \leq C h^2,
\label{eq:minimizer_area}
\end{equation}
where $\mathrm{Diff}^+(\M,\N)$ is the set of orientation-preserving diffeomorphisms.
\item Every almost minimizer $f_h^*\in\mathcal A_h$ with $E_\rS(f_h^*)\leq \inf_{f\in \mathcal{A}_h}E_\rS(f) + h^2$ satisfies
\[
\big\|J_{f_h^*}-1\big\|_{L^2(\Mh)}
\le Ch,
\]
where $C>0$ is independent of $h$.
\end{itemize}
\end{theorem}

\begin{proof}
We first show the infimum convergence. The Moser--Banyaga theorem ensures the existence of an orientation-preserving, area-preserving diffeomorphism $\widetilde f:\M\to\N$. So, by \eqref{eq:Es_var},
\[
\min_{f \in \mathrm{Diff}^+(\M,\N)} \sE_\rS(f) =  \sE_\rS(\widetilde f) = |\M|.
\]
By Lemmas~\ref{lem:injective_qh_fh} and~\ref{lem:interpolant_admissible}, its interpolant $f_h$ belongs to $\mathcal A_h$. The definition of the infimum, Lemma~\ref{lem:area_preserving_est}, and Lemma~\ref{lem:area_change} give
\begin{equation}
\inf_{f\in \mathcal{A}_h}E_\rS(f) \leq E_\rS(f_h)  \le |\Mh|+Ch^2 \leq |\M| + Ch^2.
\label{eq:minimizer_area_2}
\end{equation}

For the converse estimate, let $g_h \in \mathcal{A}_h$ whose image $g_h(\Mh) = \Nh$ satisfies Assumption~\ref{ass:triangulation} with respect to $\N$ and $k_h\leq C_{\A}h$. Lemma~\ref{lem:area_change} and $|\M| = |\N|$ give
\begin{equation}
\big||\Mh|-|\M|\big|\le Ch^2,
\qquad
\big||\Nh|-|\M|\big|\le Ch^2.
\label{eq:minimizer_area_1}
\end{equation}
Since $g_h$ is an orientation-preserving PL homeomorphism, $\int_\Mh J_{g_h}\, \d A = |\Nh|$ and
\[
0 \leq \|J_{g_h} - 1\|_{L^2(\Mh)}^2 = E_\rS(g_h) - 2|\Nh| + |\Mh|.
\]
Consequently, we obtain
\begin{align*}
E_\rS(g_h) &\geq 2|\Nh| - |\Mh|\\
& \geq 2(|\M| - Ch^2) - (|\M| + Ch^2) \geq |\M| - Ch^2.
\end{align*}
Since $g_h\in\mathcal A_h$ was arbitrary,
\begin{equation}
\inf_{f_h\in\mathcal{A}_h} E_\rS(f_h)  \geq |\M| - Ch^2.
\label{eq:minimizer_area_3}
\end{equation}

Consequently, combining both \eqref{eq:minimizer_area_2} and \eqref{eq:minimizer_area_3} gives
\[
\left | \inf_{f_h\in\mathcal{A}_h} E_\rS(f_h)  - |\M| \right| \leq Ch^2,
\]
which proves \eqref{eq:minimizer_area}, as illustrated in Figure~\ref{fig:sandwich}.

We then prove the estimate for an almost minimizer. We denote the discrete almost minimizer as $f_h^*\in\mathcal A_h$ with image $\mathcal N_h^*:=f_h^*(\Mh)$. The almost minimality and \eqref{eq:minimizer_area} give
\[
E_\rS(f_h^*) \leq
\inf_{f\in \mathcal{A}_h}E_\rS(f) + h^2
\le |\M|+Ch^2.
\]
Hence, \eqref{eq:minimizer_area_1} applied to $g_h=f_h^*$ gives
\begin{align*}
\big \| J_{f_h^*}-1 \big\|_{L^2(\Mh)}^2
&= E_\rS(f_h^*) - 2|\mathcal N_h^*| + |\Mh|\\
&\le
(|\M| + ch^2) - 2(|\M| - Ch^2) + (|\M| + C' h^2) \\
&\leq Ch^2,
\end{align*}
and taking the square root completes the proof.
\end{proof}

Note that the convergence of the optimal value does not imply convergence of the maps themselves, especially since an area-preserving map is not unique.

\begin{figure}[tbp]
\centering
\begin{tikzpicture}[line width=0.5pt,>=latex]
\fill[gray!12] (3.4,-0.12) rectangle (8.6,3.75);
\draw[dashed,gray!55!black] (3.4,-0.12) -- (3.4,3.75);
\draw[dashed,gray!55!black] (8.6,-0.12) -- (8.6,3.75);
\draw[->,line width=0.6pt] (0.3,0) -- (12.4,0) node[anchor=west,font=\small] {energy};
\foreach \x/\l in {3.4/{|\M|-Ch^2},6/{|\M|},8.6/{|\M|+Ch^2}}{
  \draw (\x,0.08) -- (\x,-0.08) node[below,font=\footnotesize] {$\l$};}
\fill (3.9,3.1) circle (1.6pt);
\node[anchor=south west,font=\footnotesize] at (3.85,3.12) {$2|\Nh|-|\Mh|$};
\draw[->,line width=1.0pt] (3.9,3.1) -- (9.0,3.1);
\node[anchor=west,font=\footnotesize,align=left] at (9.1,3.1) {every $g_h\in\mathcal A_h$:\\ $E_\rS(g_h)\geq 2|\Nh|-|\Mh|$};
\fill (7.9,2.0) circle (1.6pt);
\node[anchor=south east,font=\footnotesize] at (7.95,2.02) {$E_\rS(f_h)$};
\draw[->,line width=1.0pt] (7.9,2.0) -- (0.8,2.0);
\node[anchor=west,font=\footnotesize,align=left] at (9.1,2.0) {interpolant of $\widetilde f$:\\ $\inf_{\mathcal A_h}E_\rS\leq E_\rS(f_h)$};
\fill (3.9,0.90) circle (1.6pt);
\fill (8.1,0.90) circle (1.6pt);
\draw[<->,line width=0.6pt] (4.0,0.90) -- (8.0,0.90);
\node[anchor=south,font=\footnotesize] at (6.0,0.93) {$\|J_{f_h^*}-1\|^2_{L^2(\Mh)}$};
\node[anchor=west,font=\footnotesize] at (9.1,0.90) {almost minimizer $f_h^*$};
\end{tikzpicture}
\caption{The discrete infimum lies within a band of width $\O(h^2)$ around $|\M|$; hence the squared $L^2$ area distortion of an almost minimizer (lower arrow) is also $\O(h^2)$.
}
\label{fig:sandwich}
\end{figure}

\section{Discrete authalic flow for simplicial maps}
\label{sec:5}
The previous section established convergence of the optimal value from $E_\rS$ to $\sE_\rS$. Since the global minimizer of $\sE_\rS$ is area-preserving, this ensures that minimizing $E_\rS$ is a valid discrete approximation of the area-preserving variational problem.

This motivates us to develop a practical algorithm to minimize $E_\rS$ to get an approximately area-preserving simplicial map. In Section~\ref{sec:2}, we have developed the intrinsic $L^2$-gradient flow of $\sE_\rS$ to drive a diffeomorphism toward area preservation. In this section, we will derive the simplicial counterpart, called discrete authalic flow, for surfaces with several topologies and establish its global convergence.

\subsection{Discrete authalic flow}
\label{sec:5.1}
Recall that the authalic flow~\eqref{eq:L2_grad_flow} is the $L^2$-gradient flow of the stretch energy $\sE_\rS$, whose velocity is tangent to the target. Analogously, we define the \emph{discrete authalic flow} for simplicial maps as the ambient $L^2$-gradient flow of the discrete stretch energy $E_\rS$ with tangential projection:
\begin{equation} \label{eq:L2_grad_flow_disc}
\partial_t \f_t = - \Pi_{T\N} \big(\nabla_{L^2} E_\rS(\f_t) \big),
\end{equation}
where $\Pi_{T\mathcal{N}}$ denotes the projection onto the tangent space of $\N$. Note that the spatial discretization is based on the discrete energy rather than on a smoothed representation of the piecewise-constant area ratio used to compute the flow~\eqref{eq:L2_grad_flow}, which may cause significant discretization error.

To develop an algorithm based on \eqref{eq:L2_grad_flow_disc}, we require the explicit form of $\nabla_{L^2} E_\rS(\f_t)$, which is an $n$-by-$3$ matrix. The $L^2$-gradient is defined by the $L^2$-inner product
\begin{equation}  \label{eq:L2_grad}
\frac{\d}{\d t} \bigg|_{t = 0} E_\rS(\f + t \y) =
\langle \nabla_{L^2} E_\rS(\f), \, \y \rangle_{L^2(\Mh)}
\end{equation}
for every test simplicial map $\y \in \R^{n \times 3}$. The left-hand side of \eqref{eq:L2_grad} is the directional derivative of $E_\rS$ along $\y$, which can be written as
\begin{align}
\frac{\d}{\d t} \bigg|_{t = 0}  E_\rS(\f + t \y)
 =  \sum_{s = 1}^3 \big(\nabla_{\f^s} E_\rS(\f) \big)^\top \y^s.
\label{eq:L2_grad_1}
\end{align}

For the right-hand side of \eqref{eq:L2_grad}, we denote
\[
\g_{\widetilde M} := \nabla_{L^2} E_\rS(\f)
\]
and expand both $\g_{\widetilde M}$ and $\y$ in the hat basis~\eqref{eq:basis_fun}:
$\g_{\widetilde M} = \sum_{i=1}^n g_{i} \varphi_i$ and $\y = \sum_{j=1}^n y_j \varphi_j$. The $L^2$-inner product then becomes
\begin{align}  \label{eq:L2_grad_2}
\int_\Mh \bigg\langle \sum_{j=1}^n y_j \varphi_j,\,\sum_{i=1}^n g_{i} \varphi_i \bigg\rangle\,\d A
&= \sum_{i,j=1}^n \sum_{s=1}^3
(g_{i}^s  \,y_j^s ) \int_\Mh \varphi_i\,\varphi_j\,\d A
= \sum_{s=1}^3 {\g_{\widetilde M}^s}^\top \widetilde M \,\y^s,
\end{align}
where $\widetilde M$ is the consistent mass matrix \cite[Equation 3.66]{LaBe13}, whose entries are
\begin{subequations} \label{eq:mass}
\begin{equation}
\widetilde M_{i,j} = \int_\Mh \varphi_i\, \varphi_j \,\d A  =
\begin{cases}
	\sum_{\tau\supset [v_i, v_j]} \frac{|\tau|}{12}, & \text{if $i\neq j$ and $[v_i,v_j]\in\E$,} \\
    \sum_{\tau \ni v_i} \frac{|\tau|}{6}, & \text{if $i=j$,} \\
	0, & \text{otherwise}.
\end{cases}
\end{equation}
To simplify the computation, we instead use the lumped mass matrix \cite[Equation 5.74]{LaBe13}
\begin{equation}
M_{i,i} = \sum_{j=1}^n \widetilde{M}_{i,j}, \quad M_{i,j}=0 ~ \text{ for $i\neq j$}.
\end{equation}
\end{subequations}

Since replacing the exact $L^2$-inner product in~\eqref{eq:L2_grad_2} with its mass-lumped approximation is equivalent to replacing $\widetilde M$ with $M$, the lumped discrete $L^2$-gradient, denoted by  $\g_M$, satisfies
\begin{equation*}
\sum_{s=1}^3 {\g_M^s}^\top M\,\y^s =
\sum_{s=1}^3 \big( \nabla_{\f^s}E_{\rS}(\f) \big)^\top \y^s
\end{equation*}
for every $\y \in \R^{n \times 3}$.
Since $M$ is diagonal with positive entries, the lumped discrete $L^2$-gradient $\g_M$ is uniquely determined by
\begin{equation} \label{eq:L2_grad_vec}
\g_M^s = 2\,M^{-1}L_{\rS}(f)\,\f^s, \quad s=1,2,3.
\end{equation}
It provides the flow direction in~\eqref{eq:L2_grad_flow_disc} for the discrete spatial variables after projection onto the tangent space of the target. We then discretize the flow in time by a quasi-implicit Euler method.

\subsection{Quasi-implicit Euler step}
\label{sec:5.2}
Let $\Mh$ be a closed surface, and let $\N \subset \R^3$ denote a canonical smooth target surface, such as the unit sphere for genus-zero surfaces and a torus for genus-one surfaces. We seek a simplicial map $f:\Mh\to\R^3$ whose image $\Nh := f(\Mh)$ is a piecewise linear approximation of $\N$. In practice, we enforce this by requiring $f(v) \in \N$ for every vertex $v \in \V$.

We initialize $f^{(0)}$ as a conformal map whose vertices lie on $\N$ since most conformal methods require no iteration. For genus-zero surfaces, we compute a spherical conformal map by solving for a harmonic map on the extended complex plane and then applying inverse stereographic projection \cite{HaAT00}. For genus-one surfaces, we compute a toroidal conformal map from a holomorphic $1$-form integrated over a periodic fundamental polygon \cite{GuYa02}, where the handle and tunnel loops are computed by the Reeb graph \cite{DeFW13}. The sphere radius is fixed at $1$, whereas the major and minor radii $R$ and $r$ of the torus are first chosen to reduce the initial area distortion and are then kept fixed throughout the iterations.

The direct explicit Euler discretization with time step $\Delta t$ gives
\[
 {\f}^{(k+1)}- {\f}^{(k)}  = - \Delta t \,\Pi_{T \N}\big(M^{-1}L_{\rS}(f^{(k)})\,{\f}^{(k)} \big),
\]
where the factor of $2$ has been absorbed into the time step $\Delta t$. Yet, unlike in the continuous setting, a finite time step $\Delta t$ may cause the updated vertices  to drift off the surface $\N$. We therefore project back onto the surface:
\begin{equation*}
\f^{(k+1)} =  \Pi_{\N} \Big(  \f^{(k)}  - \Delta t \,\Pi_{T \N}\big(M^{-1}L_{\rS}(f^{(k)})\,{ \f}^{(k)} \big)  \Big),
\end{equation*}
where $\Pi_{\N}$ denotes the vertexwise projection onto $\N$. Since $M$ is diagonal with positive entries, the orthogonal projection with respect to the lumped discrete $L^2$-inner product separates vertexwise and therefore coincides with the vertexwise Euclidean projection.

Unfortunately, this explicit Euler formulation often performs poorly in practice. To address this, we use a quasi-implicit Euler step, in which $L_\rS$ is frozen at the current iterate. It computes $\widehat{\f}^{(k+1)}$ from
\[
\Big( M + \Delta t \, L_\rS(f^{(k)}) \Big) \, \widehat{\f}^{s(k+1)} = M \,\f^{s(k)}, \qquad s=1,2,3.
\]
Subtracting $\big(M + \Delta t\, L_\rS(f^{(k)})\big) \f^{s(k)}$ from both sides gives
\[
\widehat{\f}^{s(k+1)} - \f^{s(k)} = -\Delta t\,\Big( M + \Delta t \, L_\rS(f^{(k)}) \Big)^{-1} L_\rS(f^{(k)})\, \f^{s(k)},
\]
i.e., the gradient direction $L_\rS(f^{(k)})\,\f^{(k)}=\frac12\nabla_{\f}E_\rS(f^{(k)})$ is smoothed by the operator $\big(M+\Delta t\,L_\rS(f^{(k)})\big)^{-1}$.

However, this smoothing causes the normal component of the force to leak into the
tangential displacement, as shown in Figure~\ref{fig:qi_projection}(b). In particular, at a critical point of $E_\rS$ restricted to vertices on $\N$, the gradient $\nabla_{\f}E_\rS$ is normal to $\N$ at every vertex, whereas the smoothed displacement generally has a nonzero tangential component.

We, therefore, project the gradient direction onto the tangent spaces before smoothing:
\begin{equation}\label{eq:qi_force}
\g^s := L_\rS(f^{(k)})\,\f^{s(k)},\quad
(\g_{T\N})_i := \g_i-\big(\g_i\cdot\n_i^{(k)}\big)\,\n_i^{(k)}, \quad s = 1,2,3,\ \ i=1,\ldots,n,
\end{equation}
where $\g_i\in\R^3$ denotes the $i$th row of $\g=[\,\g^1\ \g^2\ \g^3\,]$ and $\n_i^{(k)}$ is the unit outward normal to $\N$ at $\f_i^{(k)}$. The smoothing is then applied to the tangential component:
\begin{equation}\label{eq:qi_projected}
\Big( M + \Delta t \,L_\rS(f^{(k)}) \Big)\, \y^s = \g_{T\N}^s, \qquad s = 1,2,3.
\end{equation}
As in the explicit Euler step, we remove the normal component of the resulting displacement,
\begin{equation} \label{eq:tangent}
\p^{(k)}_i = -\Delta t\,\Big(\y_i - \big(\y_i \cdot \n_i^{(k)} \big) \,\n^{(k)}_i\Big), \quad i = 1,\ldots, n,
\end{equation}
and project back onto $\N$:
\begin{equation*}
\f^{(k+1)} = \Pi_{\N} \big(  \f^{(k)} + \p^{(k)}   \big).
\end{equation*}
The resulting procedure is summarized in Algorithm~\ref{alg:flow_close}.

\begin{figure}[tbp]
\centering
\begin{tikzpicture}[>=latex,line width=0.5pt,font=\small]
\begin{scope}[shift={(0.00,0.00)}]
\fill[gray!10] (1.173,-0.598) arc[start angle=36,end angle=144,radius=1.45] -- cycle;
\draw[gray!50,line width=1.2pt] (1.173,-0.598) arc[start angle=36,end angle=144,radius=1.45];
\draw[gray!85,dash pattern=on 2pt off 1.5pt,line width=0.45pt] (-1.381,-0.689) -- (-0.404,0.074);
\draw[gray!85,dash pattern=on 2pt off 1.5pt,line width=0.45pt] (-0.620,0.000) -- (0.620,-0.000);
\draw[gray!85,dash pattern=on 2pt off 1.5pt,line width=0.45pt] (0.404,0.074) -- (1.381,-0.689);
\fill[black] (-0.893,-0.307) circle (1.4pt);
\fill[black] (0.000,0.000) circle (1.4pt);
\fill[black] (0.893,-0.307) circle (1.4pt);
\draw[line width=0.9pt,->] (-0.893,-0.307) -- (-0.834,0.690);
\draw[line width=0.9pt,->] (0.000,0.000) -- (0.330,0.550);
\draw[line width=0.9pt,->] (0.893,-0.307) -- (0.971,0.418);
\end{scope}
\begin{scope}[shift={(3.70,0.00)}]
\fill[gray!10] (1.173,-0.598) arc[start angle=36,end angle=144,radius=1.45] -- cycle;
\draw[gray!50,line width=1.2pt] (1.173,-0.598) arc[start angle=36,end angle=144,radius=1.45];
\draw[gray!85,dash pattern=on 2pt off 1.5pt,line width=0.45pt] (-1.381,-0.689) -- (-0.404,0.074);
\draw[gray!85,dash pattern=on 2pt off 1.5pt,line width=0.45pt] (-0.620,0.000) -- (0.620,-0.000);
\draw[gray!85,dash pattern=on 2pt off 1.5pt,line width=0.45pt] (0.404,0.074) -- (1.381,-0.689);
\fill[black] (-0.893,-0.307) circle (1.4pt);
\fill[black] (0.000,0.000) circle (1.4pt);
\fill[black] (0.893,-0.307) circle (1.4pt);
\draw[line width=1.4pt,->] (-0.893,-0.307) -- (-0.373,0.099);
\draw[line width=1.4pt,->] (0.000,0.000) -- (0.330,-0.000);
\draw[line width=1.4pt,->] (0.893,-0.307) -- (0.589,-0.070);
\end{scope}
\begin{scope}[shift={(7.40,0.00)}]
\fill[gray!10] (1.173,-0.598) arc[start angle=36,end angle=144,radius=1.45] -- cycle;
\draw[gray!50,line width=1.2pt] (1.173,-0.598) arc[start angle=36,end angle=144,radius=1.45];
\draw[gray!85,dash pattern=on 2pt off 1.5pt,line width=0.45pt] (-1.381,-0.689) -- (-0.404,0.074);
\draw[gray!85,dash pattern=on 2pt off 1.5pt,line width=0.45pt] (-0.620,0.000) -- (0.620,-0.000);
\draw[gray!85,dash pattern=on 2pt off 1.5pt,line width=0.45pt] (0.404,0.074) -- (1.381,-0.689);
\fill[gray!60] (-0.893,-0.307) circle (1.4pt);
\fill[black] (0.000,0.000) circle (1.4pt);
\fill[gray!60] (0.893,-0.307) circle (1.4pt);
\draw[line width=0.6pt,->] (0.000,0.000) -- (0.518,0.380);
\draw[dotted,line width=0.9pt] (0.518,0.000) -- (0.518,0.380);
\draw[line width=1.4pt,->] (0.000,0.000) -- (0.518,0.000);
\end{scope}
\begin{scope}[shift={(11.10,0.00)}]
\fill[gray!10] (1.173,-0.598) arc[start angle=36,end angle=144,radius=1.45] -- cycle;
\draw[gray!50,line width=1.2pt] (1.173,-0.598) arc[start angle=36,end angle=144,radius=1.45];
\draw[gray!85,dash pattern=on 2pt off 1.5pt,line width=0.45pt] (-1.381,-0.689) -- (-0.404,0.074);
\draw[gray!85,dash pattern=on 2pt off 1.5pt,line width=0.45pt] (-0.620,0.000) -- (0.620,-0.000);
\draw[gray!85,dash pattern=on 2pt off 1.5pt,line width=0.45pt] (0.404,0.074) -- (1.381,-0.689);
\fill[gray!60] (-0.893,-0.307) circle (1.4pt);
\fill[black] (0.000,0.000) circle (1.4pt);
\fill[gray!60] (0.893,-0.307) circle (1.4pt);
\draw[line width=1.4pt,->] (0.000,0.000) -- (-0.620,0.000);
\draw[line width=0.5pt,fill=white] (-0.620,0.000) circle (1.6pt);
\fill (-0.570,-0.117) circle (1.6pt);
\end{scope}
\begin{scope}[shift={(0.00,-2.55)}]
\fill[gray!10] (1.173,-0.598) arc[start angle=36,end angle=144,radius=1.45] -- cycle;
\draw[gray!50,line width=1.2pt] (1.173,-0.598) arc[start angle=36,end angle=144,radius=1.45];
\draw[gray!85,dash pattern=on 2pt off 1.5pt,line width=0.45pt] (-1.381,-0.689) -- (-0.404,0.074);
\draw[gray!85,dash pattern=on 2pt off 1.5pt,line width=0.45pt] (-0.620,0.000) -- (0.620,-0.000);
\draw[gray!85,dash pattern=on 2pt off 1.5pt,line width=0.45pt] (0.404,0.074) -- (1.381,-0.689);
\fill[black] (-0.893,-0.307) circle (1.4pt);
\fill[black] (0.000,0.000) circle (1.4pt);
\fill[black] (0.893,-0.307) circle (1.4pt);
\draw[line width=0.9pt,->] (-0.893,-0.307) -- (-1.478,0.441);
\draw[line width=0.9pt,->] (0.000,0.000) -- (0.000,0.600);
\draw[line width=0.9pt,->] (0.893,-0.307) -- (1.077,-0.071);
\end{scope}
\begin{scope}[shift={(7.40,-2.55)}]
\fill[gray!10] (1.173,-0.598) arc[start angle=36,end angle=144,radius=1.45] -- cycle;
\draw[gray!50,line width=1.2pt] (1.173,-0.598) arc[start angle=36,end angle=144,radius=1.45];
\draw[gray!85,dash pattern=on 2pt off 1.5pt,line width=0.45pt] (-1.381,-0.689) -- (-0.404,0.074);
\draw[gray!85,dash pattern=on 2pt off 1.5pt,line width=0.45pt] (-0.620,0.000) -- (0.620,-0.000);
\draw[gray!85,dash pattern=on 2pt off 1.5pt,line width=0.45pt] (0.404,0.074) -- (1.381,-0.689);
\fill[gray!60] (-0.893,-0.307) circle (1.4pt);
\fill[black] (0.000,0.000) circle (1.4pt);
\fill[gray!60] (0.893,-0.307) circle (1.4pt);
\draw[line width=0.6pt,->] (0.000,0.000) -- (-0.390,0.710);
\draw[dotted,line width=0.9pt] (-0.390,0.000) -- (-0.390,0.710);
\draw[line width=1.4pt,->] (0.000,0.000) -- (-0.390,0.000);
\end{scope}
\begin{scope}[shift={(11.10,-2.55)}]
\fill[gray!10] (1.173,-0.598) arc[start angle=36,end angle=144,radius=1.45] -- cycle;
\draw[gray!50,line width=1.2pt] (1.173,-0.598) arc[start angle=36,end angle=144,radius=1.45];
\draw[gray!85,dash pattern=on 2pt off 1.5pt,line width=0.45pt] (-1.381,-0.689) -- (-0.404,0.074);
\draw[gray!85,dash pattern=on 2pt off 1.5pt,line width=0.45pt] (-0.620,0.000) -- (0.620,-0.000);
\draw[gray!85,dash pattern=on 2pt off 1.5pt,line width=0.45pt] (0.404,0.074) -- (1.381,-0.689);
\fill[gray!60] (-0.893,-0.307) circle (1.4pt);
\fill[black] (0.000,0.000) circle (1.4pt);
\fill[gray!60] (0.893,-0.307) circle (1.4pt);
\draw[line width=1.4pt,->] (0.000,0.000) -- (0.507,0.000);
\draw[line width=0.5pt,fill=white] (0.507,0.000) circle (1.6pt);
\fill (0.479,-0.081) circle (1.6pt);
\end{scope}
\node[anchor=north] at (0.00,-0.72) {$\g$};
\node[anchor=north] at (3.70,-0.72) {$\g_{T\N}$};
\node[anchor=north] at (7.40,-0.72) {$\y$};
\node[anchor=north] at (11.10,-0.72) {$\Pi_\N(\f+\p)$};
\node[anchor=north] at (0.00,-3.27) {$\g$ with $\g_{T\N}=\0$};
\node[anchor=north] at (7.40,-3.27) {$\y$};
\node[anchor=north] at (11.10,-3.27) {$\p\neq\0$};
\node[font=\footnotesize,anchor=north] at (-0.893,-0.357) {$f_{i-1}$};
\node[font=\footnotesize,anchor=north] at (0.000,-0.050) {$f_i$};
\node[font=\footnotesize,anchor=north] at (0.893,-0.357) {$f_{i+1}$};
\draw[->,line width=0.6pt] (1.42,-0.22) -- node[above,font=\footnotesize] {\eqref{eq:qi_force}} (2.28,-0.22);
\draw[->,line width=0.6pt] (5.12,-0.22) -- node[above,font=\footnotesize] {\eqref{eq:qi_projected}} (5.98,-0.22);
\draw[->,line width=0.6pt] (8.82,-0.22) -- node[above,font=\footnotesize] {\eqref{eq:tangent}} (9.68,-0.22);
\draw[->,line width=0.6pt,dashed] (1.42,-2.77) -- node[above,font=\footnotesize] {\eqref{eq:qi_projected} without \eqref{eq:qi_force}} (5.98,-2.77);
\draw[->,line width=0.6pt] (8.82,-2.77) -- node[above,font=\footnotesize] {\eqref{eq:tangent}} (9.68,-2.77);
\node[anchor=west] at (-1.55,1.05) {(a)};
\node[anchor=west] at (-1.55,-1.50) {(b)};
\end{tikzpicture}
\caption{Tangential projections before and after the smoothing (dashed: tangent lines of $\N$). (a) The smoothed $\y$ gets a normal component (dotted), removed by~\eqref{eq:tangent}. (b) Without~\eqref{eq:qi_force}, a purely normal component produces $\p\neq\0$.}
\label{fig:qi_projection}
\end{figure}

For genus-zero surfaces, we take $\N=\bS^2$, for which the projection is
\begin{equation}\label{eq:Proj_sphere}
\Pi_{\bS^2}(\f)=\bigg(\frac{f_1}{\|f_1\|_2},\ldots,\frac{f_n}{\|f_n\|_2}\bigg)^\top.
\end{equation}
For genus-one surfaces, we take $\N=\bT^2$, a torus with major radius $R$ and minor radius $r$, for which the projection is
\[
\Pi_{\bT^2}(\f)=\bigg(
c_1 + r\,\frac{f_1-c_1}{\|f_1-c_1\|_2},\, \ldots,\,
c_n + r\,\frac{f_n-c_n}{\|f_n-c_n\|_2}
\bigg)^\top,
\quad
c_i=(R\cos\theta_i,\, R\sin\theta_i,\, 0),
\]
where $\theta_i = \mathrm{atan2}(f_i^2, f_i^1)$ for $i = 1,\ldots,n$.

\begin{algorithm}[t]
\caption{Discrete authalic flow for closed surfaces}
\label{alg:flow_close}
\begin{algorithmic}[1]
\Require A closed triangular mesh $\Mh$ and a smooth target surface $\N$.
\Ensure An approximately area-preserving simplicial map $f:\Mh\to\Nh$.
\State Initialize $\f$ by a conformal map whose vertices lie on $\N$.
\State Assemble the mass matrix $M$ via \eqref{eq:mass}.
\While{not converged}
    \State Assemble $L \gets L_{\rS}(\f)$ via \eqref{eq:Ls}.
    \State Set $\g\gets L\f$ and project it to $T\N$ via \eqref{eq:qi_force}.
    \State Choose a time step $\Delta t > 0$.
    \State Solve for $s=1,2,3$:
    \[
    (M + \Delta t\,L)\,\y^s  = \g^s_{T\N}.
    \]
    \State Set the displacement $\p\gets-\Delta t\,\y$ and project it to $T \N$ via \eqref{eq:tangent}.
    \State Project $\f^{\text{new}} \gets \Pi_{\N}( \f + \p)$.
    \State Update $\f \gets {\f}^{\text{new}}$.
\EndWhile
\State \Return $\f$.
\end{algorithmic}
\end{algorithm}

\subsection{Global convergence}
\label{sec:5.3}
The above temporal discretization is consistent with the discrete authalic flow~\eqref{eq:L2_grad_flow_disc}: as $\Delta t \to 0$,
\[
\frac{1}{\Delta t}\, \p^{(k)} = -\Pi_{T\N} \,\y\;\longrightarrow\; -M^{-1}\g_{T\N}
= -\tfrac{1}{2}\, \Pi_{T\N}\g_M.
\]
In fact, with a sufficiently small fixed time step and from any initialization, every accumulation point of the iterates is a critical point.

Note that the iterates are not guaranteed to stay orientation-preserving, but the stretch Laplacian remains well defined at a degenerate image triangle, even if the cotangents in~\eqref{eq:Ld} are undefined. Specifically, for $\tau = [v_i, v_j, v_k]$, the $(i,j)$-component of the stretch Laplacian is written as
\begin{align*}
\big[L_\rS(f|_\tau)\big]_{i,j} &=-\frac{|f(\tau)|}{2|\tau|}\cot\theta_{ij}^k(f) \\
&= -\frac{|f(\tau)|}{2|\tau|}\frac{\| f_i - f_k\| \, \| f_j - f_k\| \cos \theta_{ij}^k(f)}{\| f_i - f_k\| \, \| f_j - f_k\| \sin\theta_{ij}^k(f)}
= -\frac{|f(\tau)|}{2|\tau|} \frac{(f_i - f_k) \cdot (f_j - f_k)}{ 2 |f(\tau)|}\\
&= -\frac{(f_i - f_k) \cdot (f_j - f_k)}{4\,|\tau|},
\end{align*}
which is a quadratic polynomial in $\f$. Hence $L_\rS(f)$ and $\nabla E_\rS$ are well defined for degenerate or folded image triangles.

Since a simplicial map $f$ is represented as a matrix $\f\in \R^{n \times 3}$, we consider the inner product
\[
\langle \u,\v\rangle=\sum_{s=1}^3{\u^s}^\top\v^s=\sum_{i=1}^n\u_i\cdot\v_i,
\qquad \|\u\|=\langle\u,\u\rangle^{1/2}.
\]
We first show that $\p^{(k)}$ is a descent direction for every $\Delta t>0$.

\begin{lemma}\label{lem:qi_descent}
Let $f$ be a simplicial map with $f_i\in\N$ for all $i$, and let $\g_{T\N}$, $\y$, and $\p$ be defined by \eqref{eq:qi_force}--\eqref{eq:tangent} with $\Delta t>0$. Then
\begin{equation}\label{eq:qi_descent}
\Big\langle \nabla_{\f}E_\rS(f), \, \p \Big\rangle
= -2\Delta t\Big\langle \g_{T\N}, \, \Big(M + \Delta t L_\rS(f) \Big)^{-1} \g_{T\N} \Big\rangle
\;\le\;0,
\end{equation}
with equality if and only if $\g_{T\N}=\0$, i.e., if and only if $\Pi_{T\N}\g_M=\0$.
\end{lemma}

\begin{proof}
The vertexwise tangential projection gives
\begin{align*}
\langle\g, \, \Pi_{T\N}\y \rangle &=\sum_{i=1}^n\g_i\cdot\big(\y_i-(\y_i\cdot\n_i)\n_i\big) =\sum_{i=1}^n\big(\g_i-(\g_i\cdot\n_i)\n_i\big)\cdot\y_i
=\langle\g_{T\N} ,\,\y\rangle .
\end{align*}
Since $\nabla_{\f}E_\rS(f)=2\g$ by \eqref{eq:Es_grad}, we obtain
\begin{align*}
\Big\langle \nabla_{\f}E_\rS(f), \,\p \Big\rangle
&=\Big\langle 2\g, \,-\Delta t\, \Pi_{T\N}\y \Big\rangle
=-2\Delta t\,\langle \g_{T\N},\,\y\rangle\\
&=-2\Delta t\Big\langle \g_{T\N}, \, \Big(M + \Delta t L_\rS(f) \Big)^{-1} \g_{T\N} \Big\rangle.
\end{align*}
Since the matrix $M$ is diagonal with positive entries and $L_\rS(f)$ is positive semidefinite, $M+\Delta t\,L_\rS(f)$ is symmetric positive definite; thus,
\[
-2\Delta t\Big\langle \g_{T\N}, \, \Big(M + \Delta t L_\rS(f) \Big)^{-1} \g_{T\N} \Big\rangle \leq 0
\]
with equality if and only if $\g_{T\N}=\0$. Since $M$ is diagonal, $\Pi_{T\N}\g_M=2M^{-1}\g_{T\N}$, so this is equivalent to $\Pi_{T\N}\g_M=\0$.
\end{proof}

As a result, we may view the quasi-implicit Euler direction $\p$ as the negative tangential gradient $-\g_{T\N}$ preconditioned by a symmetric positive definite matrix. This property lets it behave like a gradient descent method, which yields the following result.

\begin{lemma}\label{lem:qi_uniform}
There exists $\Delta t_* \in (0, 1]$ such that, for every simplicial map $f$ with $f_i\in\N$ for all $i$ and every $0<\Delta t\le\Delta t_*$, the update defined by \eqref{eq:qi_force}--\eqref{eq:tangent} satisfies
\begin{equation}\label{eq:qi_uniform}
E_\rS(\f)-E_\rS\big(\Pi_\N(\f+\p)\big) \geq \frac{\Delta t}{\kappa}\,\|\g_{T\N}\|^2,
\end{equation}
where
\[
\kappa := \lambda_{\max} (M) + \max_{f_i\in\N}\lambda_{\max}(L_\rS(f))
\]
and $\lambda_{\max}(\cdot)$ denotes the maximum eigenvalue.
\end{lemma}

\begin{proof}
Since $L_\rS(f)$ is a polynomial in $\f$ and the set of simplicial maps with $f_i\in\N$ is compact, both $\kappa$ and $\max_{f_i\in\N}\|\g_{T\N}\|$ are finite. Since $L_\rS(f)$ is symmetric positive semidefinite and $0 < \Delta t \leq 1$, every eigenvalue $\lambda$ of $M+\Delta t\,L_\rS(f)$ satisfies
\[
\lambda_{\min}(M) \leq \lambda(M+\Delta t\,L_\rS(f)) \leq \kappa .
\]
Hence, by the definition of the direction $\p$,
\begin{align}
\|\p\|^2 &= \Delta t^2\, \| \Pi_{T\N} \,\y\|^2 \leq \Delta t^2\,\|\y\|^2  \nonumber\\
&= \Delta t^2 \, \big\langle \g_{T\N},\, (M+\Delta t\,L_\rS(f))^{-2} \, \g_{T\N} \big\rangle
\leq \frac{\Delta t^2}{\lambda_{\min}(M)^2} \,\|  \g_{T\N}\|^2 . \label{eq:qi_uniform_3}
\end{align}
Thus, for $\Delta t$ small enough, the points $\f+t\p$, $t\in[0,1]$, stay in a fixed compact subset of the tubular neighborhood of $\N$ on which $\Pi_\N$ is smooth.

We define $F(\p) := E_\rS\big(\Pi_\N(\f+\p)\big)$. Taylor's theorem gives
\begin{equation}\label{eq:qi_uniform_1}
F(\p) \leq F(\0) + \langle\nabla F(\0),\, \p\rangle + C\|\p\|^2,
\end{equation}
where the constant $C>0$ is independent of $f$ by compactness. Since $\Pi_\N(\f) = \f$, the first term of right-hand side of \eqref{eq:qi_uniform_1} is
\[
 F(\0)  = E_\rS\big(\Pi_\N(\f)\big) = E_\rS(\f).
\]
Also, since $\mathrm{D}\Pi_\N(\f)=\Pi_{T\N}$ and $\p$ is tangent, by the chain rule, the second term of right-hand side of \eqref{eq:qi_uniform_1} is
\[
\langle\nabla F(\0),\, \p\rangle = \Big\langle\nabla E_\rS \big(\Pi_\N(\f)\big),\, \,\mathrm{D} \Pi_\N(\f)  \p \Big\rangle
= \Big\langle \nabla E_\rS (\f), \, \Pi_{T\N} \p \Big\rangle=  \langle\nabla E_\rS (\f),\, \p\rangle.
\]
Hence, substituting into \eqref{eq:qi_uniform_1}, we have
\begin{equation}
E_\rS\big(\Pi_\N(\f+\p)\big)
\leq E_\rS(\f) + \langle\nabla E_\rS (\f),\, \p\rangle + C\|\p\|^2.
\label{eq:qi_uniform_2}
\end{equation}
In addition, Lemma~\ref{lem:qi_descent} gives the bound
\begin{equation}
\langle\nabla E_\rS (\f),\, \p\rangle =
-2\Delta t\Big\langle \g_{T\N}, \, \Big(M + \Delta t L_\rS(f) \Big)^{-1} \g_{T\N} \Big\rangle \leq
 -\frac{2\Delta t}{ \kappa} \, \|  \g_{T\N}\|^2.
\label{eq:qi_uniform_4}
\end{equation}

Consequently, substituting \eqref{eq:qi_uniform_3} and \eqref{eq:qi_uniform_4} into \eqref{eq:qi_uniform_2} gives
\begin{align*}
E_\rS\big(\Pi_\N(\f+\p)\big)
&\leq E_\rS(\f) -  \bigg(\frac{2\Delta t}{\kappa} - \frac{C\,\Delta t^2}{\lambda_{\min}(M)^2}  \bigg)\|  \g_{T\N}\|^2\\
&= E_\rS(\f) - \frac{\Delta t}{\kappa} \bigg( 2 - \frac{C\,\kappa\, \Delta t}{\lambda_{\min}(M)^2} \bigg)  \|\g_{T\N}\|^2.
\end{align*}
Choosing $\Delta t_*\le\lambda_{\min}(M)^2/(C\kappa)$, we conclude
\[
E_\rS\big(\Pi_\N(\f+\p)\big)  \leq E_\rS(\f) - \frac{\Delta t}{\kappa}  \|\g_{T\N}\|^2
\]
for all $0<\Delta t\le\Delta t_*$.
\end{proof}

Since the squared norm of the tangential gradient is bounded by the energy decrease in each iteration, we obtain the following global convergence result.

\begin{theorem}\label{thm:qi_convergence}
Let $\{f^{(k)}\}$ be generated by Algorithm~\ref{alg:flow_close} with a fixed time step $\Delta t\in(0,\Delta t_*]$, where $\Delta t_*$ and $\kappa$ are given by Lemma~\ref{lem:qi_uniform}. Then,
\begin{enumerate}[label=\textup{(\roman*)}]
\item The sequence $\{E_\rS(f^{(k)})\}$ is nonincreasing and converges.
\item The tangential gradient satisfies
\begin{equation} \label{eq:qi_convergence}
\sum_{k=0}^\infty\big\|\g_{T\N}^{(k)}\big\|^2\le\frac{\kappa\,E_\rS(f^{(0)})}{\Delta t},
\qquad
\min_{0\le k<K}\big\|\g^{(k)}_{T\N}\big\|\le\bigg(\frac{\kappa\,E_\rS(f^{(0)})}{\Delta t\,K}\bigg)^{1/2}.
\end{equation}
In particular, $\Pi_{T\N}\g_M^{(k)}\to\0$, and every accumulation point of $\{f^{(k)}\}$ is a critical point of $E_\rS$ restricted to simplicial maps with $f_i \in \N$.
\end{enumerate}
\end{theorem}

\begin{proof}
Lemma~\ref{lem:qi_uniform} gives
\begin{equation} \label{eq:qi_convergence_1}
E_\rS(\f^{(k)}) - E_\rS(\f^{(k+1)}) \geq \frac{\Delta t}{\kappa} \| \g^{(k)}_{T\N}\|^2.
\end{equation}
Therefore, $\{E_\rS(f^{(k)})\}$ is nonincreasing and converges to some $E^*_\rS$.

In addition, summing \eqref{eq:qi_convergence_1} over $0 \leq k \leq K-1$ gives
\begin{align*}
\frac{\Delta t}{\kappa} \sum_{k = 0}^{K-1} \| \g^{(k)}_{T\N}\|^2
&\leq\sum_{k = 0}^{K-1} \Big(E_\rS(\f^{(k)}) - E_\rS(\f^{(k+1)})\Big)
= E_\rS(\f^{(0)}) - E_\rS(\f^{(K)}) \leq E_\rS(f^{(0)}) .
\end{align*}
Consequently, we have
\[
\frac{\Delta t \, K }{\kappa} \min_{0 \leq k < K}\| \g^{(k)}_{T\N}\|^2  \leq
\frac{\Delta t}{\kappa} \sum_{k = 0}^{K-1} \| \g^{(k)}_{T\N}\|^2  \leq E_\rS(f^{(0)}),
\]
which shows the right inequality of \eqref{eq:qi_convergence}. Moreover, taking $K \to \infty$ gives the left inequality.

Finally, since $M$ is diagonal, $\Pi_{T\N} \g_M^{(k)} = 2 M^{-1} \g_{T\N}^{(k)} \to \0$. Also, the set of simplicial maps with $f_i\in\N$ is compact and $\g_{T\N}$ depends continuously on $f$, so every accumulation point $f^*$ satisfies $\g_{T\N}(f^*)=\0$.
\end{proof}

In practice, we select the time step $\Delta t$ by a line search using MATLAB's \texttt{fminbnd} during the first ten iterations and keep the selected value fixed thereafter. If the energy does not decrease, $\Delta t$ is halved at most twenty times.

\subsection{Extension to surfaces with boundary}
\label{sec:5.4}
We then extend the algorithm to a genus-zero open surface $\Mh$. If $\Mh$ has more than one boundary component, we cap every component except a selected outer boundary by adding a center vertex and joining it to all the vertices of that component. So the resulting mesh is simply connected, and we continue to denote it by $\Mh$ and find a simplicial map $f:\Mh\to\mathbb{R}^2$ that parameterizes $\Mh$ over the unit disk $\mathbb{D}$. Since $f$ is piecewise linear, its image $\Nh:=f(\Mh)$ is a polygonal approximation of $\mathbb{D}$.

We denote the boundary and interior index sets by
\begin{equation} \label{eq:BI}
\B = \{ i \mid v_i \in \partial\Mh \},
\qquad
\I = \{ i \mid v_i \notin \partial\Mh \}.
\end{equation}
To keep the boundary vertices on the unit circle, we let them flow tangentially along the circle while allowing the interior vertices to adapt without an explicit disk constraint. Unlike the fixed-boundary variations used in the continuous characterization, this boundary update is a practical discrete relaxation that further reduces the residual area distortion.

We follow the closed-surface procedure, but project only the boundary vertices. After initializing $f^{(0)}$ as a harmonic map with arc-length boundary parameterization, we set $\g^s:=L_\rS(f^{(k)})\,\f^{s(k)}$, $s=1,2$, keep the interior force, and project the boundary force onto the tangent lines of $\bS^1$:
\begin{equation} \label{eq:tangent_circle}
(\g_{T\bS^1})_i=\g_i,\quad i\in\I,\qquad
(\g_{T\bS^1})_i=\g_i-\big(\g_i\cdot\n_i^{(k)}\big)\,\n_i^{(k)},\quad i\in\B,
\end{equation}
where $\n_i^{(k)}=\f_i^{(k)}/\|\f_i^{(k)}\|_2$ is the unit outward normal to $\bS^1$ at $\f_i^{(k)}$. We then solve
\begin{equation*}
\Big( M + \Delta t \,L_\rS(f^{(k)}) \Big)\,\y^s=\g_{T\bS^1}^s, \qquad s = 1,2,
\end{equation*}
set
\[
\p^{(k)}_i=-\Delta t\,\y_i,\quad i\in\I, \qquad \p^{(k)}_i = -\Delta t\,\big(\y_i - (\y_i \cdot \n_i^{(k)})\,\n^{(k)}_i\big),\quad i\in\B,
\]
and update
\begin{equation*}
\widetilde \f_\I=\f^{(k)}_\I+\p^{(k)}_\I,\qquad
\widetilde\f_\B= \Pi_{\bS^1} \!\Big ({\f}^{(k)}_\B + \p^{(k)}_\B \Big),
\end{equation*}
where $\Pi_{\bS^1}$ is the two-dimensional analogue of \eqref{eq:Proj_sphere}.

Since boundary vertices are typically far fewer than interior vertices, we then update the interior more aggressively with the boundary fixed. Namely, we minimize the frozen quadratic form $\frac12\sum_s\f^{s\top}L_\rS(\widetilde f)\,\f^s$ over the interior vertices, which amounts to solving
\begin{equation*}
[L_\rS(\widetilde f)]_{\I,\I} \,\widehat\f^s_\I
= - [L_\rS(\widetilde f)]_{\I,\B} \,\widetilde\f^s_\B,
\quad s=1,2,
\end{equation*}
and set
\begin{equation}  \label{eq:open_update}
(\f^{(k+1)}_\I,\, \f_\B^{(k+1)}) =
\begin{cases}
    (\widehat{\f}_\I, \, \widetilde \f_\B) & \text{if } E_\rS(\widehat{\f}_\I, \, \widetilde \f_\B) \leq E_\rS(\widetilde{\f}_\I, \, \widetilde \f_\B) \\
    (\widetilde{\f}_\I, \, \widetilde \f_\B) & \text{otherwise}.
\end{cases}
\end{equation}
This interior update coincides with the fixed-point update of stretch energy minimization~\cite{YuLW19}. The overall procedure is summarized in Algorithm~\ref{alg:flow_open}. As in the closed case, the time step $\Delta t$ is selected by a line search, here during the first five iterations, and then kept fixed thereafter.

\begin{algorithm}[t]
\caption{Discrete authalic flow for open surfaces}
\label{alg:flow_open}
\begin{algorithmic}[1]
\Require A genus-zero open triangular mesh $\Mh$, with all boundary components except the outer one capped.
\Ensure An approximately area-preserving simplicial map $f:\Mh\to\Nh$.
\State Set the index sets $\B$ and $\I$ as in \eqref{eq:BI}.
\State Initialize $\f$ by a harmonic map with arc-length boundary parameterization.
\State Assemble the mass matrix $M$ via \eqref{eq:mass}.
\While{not converged}
    \State Assemble $L \gets L_{\rS}(\f)$ via \eqref{eq:Ls}, set $\g\gets L\f$, and project $\g_\B$ to $T\bS^1$ via \eqref{eq:tangent_circle}.
    \State Choose a time step $\Delta t>0$.
    \State Solve $$(M +\Delta t\,L)\,\y^s = \g_{T\bS^1}^s \text{  for } s=1,2.$$
    \State Set $\p \gets -\Delta t\,\y$ and project $\p_\B$ to $T\bS^1$.
    \State Set $\widetilde \f_\I\gets\f_\I+\p_\I$ and $\widetilde \f_\B\gets\Pi_{\bS^1}(\f_\B + \p_\B)$.
    \State Assemble $L \gets L_{\rS}(\widetilde \f)$.
    \State Solve $$L_{\I,\I}\,\widehat\f_\I^s = -L_{\I,\B}\,\widetilde\f_\B^s  \text{  for } s=1,2.$$
    \State Update $\f$ by \eqref{eq:open_update}
\EndWhile
\State \Return $\f$.
\end{algorithmic}
\end{algorithm}

Although Algorithm~\ref{alg:flow_open} does not explicitly constrain the interior vertices, the boundary vertices remain on the unit circle, and in practice the interior vertices stay inside the disk. It is therefore natural to assume that, for some $R>0$, the iterates remain in the compact set
\[
\mathcal K_R:=\big\{f \;\big|\; f_i\in\bS^1 \text{ for } i\in\B,\quad \|f_i\|_2\le R \text{ for } i\in\I\big\}.
\]
Under this assumption, the convergence result for closed surfaces extends directly to the open-surface case.

\begin{corollary}\label{cor:open_convergence}
For every $R>0$, there exist $\Delta t_R\in(0,1]$ and $\kappa_R>0$, depending only on $R$ and $\Mh$, such that if Algorithm~\ref{alg:flow_open} is run with a fixed time step $0<\Delta t\le\Delta t_R$ and its iterates satisfy $f^{(k)}\in\mathcal K_R$ for all $k$, then $\{E_\rS(f^{(k)})\}$ is nonincreasing and
\[
\sum_{k=0}^{\infty}\big\|\g^{(k)}_{T\bS^1}\big\|^2\le\frac{\kappa_R\,E_\rS(f^{(0)})}{\Delta t}.
\]
Consequently, $\Pi_{T\bS^1}\g_M^{(k)}\to\0$, and every accumulation point is a critical point of $E_\rS$ restricted to simplicial maps with $f_i\in\bS^1$ for $i\in\B$.
\end{corollary}

\begin{proof}
Lemma~\ref{lem:qi_descent} applies directly with $\Pi_{T\N}$ replaced by $\Pi_{T\bS^1}$, so $\p^{(k)}$ is a descent direction.
Moreover, since $\mathcal K_R$ is compact,
\[
\kappa_R:=\lambda_{\max}(M)+\max_{f\in\mathcal K_R}\lambda_{\max}\big(L_\rS(f)\big)
\qquad\text{and}\qquad
\max_{f\in\mathcal K_R}\|\g_{T\bS^1}\|
\]
are finite. Hence the proof of Lemma~\ref{lem:qi_uniform} applies with $\Pi_\N$ replaced by $\f\mapsto(\f_\I,\Pi_{\bS^1}(\f_\B))$, whose differential at $f\in\mathcal K_R$ is $\Pi_{T\bS^1}$. This gives $\Delta t_R\in(0,1]$ such that
\[
E_\rS(f^{(k)})-E_\rS(\widetilde f)\ge\frac{\Delta t}{\kappa_R}\,\big\|\g^{(k)}_{T\bS^1}\big\|^2,
\qquad 0<\Delta t\le\Delta t_R,
\]
whenever $f^{(k)}\in\mathcal K_R$.

Since $E_\rS(f^{(k+1)})\le E_\rS(\widetilde f)$ by~\eqref{eq:open_update}, the same bound holds for $E_\rS(f^{(k)})-E_\rS(f^{(k+1)})$, and the conclusion follows as in the proof of Theorem~\ref{thm:qi_convergence}.
\end{proof}

\subsection{Numerical experiments of the proposed method}
\label{sec:5.5}
In this section, we evaluate the effectiveness of the proposed method by numerical experiments. All computations were carried out in MATLAB R2024b on a laptop equipped with an AMD Ryzen 9 5900HS processor and 32 GB of RAM. We test the method on eight benchmark meshes: four open surfaces (Lion, Bunny, Beetle, and Face) and four closed surfaces (Bulldog, Gargoyle, Kitten, and Rocker Arm).

After computing each parameterization, we normalize the total area so that $|\Mh| = |\Nh|$. We measure area distortion by the unweighted variance of the triangle-wise area ratios defined in \eqref{eq:f_tau}:
\begin{equation} \label{eq:res_var}
e_{\mathrm{var}}(f) = \operatornamewithlimits{Var}_{\tau \in \F}\bigg( \frac{|f(\tau)|}{|\tau|} \bigg),
\end{equation}
and we also report the normalized stretch-energy residual
\begin{equation} \label{eq:res_en}
e_\mathrm{en}(f) = \frac{E_\rS(f)}{|\Mh|}  - 1,
\end{equation}
which equals the area-weighted variance of the area ratios by \eqref{eq:Es_var_discrete}.

Figure~\ref{fig:MeshModel_open} displays the disk parameterizations and corresponding histograms of area ratios for the open meshes, including both simply connected and multiply connected examples. The resulting planar maps preserve local areas well, and the area ratios are tightly concentrated around $1$, indicating excellent area preservation.

The method performs similarly well on closed surfaces. Figure~\ref{fig:MeshModel_close} shows spherical and toroidal parameterizations together with the corresponding area-ratio histograms. For the toroidal surfaces, the handle and tunnel loops are computed by the Reeb graph method~\cite{DeFW13}, which depends on a random seed. We therefore perform five runs and report the one with the lowest energy residual $e_\mathrm{en}$. As with the open surfaces, both the spherical and toroidal parameterizations have area ratios concentrated around $1$.

Table~\ref{tab:flow} reports runtime, energy residual \eqref{eq:res_en}, variance \eqref{eq:res_var}, number of folded triangles, and number of iterations. Across all examples, both $e_{\mathrm{var}}$ and $e_{\mathrm{en}}$ lie between about $10^{-4}$ and $3\times 10^{-3}$, while the runtime stays below $30$ seconds. Moreover, none of the computed parameterizations contains folded triangles.

Overall, discrete authalic flow handles surfaces of various topologies and efficiently produces high-quality area-preserving parameterizations.

\begin{figure}[t]
\centering
\resizebox{\textwidth}{!}{
\begin{tabular}{ccccccc}
\specialrule{.2em}{.1em}{.1em}
\multicolumn{3}{c}{Lion} && \multicolumn{3}{c}{Bunny} \\
\multicolumn{3}{c}{$\#\F = 34{,}421$  $\#\V = 17{,}334$} && \multicolumn{3}{c}{$\#\F = 62{,}946$  $\#\V = 31{,}593$} \\
\includegraphics[height=3cm]{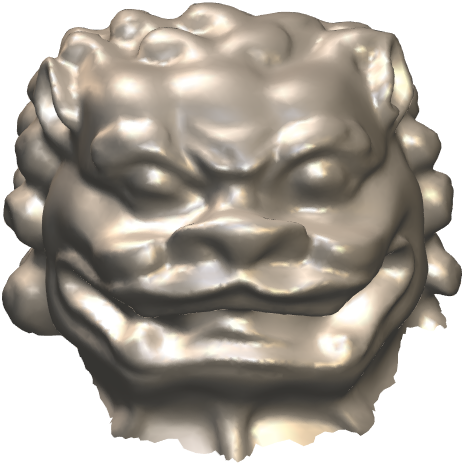} &
\includegraphics[height=2.8cm]{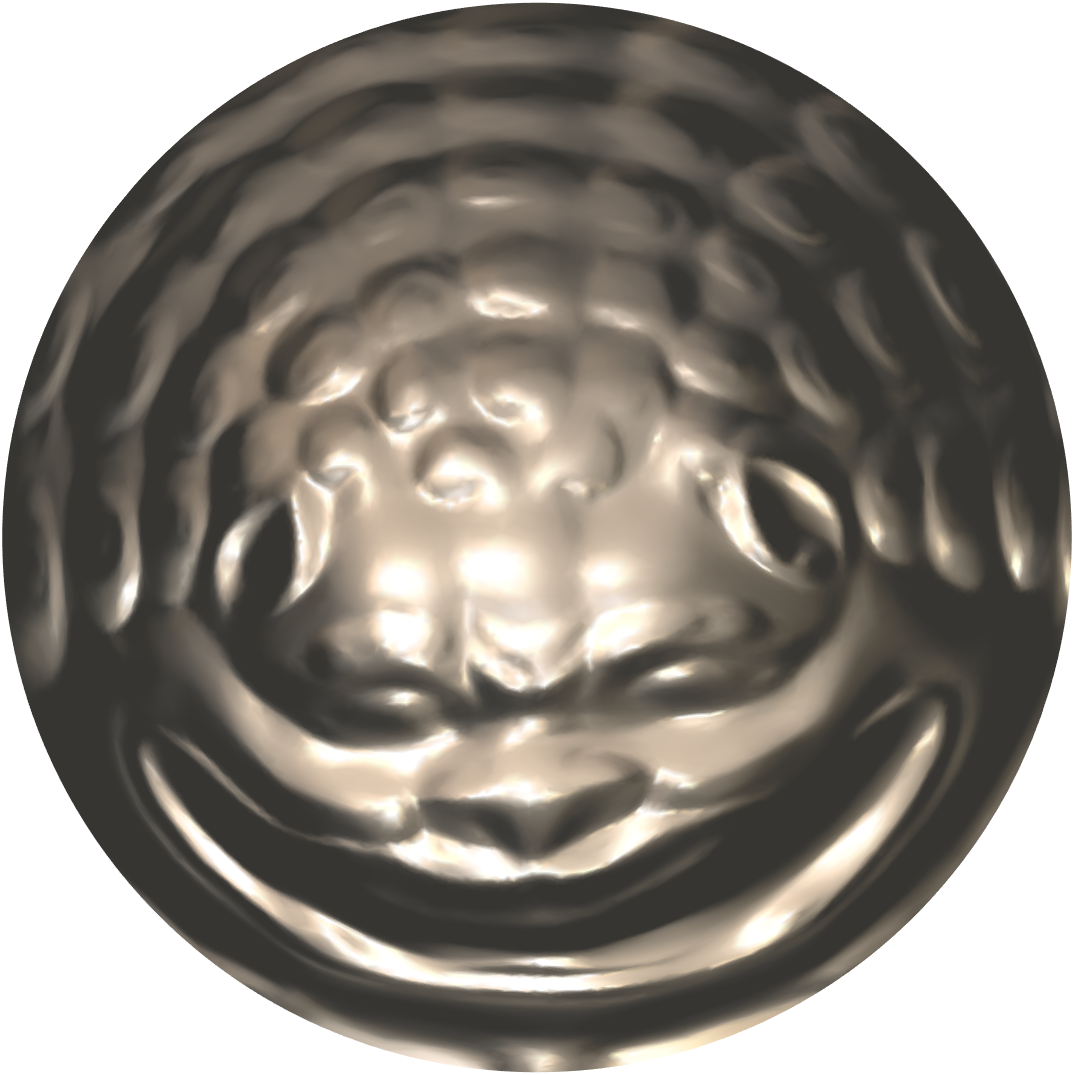} &
\includegraphics[height=3cm]{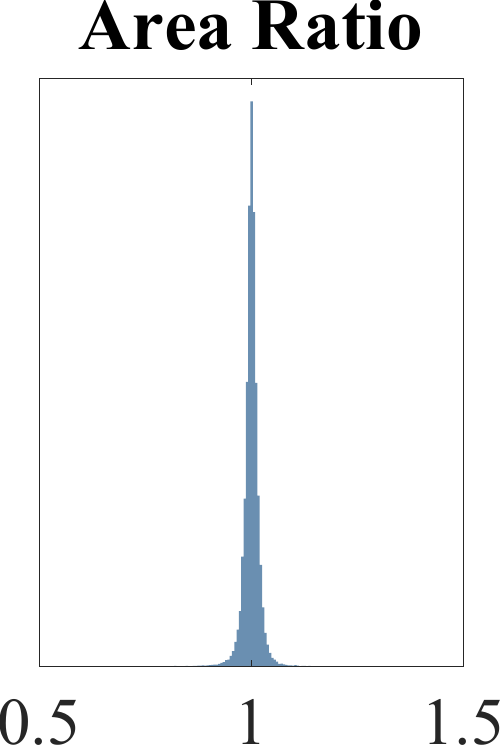} &&
\includegraphics[height=3cm]{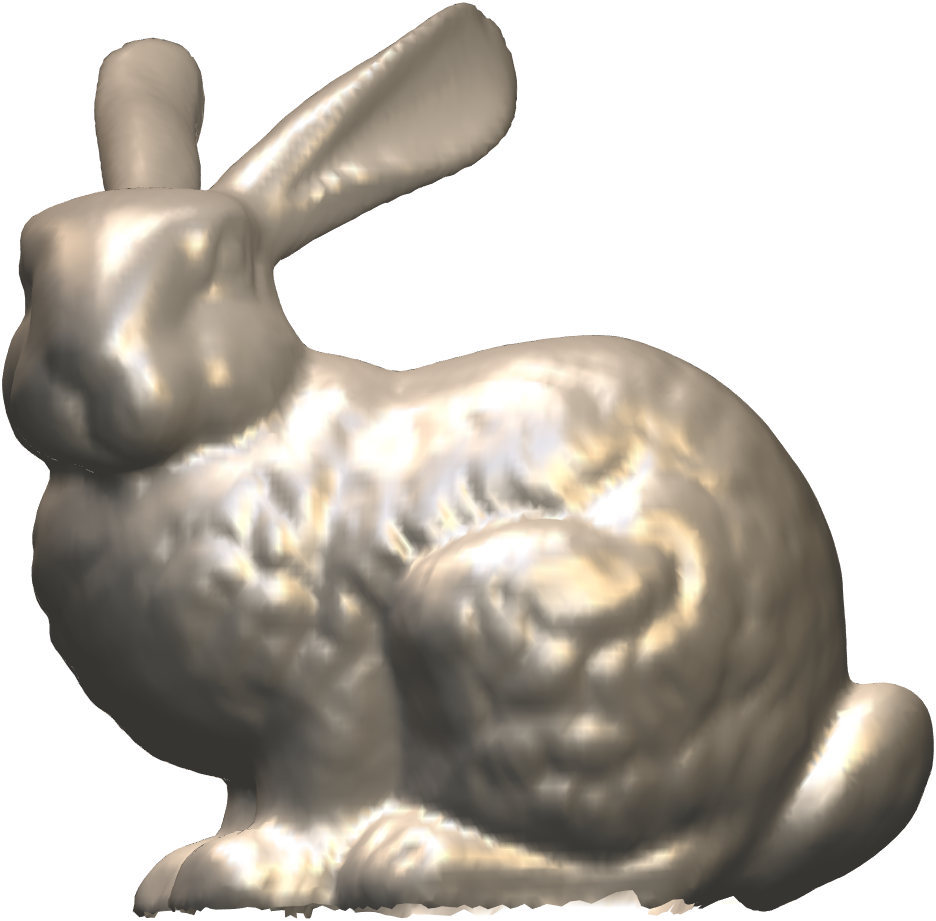} &
\includegraphics[height=2.7cm]{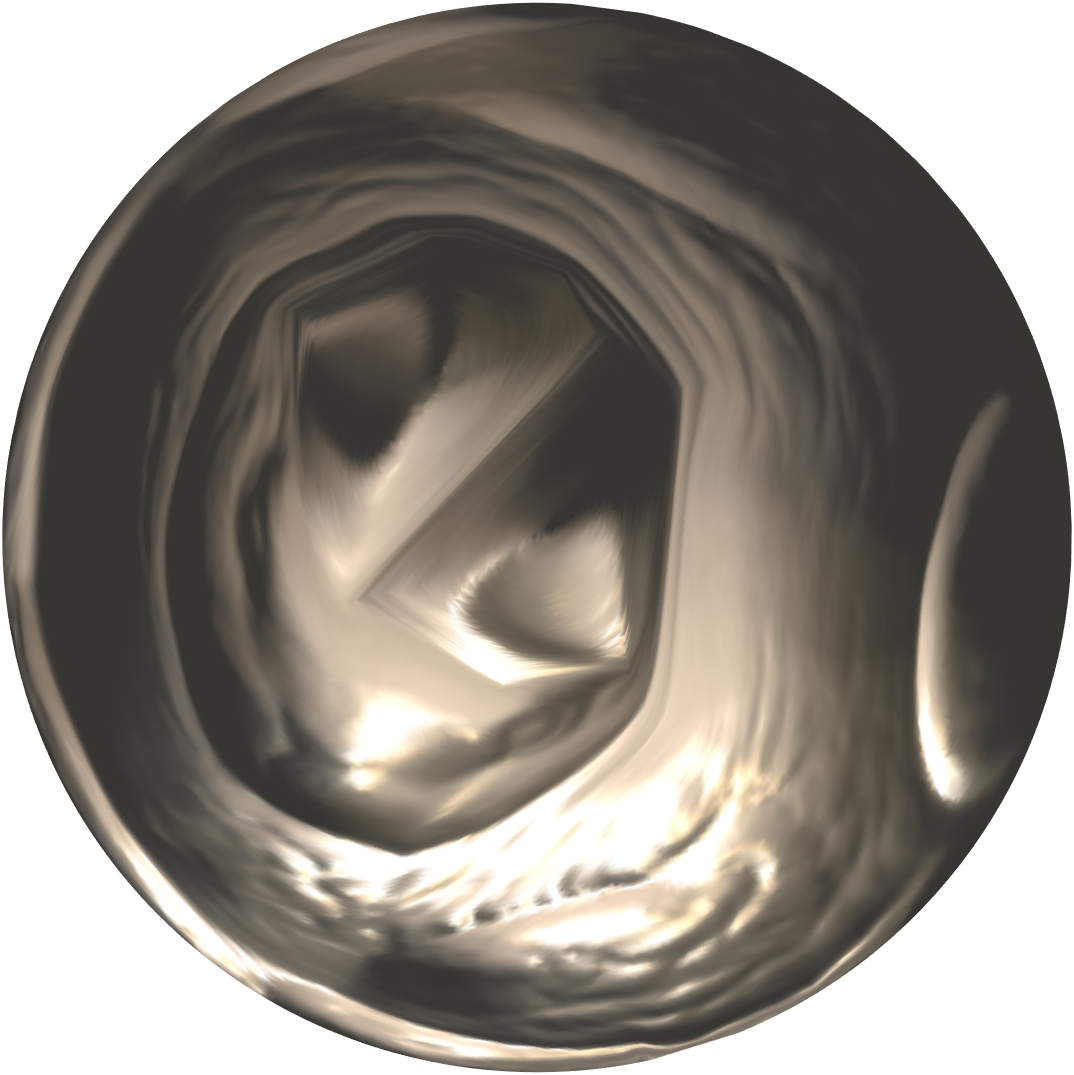} &
\includegraphics[height=3cm]{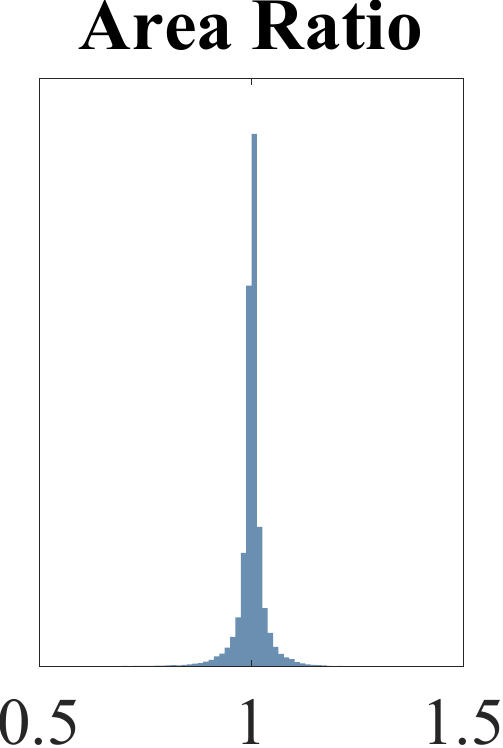} \\
\Cline{1pt}{1-3}\Cline{1pt}{5-7}
\multicolumn{3}{c}{Beetle} && \multicolumn{3}{c}{Face} \\
\multicolumn{3}{c}{$\#\F = 1{,}763$  $\#\V = 988$} && \multicolumn{3}{c}{$\#\F = 27{,}306$  $\#\V = 13{,}969$} \\
\includegraphics[height=3cm]{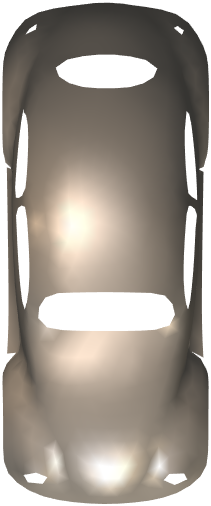} &
\includegraphics[height=2.6cm]{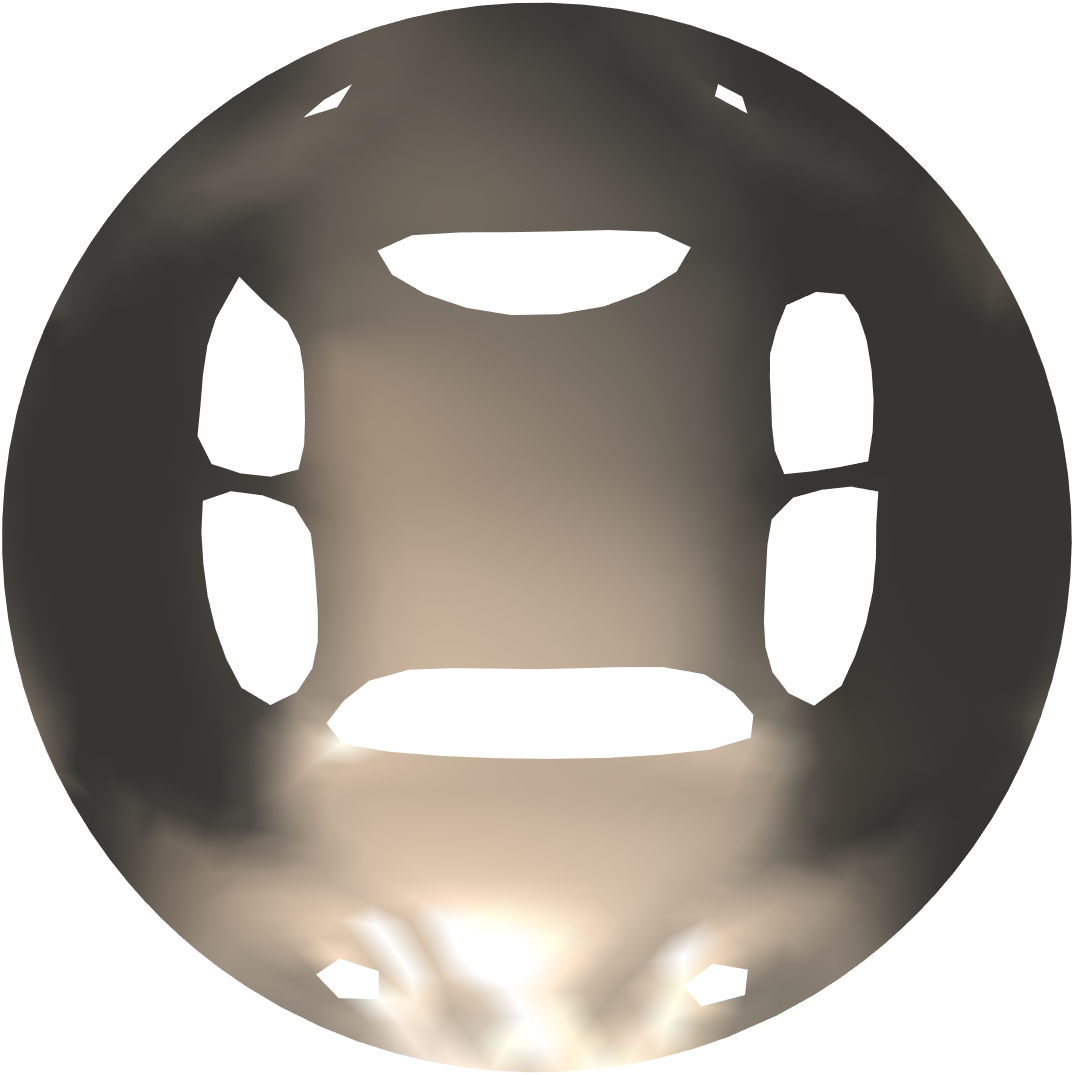} &
\includegraphics[height=3cm]{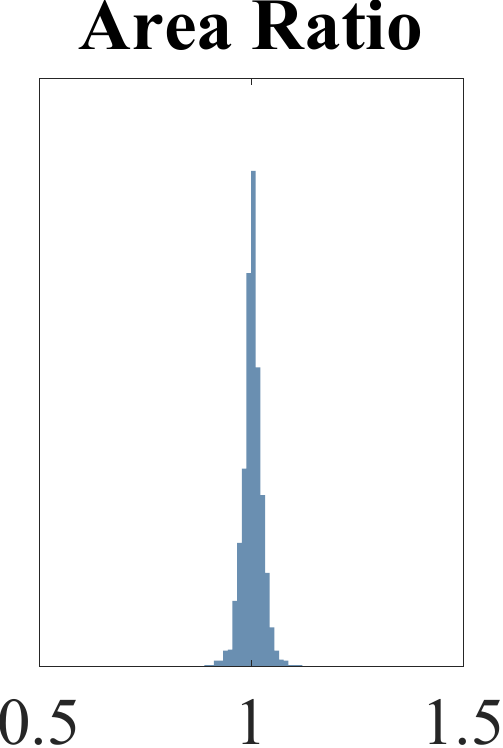} &&
\includegraphics[height=3cm]{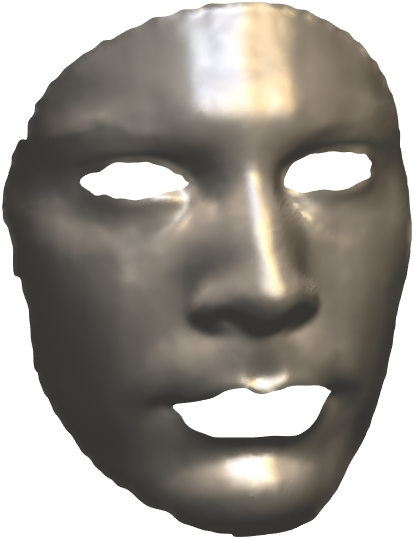} &
\includegraphics[height=2.8cm]{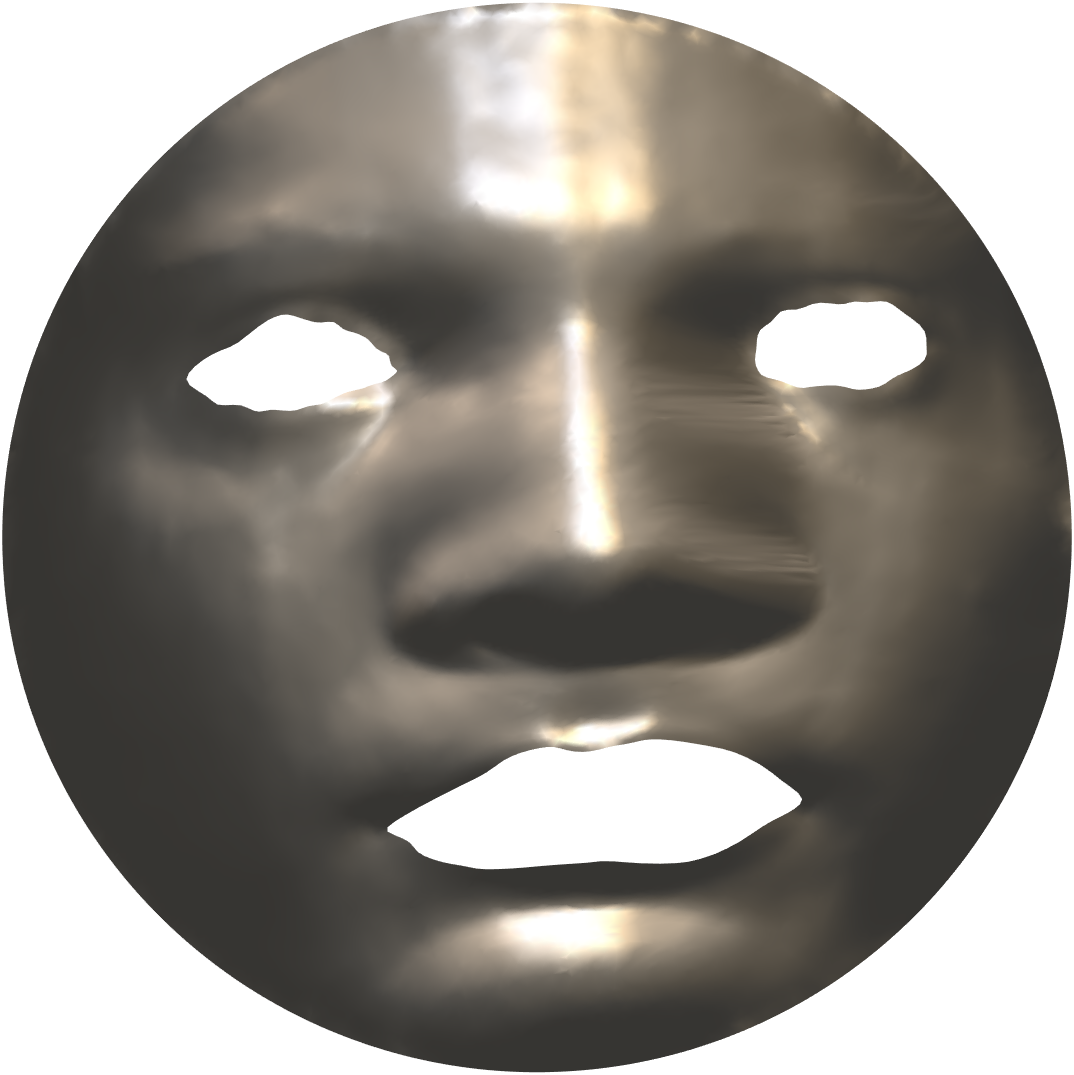} &
\includegraphics[height=3cm]{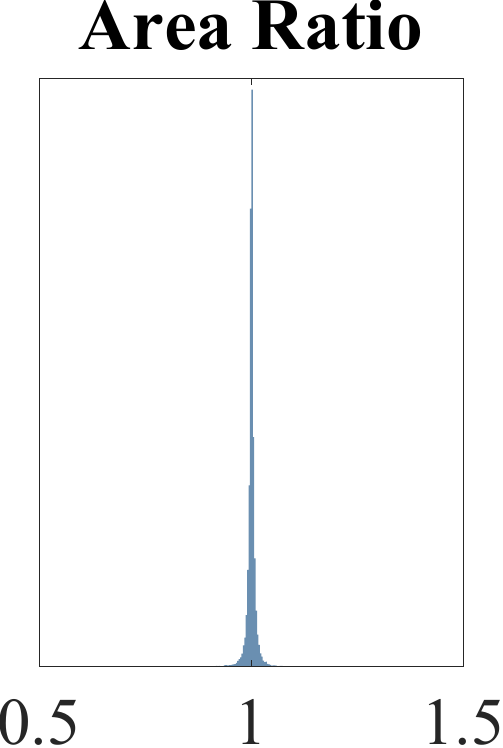} \\
\specialrule{.2em}{.1em}{.1em}
\end{tabular}
}
\caption{Open-surface models (left), corresponding area-preserving parameterizations (middle), and histograms of triangle-wise area ratios (right).}
\label{fig:MeshModel_open}
\end{figure}

\begin{figure}[t]
\centering
\resizebox{\textwidth}{!}{
\begin{tabular}{ccccccc}
\specialrule{.2em}{.1em}{.1em}
\multicolumn{3}{c}{Bulldog} && \multicolumn{3}{c}{Gargoyle} \\
\multicolumn{3}{c}{$\#\F = 99{,}590$  $\#\V = 49{,}797$} && \multicolumn{3}{c}{$\#\F = 100{,}000$  $\#\V = 50{,}002$} \\
\includegraphics[height=3cm]{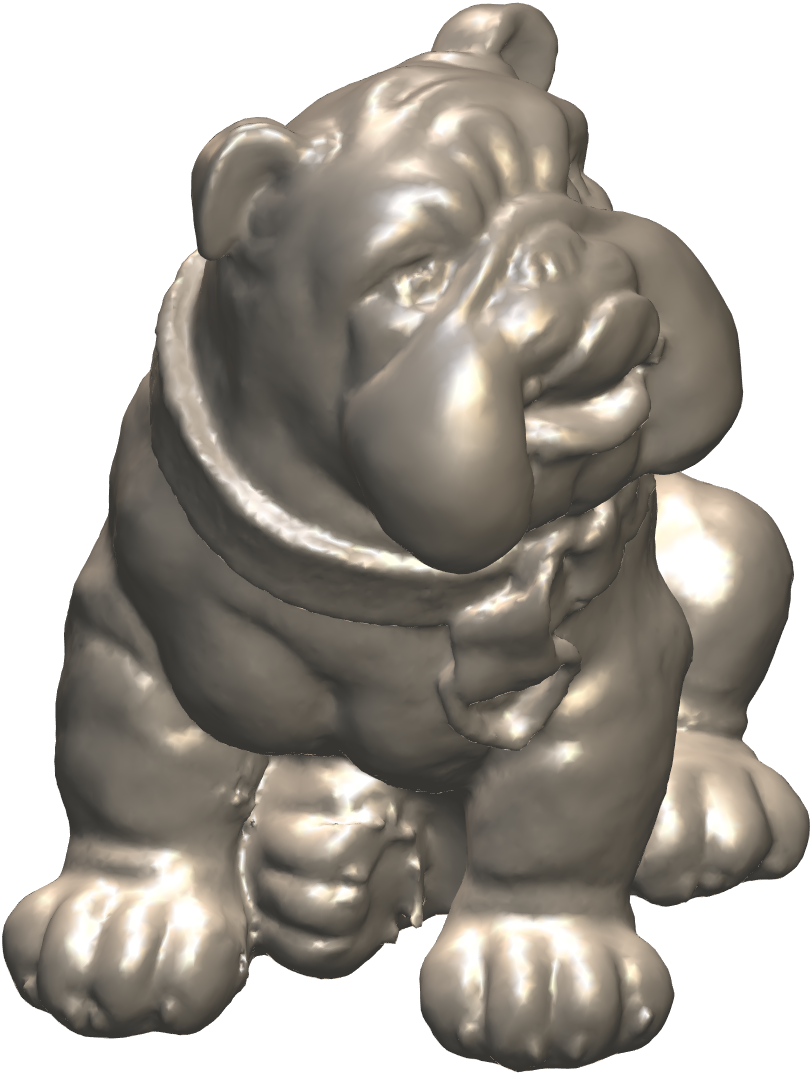} &
\includegraphics[height=2.7cm]{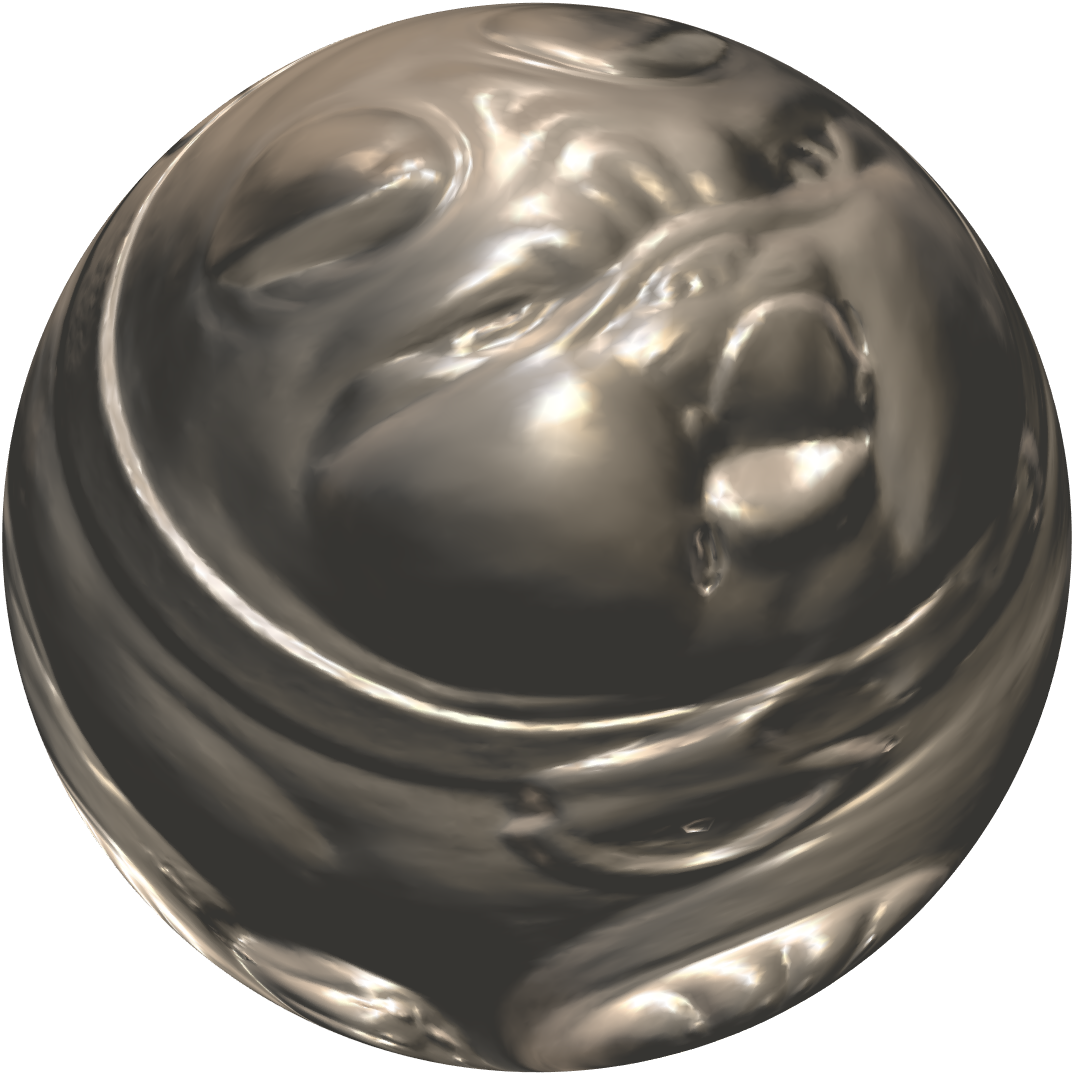} &
\includegraphics[height=3cm]{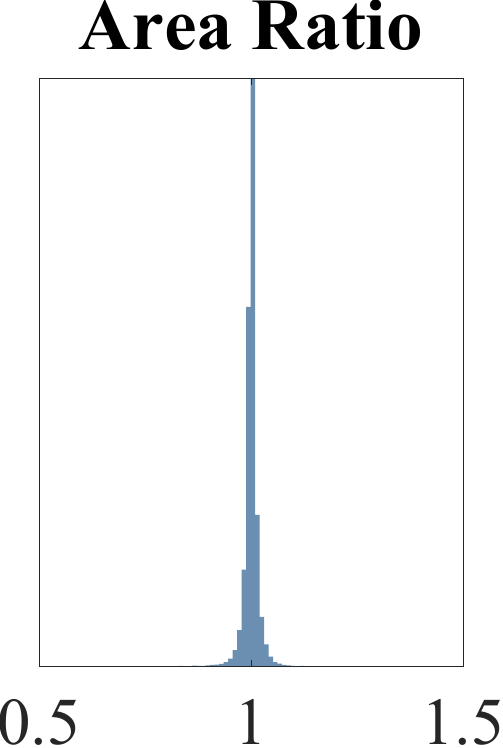} &&
\includegraphics[height=3cm]{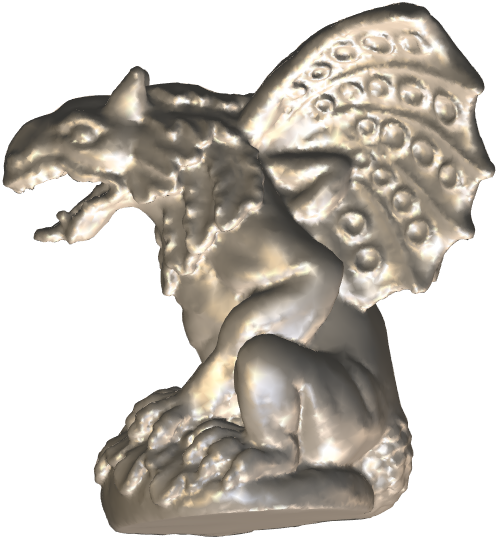} &
\includegraphics[height=2.7cm]{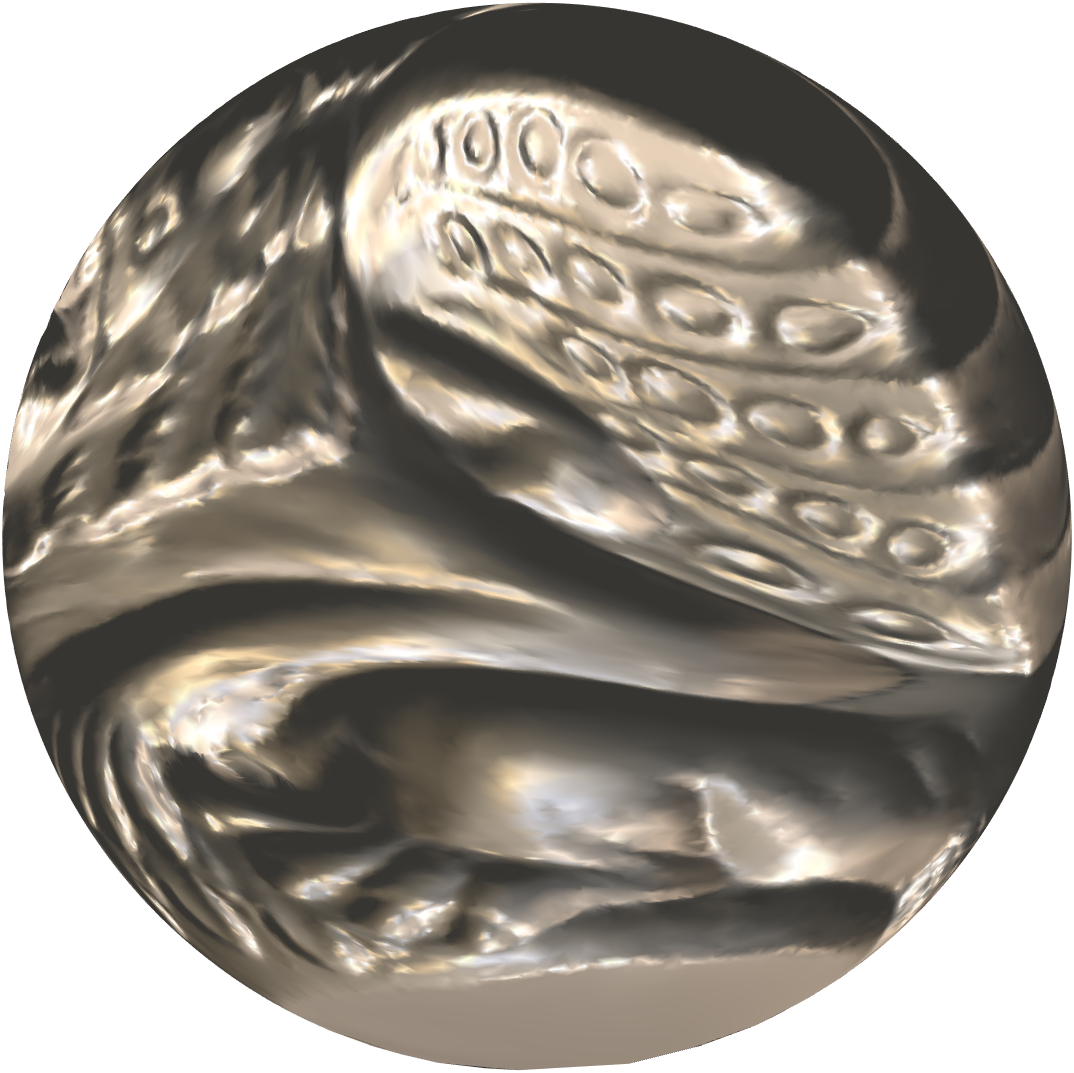} &
\includegraphics[height=3cm]{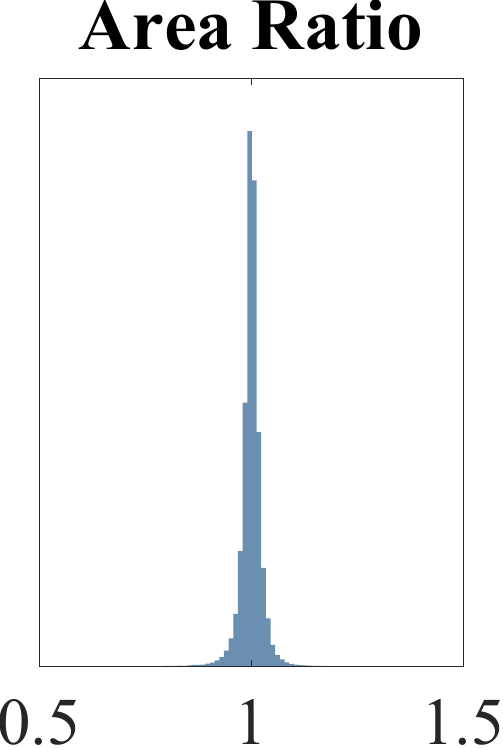} \\
\Cline{1pt}{1-3}\Cline{1pt}{5-7}
\multicolumn{3}{c}{Kitten} && \multicolumn{3}{c}{Rocker Arm} \\
\multicolumn{3}{c}{$\#\F = 20{,}000$  $\#\V = 10{,}000$} && \multicolumn{3}{c}{$\#\F = 20{,}088$  $\#\V = 10{,}044$} \\
\includegraphics[height=3cm]{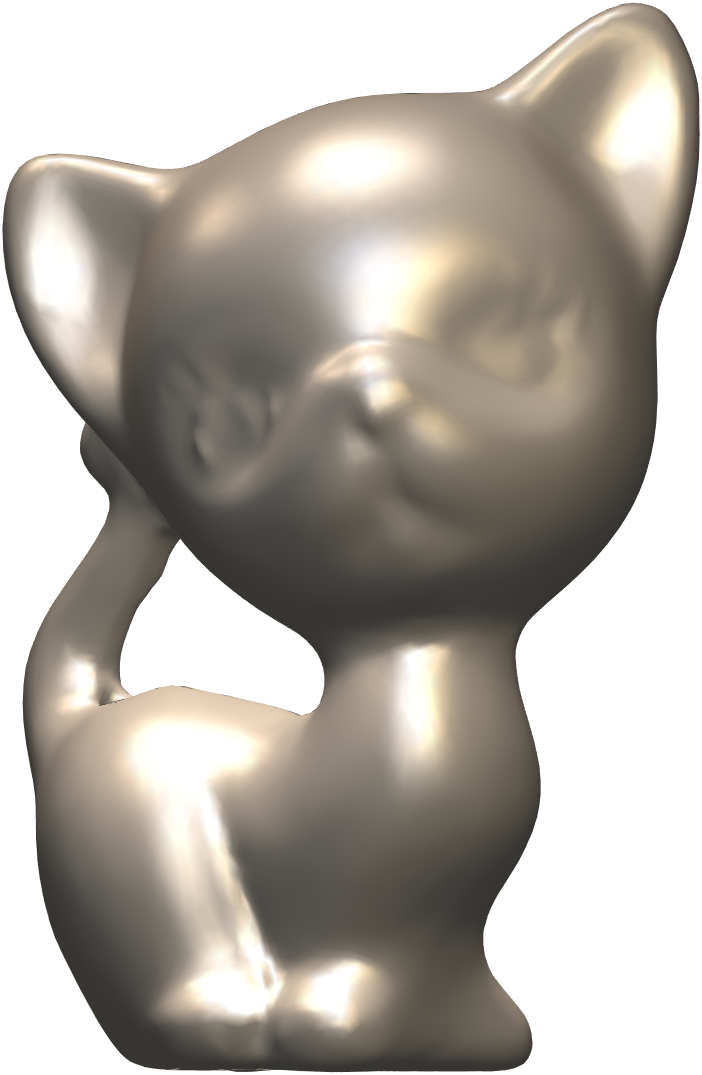} &
\includegraphics[height=2.3cm]{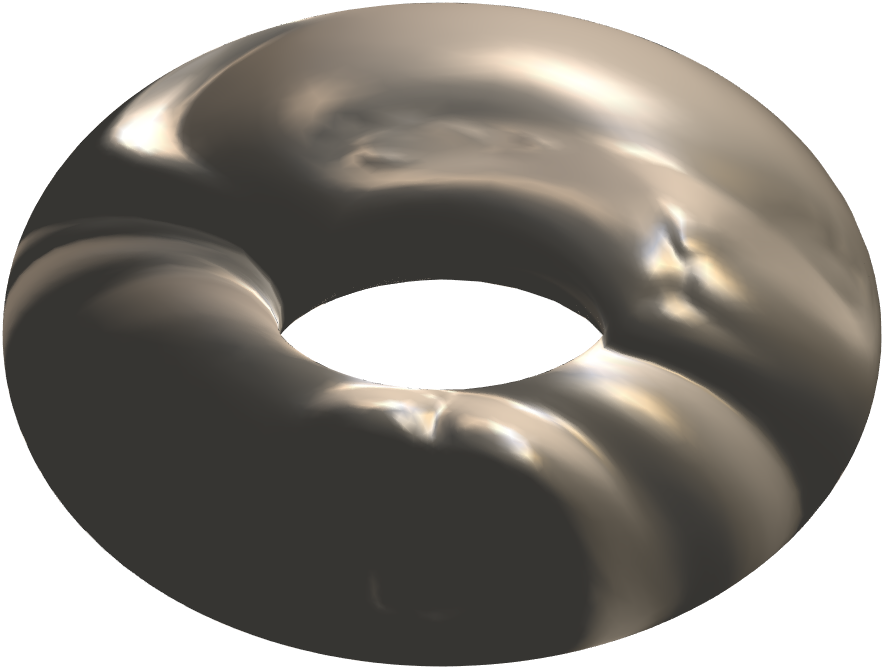} &
\includegraphics[height=3cm]{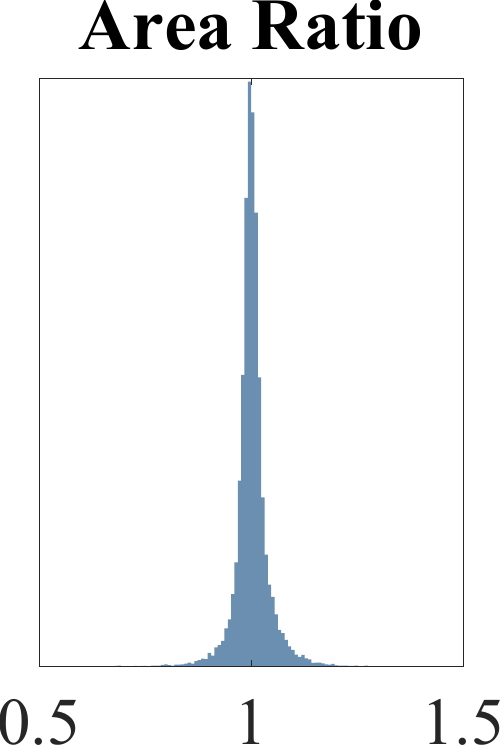} &&
\includegraphics[height=3cm]{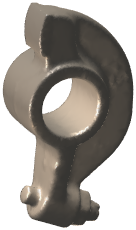} &
\includegraphics[height=2.3cm]{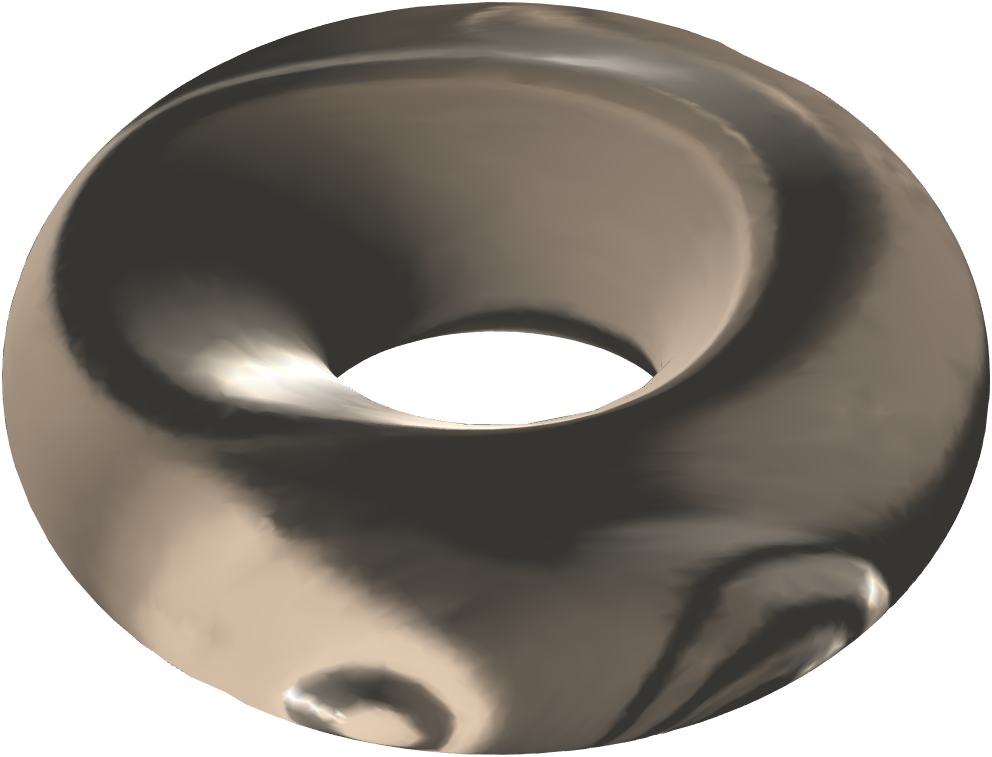} &
\includegraphics[height=3cm]{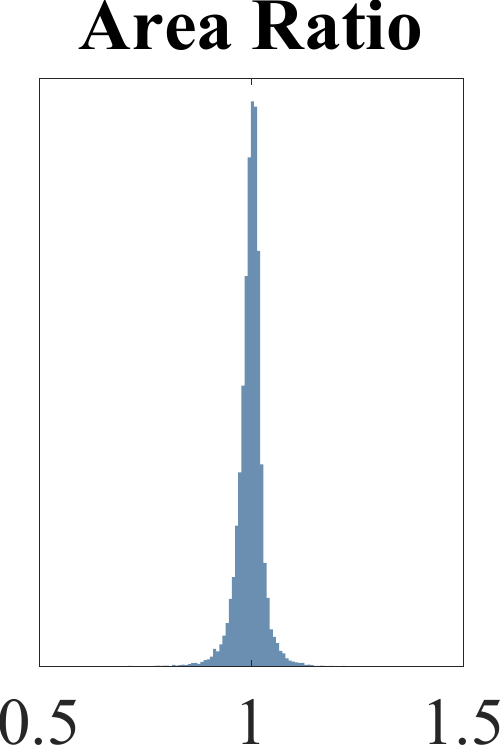} \\
\specialrule{.2em}{.1em}{.1em}
\end{tabular}
}
\caption{Closed-surface models (left), corresponding area-preserving parameterizations (middle), and histograms of triangle-wise area ratios (right).}
\label{fig:MeshModel_close}
\end{figure}

\begin{table}[tbp]
\centering
\caption{
Numerical results for discrete authalic flow with stopping criterion: normalized energy decrease $< 10^{-5}$.
$e_{\mathrm{en}}$: energy residual \eqref{eq:res_en};
$e_{\mathrm{var}}$: unweighted variance of area ratios \eqref{eq:res_var}.
}
\label{tab:flow}
\begin{tabular}{lrcccc}
\specialrule{.2em}{.1em}{.1em}
Model name & Time (s) & $e_\mathrm{en}$ & $e_\mathrm{var}$ & $\#\text{Fold.}$ & $\#\text{Iter.}$ \\
\hline
Lion       & 2.99  &$3.64 \times 10^{-4}$ & $3.81 \times 10^{-4}$  & 0 & 14\\
Bunny      & 16.44 &$1.38 \times 10^{-3}$ & $1.42 \times 10^{-3}$  & 0 & 65\\
Beetle     & 0.13  &$2.55 \times 10^{-3}$ & $2.64 \times 10^{-3}$  & 0 & 15\\
Face       & 2.31  &$9.17 \times 10^{-5}$ & $1.06 \times 10^{-4}$  & 0 & 7\\
Bulldog    & 24.32 &$2.39 \times 10^{-4}$ & $3.83 \times 10^{-4}$  & 0 & 23\\
Gargoyle   & 26.52 &$4.64 \times 10^{-4}$ & $6.67 \times 10^{-4}$  & 0 & 33\\
Kitten     & 6.75  &$1.23 \times 10^{-3}$ & $1.56 \times 10^{-3}$  & 0 & 78\\
Rocker Arm & 6.55  &$7.84 \times 10^{-4}$ & $9.59 \times 10^{-4}$  & 0 & 74\\
\specialrule{.2em}{.1em}{.1em}
\end{tabular}
\end{table}

\subsection{Comparison with state-of-the-art methods}
\label{sec:5.6}
We compare discrete authalic flow (DAF) with stretch energy minimization (SEM) \cite{YuLW19, YuLL19, YuLL20}, density-equalizing map (DEM) \cite{ChRy18, LyLC24, YaCh26}, and optimal transport (OT) \cite{ZhSG13, CuQW19}. The implementations of DEM and OT were obtained from Choi's website\footnote{\url{https://www.math.cuhk.edu.hk/~ptchoi/software.html}} and Gu's website\footnote{\url{https://www.cs.stonybrook.edu/~gu/software/index.html}}, respectively. As in Section~\ref{sec:5.5}, each method is run five times on the toroidal surfaces, and the run with the lowest $e_\mathrm{en}$ is reported.
No OT implementation was available for toroidal or multiply connected open surface parameterizations.

Table~\ref{tab:others} summarizes the energy residual $e_{\mathrm{en}}$ and the number of folded triangles produced by each method. DAF attains the smallest energy residual on all eight benchmarks, and SEM is the closest competitor on all models except for Kitten. DAF produces no folded triangles in any of these tests, whereas SEM produces folds on Kitten and DEM does so on several models.

Figure~\ref{fig:others} visualizes these comparisons: the left panel reports the energy residuals, while the right panel plots the residual ratios of SEM, DEM, and OT relative to DAF. We observe that DAF and SEM generally yield substantially smaller residuals than DEM and OT. Moreover, averaged over the eight benchmarks, the SEM-to-DAF residual ratio is approximately $23.6$, while the residuals of DEM and OT are substantially larger on several models.

Overall, DAF remains fold-free and provides the lowest energy residual on all of the tested benchmarks.

\begin{table}[tbp]
\centering
\caption{
Numerical comparison of discrete authalic flow (DAF), stretch energy minimization (SEM) \cite{YuLW19, YuLL19, YuLL20}, density-equalizing map (DEM) \cite{ChRy18, LyLC24, YaCh26}, and optimal transport (OT) \cite{ZhSG13, CuQW19}. $e_\mathrm{en}$: energy residual \eqref{eq:res_en}; $\#\text{Fold.}$: number of folded triangles.
}
\label{tab:others}
\resizebox{\textwidth}{!}{
\begin{tabular}{lcrcrcrcr}
\specialrule{.2em}{.1em}{.1em}
\multirow{2}{*}{Model name}  &  \multicolumn{2}{c}{DAF} &
\multicolumn{2}{c}{SEM \cite{YuLW19, YuLL19, YuLL20}$^*$} &
\multicolumn{2}{c}{DEM \cite{ChRy18, LyLC24, YaCh26}$^\dagger$} &
\multicolumn{2}{c}{OT \cite{ZhSG13, CuQW19}$^\ddagger$}  \\
& $e_\mathrm{en}$ & $\#\text{Fold.}$ & $e_\mathrm{en}$ & $\#\text{Fold.}$ & $e_\mathrm{en}$ & $\#\text{Fold.}$ & $e_\mathrm{en}$ & $\#\text{Fold.}$ \\
\hline
Lion       &$3.64 \times 10^{-4}$ & 0 & $4.39 \times 10^{-4}$ & 0 & $1.71 \times 10^{-1}$  & 363 & $3.38 \times 10^{-2}$ &0\\
Bunny      &$1.38 \times 10^{-3}$ & 0 & $1.54 \times 10^{-3}$ & 0 & $1.99 \times 10^{\phantom{-}0}$ & 2831& -- & --\\
Beetle     &$2.55 \times 10^{-3}$ & 0 & $2.69 \times 10^{-3}$ & 0 & $3.83 \times 10^{-2}$  & 0   & -- & --\\
Face       &$9.17 \times 10^{-5}$ & 0 & $1.14 \times 10^{-4}$ & 0 & $2.26 \times 10^{-1}$  & 255 & -- & --\\
Bulldog    &$2.39 \times 10^{-4}$ & 0 & $3.21 \times 10^{-3}$ & 0 & $1.47 \times 10^{-1}$  & 0   & $8.02 \times 10^{-1}$ &0\\
Gargoyle   &$4.64 \times 10^{-4}$ & 0 & $2.72 \times 10^{-3}$ & 0 & $1.19 \times 10^{\phantom{-}0}$ & 1   & $1.89 \times 10^{\phantom{-}0}$ &0\\
Kitten     &$1.23 \times 10^{-3}$ & 0 & $1.84 \times 10^{-1}$ & 28 & $1.83 \times 10^{-1}$  & 21  & -- & --\\
Rocker Arm &$7.84 \times 10^{-4}$ & 0 & $1.18 \times 10^{-2}$ & 0 & $2.35 \times 10^{-1}$  & 60  & -- & --\\
\specialrule{.2em}{.1em}{.1em}
\multicolumn{9}{l}{\small --: No source code available. The implementation fails on Bunny due to a non-convexity issue.} \\
\multicolumn{9}{l}{\small $^*$: Stops when the normalized energy decrease is $< 10^{-5}$.} \\
\multicolumn{9}{l}{\small $^\dagger$: Stops at 50 iterations (disk and sphere) or 200 (toroidal).} \\
\multicolumn{9}{l}{\small $^\ddagger$: Stops at 5 iterations with step length $= 0.5$ (0.01 for Bulldog and 0.0005 for Gargoyle).} \\
\end{tabular}
}
\end{table}

\begin{figure}[ht]
\centering
\resizebox{\textwidth}{!}{
\begin{tabular}{cc}
\includegraphics[height=5.7cm]{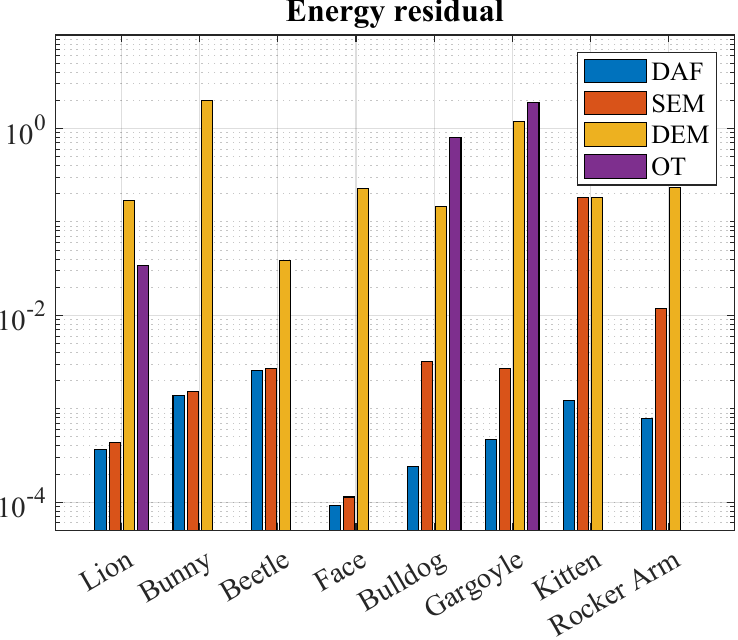} &
\raisebox{0.5cm}{\includegraphics[height=5.2cm]{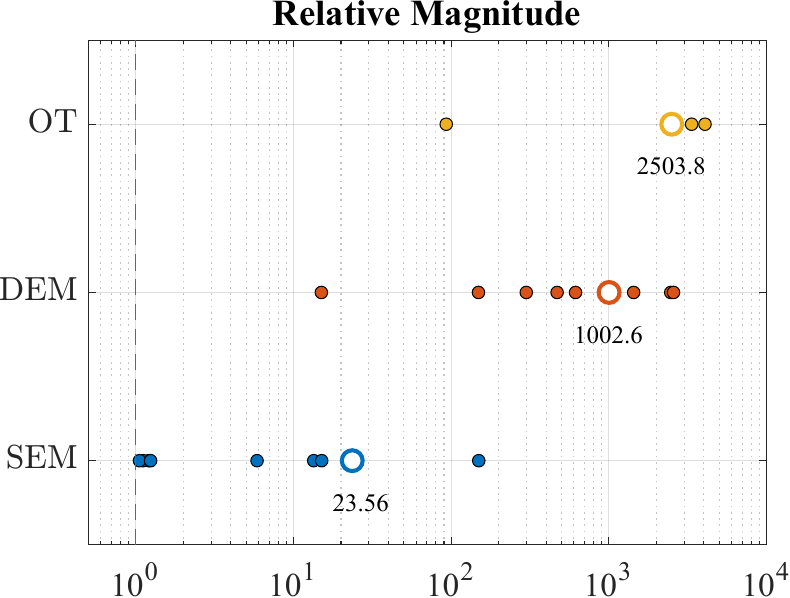}}
\end{tabular}
}
\caption{
Left: Energy residual $e_\mathrm{en}$ for DAF, SEM, DEM, and OT. Right: Ratio of energy residuals relative to DAF, where the large circles denote the average for each method.
}
\label{fig:others}
\end{figure}

\section{Conclusion and discussion}
\label{sec:6}

 We developed a discrete-to-continuum variational theory for area-preserving parameterization based on the stretch energy. We first formulated the stretch energy for orientation-preserving diffeomorphisms and showed that its critical points are precisely the area-preserving maps. We then established first-order consistency between the discrete and smooth stretch energies and showed that, within a geometrically controlled admissible class, the discrete infimum converges to the smooth minimum at rate \(\O(h^2)\), while $h^2$-almost minimizers have \(\O(h)\) area distortion in \(L^2\).

These results distinguish the present theory from existing work. Flow-based methods are derived from smooth constructions, but their discretizations require vertex-based or smoothed representations of the triangle-wise area ratios. Thus, their discretization error is not directly bounded in terms of the mesh size. Methods based on optimal transport guarantee area preservation on dual cells, yet the area error on the primal triangles remains uncontrolled. Methods formulated directly on the triangles capture the discrete area distortion but lack a link to smooth area-preserving maps. By comparison, our framework connects the discrete and smooth variational theories for area-preserving geometry.

On the computational side, we derived the intrinsic \(L^2\)-gradient flow of the smooth stretch energy and obtained its simplicial counterpart directly from the gradient of the discrete stretch energy. This brings the variational structure of the smooth formulation to simplicial surfaces. We then combined it with a projected quasi-implicit time discretization and proved its global convergence with a sufficiently small fixed time step. In numerical experiments, the method attained the lowest normalized stretch-energy residual on all eight models among stretch-energy minimization, density-equalizing maps, and optimal transport maps.

Several questions remain open. Since the stretch energy controls only the Jacobian determinant and not the full differential, it does not, by itself, prevent image triangles from degenerating. The present analysis, therefore, assumes uniform geometric control of the admissible image meshes. Moreover, global existence and long-time convergence of the continuous authalic flow have not yet been established. We leave these questions for future work.

\section*{Acknowledgments}

The work of Liu and Yueh was partially supported by the National Science and Technology Council, Taiwan, under Grant NSTC 115-2628-M-003-002-MY4, and by the National Center for Theoretical Sciences, Taiwan. The work of Sakai was supported by JSPS KAKENHI Grant Number 25KJ0069.

\small
\bibliographystyle{abbrvurl}
\bibliography{reference}

\end{document}